\documentclass[10pt]{article}
\usepackage[english]{babel}
\usepackage{amsmath,amssymb,mathrsfs,amsthm}
\usepackage[latin1]{inputenc}
\usepackage{eucal}
\usepackage{paralist}
\usepackage{verbatim}
\numberwithin{equation}{section}
\makeatletter
\let\@fnsymbol\@arabic
\makeatother
\theoremstyle{plain}
\begingroup
\newtheorem{thm}{Theorem}[section]
\newtheorem{lemma}[thm]{Lemma}
\newtheorem{prop}[thm]{Proposition}
\newtheorem{problem}[thm]{Problem}

\endgroup
\theoremstyle{definition}
\begingroup
\newtheorem{defin}[thm]{Definition}
\newtheorem{rmk}[thm]{Remark}

\endgroup
\theoremstyle{remark}
\begingroup
\endgroup
\newcommand{\BV}{\mathrm{BV}}
\newcommand{\foraa}{\text{for a.a.\ }}
\newcommand{\pairing}[4]{ \sideset{_{#1 }}{_{ #2}}  {\mathop{\left\langle #3 , #4  \right\rangle}}}
\newcommand{\testw}{\eta}
\newcommand{\bigF}{f}
\newcommand{\aein}{\text{a.e.\ in }}
\newcommand{\eps}{\varepsilon}

\newcommand{\om}{\omega}
\newcommand{\kk}{_n^k}
\newcommand{\kkm}{_{n,M}^k}
\newcommand{\km}{_n^{k-1}}

\newcommand{\kmm}{_n^{k-2}}
\newcommand{\wto}{\rightharpoonup}
\newcommand{\wtos}{\mathrel{\mathop{\rightharpoonup}\limits^*}}
\newcommand{\E}{\mathcal{E}}

\newcommand{\calE}{\mathcal{E}}

\newcommand{\calG}{\mathcal{G}}
\newcommand{\calH}{\mathcal{H}}
\newcommand{\calL}{\mathcal{L}}
\newcommand{\calR}{\mathcal{R}}

\newcommand{\calC}{\mathcal{C}}
\newcommand{\calX}{\mathcal{X}}
\newcommand{\bdm}{\begin{displaymath}}
\newcommand{\edm}{\end{displaymath}}
\newcommand{\be}{\begin{equation}}
\newcommand{\ee}{\end{equation}}
\newcommand{\bes}{\begin{equation*}}
\newcommand{\ees}{\end{equation*}}
\newcommand{\bea}{\begin{eqnarray}}
\newcommand{\eea}{\end{eqnarray}}
\newcommand{\beas}{\begin{eqnarray*}}
\newcommand{\eeas}{\end{eqnarray*}}

\mathchardef\emptyset="001F
\newcommand{\N}{\mathbb{N}}
\newcommand{\R}{\mathbb{R}}
\newcommand{\Rsym}{ {\R^{d \times d}_{\textnormal{sym}} } }
\newcommand{\Tsym}{ {\R^{d \times d \times d \times d}_{\textnormal{sym}} } }

\newcommand{\KK}{\mathbb{K}}
\newcommand{\A}{\mathbb{A}}
\newcommand{\BB}{\mathbb{B}}
\newcommand{\DD}{\mathbb{D}}
\newcommand{\CC}{\mathbb{C}}

\newcommand{\Om}{\Omega}
\newcommand{\w}{\theta}
\newcommand{\calZ}{\mathcal{Z}}
\newcommand{\baru}{\overline u_n}
\newcommand{\barv}{\overline v_n}
\newcommand{\barw}{\overline \theta_n}
\newcommand{\barz}{\overline z_n}
\newcommand{\barth}{\overline \theta_n}
\newcommand{\bartheta}{\overline \theta_n}
\newcommand{\bartau}{\overline \tau_n}
\newcommand{\bareta}{\overline \eta_n}
\newcommand{\underu}{\underline u_n}
\newcommand{\underz}{\underline z_n}
\newcommand{\underth}{\underline\theta_n}

\newcommand{\undertau}{\underline \tau_n}

\newcommand{\barh}{\overline{\hs}_n}
\newcommand{\barbigF}{\overline \bigF_n}
\newcommand{\barH}{\overline{\hv}_n}
\newcommand{\acca}{\mathcal{H}}

\newcommand{\integ}[3]{\int_{#1} #2 \, \d #3}
\newcommand{\integlin}[4]{\int_{#1}^{#2} #3 \, \d #4}

\newcommand{\argmin}{\mathop{\rm argmin}}

\renewcommand{\d}{\mathrm{d}}
\newcommand{\dd}{\,\mathrm{d}}

\renewcommand{\div}[1]{\mathop\mathrm{div}\,#1}

\renewcommand{\mod}[1]{\left|#1\right|}
\newcommand{\norm}[2]{\left\|#1\right\|_{#2}}
\newcommand{\ps}{\cdot}
\newcommand{\psm}{:}

\newcommand{\Xteta}{X}
\newcommand{\dom}{\mathrm{dom}}
\newcommand{\weaksto}{\wtos}
\newcommand{\rmC}{\mathrm{C}}
\newcommand{\weakto}{\rightharpoonup}
\newcommand{\fv}{f_\mathrm{V}}
\newcommand{\fs}{f_\mathrm{S}}
\newcommand{\hv}{H}
\newcommand{\hs}{h}
\newcommand{\taun}{\tau_n}
\newcommand{\fte}{\eps t}
\newcommand{\ures}{u_\eps}
\newcommand{\zres}{z_\eps}
\newcommand{\thres}{\theta_\eps}
\newcommand{\dures}{\dot u_\eps}

\newcommand{\dzres}{\dot z_\eps}
\newcommand{\dthres}{\dot\theta_\eps}
\newcommand{\uresn}{u_{\eps_n}}
\newcommand{\zresn}{z_{\eps_n}}
\newcommand{\thresn}{\theta_{\eps_n}}
\newcommand{\duresn}{\dot u_{\eps_n}}

\newcommand{\fve}{f_{\mathrm{V},\eps}}

\newcommand{\hve}{H_\eps}
\newcommand{\hse}{h_\eps}
\newcommand{\bigFe}{\bigF_\eps}

\newcommand{\hvue}{H^\eps}
\newcommand{\hsue}{h^\eps}

\begin{document}
\hyphenation{quasi-stat-ic}
\nocite{*}
\title{Rate-independent damage in thermo-viscoelastic materials \\ with inertia }
\author{
Giuliano Lazzaroni
\thanks{DMA, Universit\`{a} degli Studi di Napoli Federico II, via Cintia, Monte S.\ Angelo, 80126 Napoli, Italy.
Email: {\ttfamily giuliano.lazzaroni@sissa.it}}
\and
Riccarda Rossi
\thanks{DIMI, Universit\`{a} degli Studi di Brescia, Via Valotti 9, 25133 Brescia, 
Italy.
Email: {\ttfamily riccarda.rossi@unibs.it} } 
\and
Marita Thomas
\thanks{
Weierstrass Institute for Applied Analysis and
Stochastics, Mohrenstr.~39, 10117 Berlin, Germany.
Email: {\ttfamily marita.thomas@wias-berlin.de}
}
\and
Rodica Toader
\thanks{DMIF, Universit\`{a} degli Studi di Udine, Via delle Scienze 206, 33100 Udine, Italy.
Email: {\ttfamily rodica.toader@uniud.it}
}
}
\date{}
\maketitle
\begin{abstract}
We present a model for rate-independent, unidirectional, partial damage in visco-elastic
materials with inertia and thermal effects.
The damage process is modeled by means of an internal variable, governed by a rate-independent flow rule.
The heat equation and the momentum balance for the displacements are coupled in a highly nonlinear way.
Our assumptions on the corresponding energy functional also comprise the case of the
Ambrosio-Tortorelli phase-field model (without passage to the brittle limit).
We discuss a suitable weak formulation and prove an existence theorem obtained with the aid of a  (partially)
decoupled time-discrete scheme
and variational convergence methods.
We also carry out the asymptotic analysis for vanishing viscosity and inertia and obtain
a fully rate-independent limit model
for displacements and damage, which is independent of temperature.
\end{abstract}
%
%
\section{Introduction}
Gradient damage models have been extensively studied in recent years,
in particular in order to understand the behavior of brittle or quasi-brittle materials.
In this paper we present a model for rate-independent, unidirectional, partial damage
in visco-elastic materials with inertia and thermal effects.
Thus we deal with a PDE system composed of the (damped) equation of elastodynamics,
a rate-independent flow rule for the damage variable,
and the heat equation, coupled in a highly nonlinear way.
We prove an existence result basing on time-discretization
and variational convergence methods, where
the analytical difficulties arise from the interaction of
rate-independent and rate-dependent phenomena.
We study also the relationship of our model with a
fully rate-independent system by time rescaling.
\par
Following Fr\'emond's approach \cite{Fre02NST}, damage is
represented through an internal variable, in the context of
generalized standard materials \cite{HalNgu75MSG}. The damage
process is unidirectional, meaning that no healing is allowed; we do
not use the term ``irreversibility'' to avoid confusion with
thermodynamical notions. In our model the evolution of this variable
is rate-independent: this choice is due to the consideration that,
to damage a certain portion of the material, one needs a quantity of
energy that is independent of the  rate of damage,  see e.g.\
\cite{KoMiRo006RIAD}. Rate-independent damage has been widely
explored over the last years, cf.\ e.g.\ 
\cite{MiRou06,FraGar,BouMiRou07,ThoMie09DNEM,GarLar,
Thom11QEBV,FiKnSt10YMQS,KnRoZa11}. 
 For different studies on rate-dependent damage we refer to e.g.\ \cite{FreNed96DGDP,BS,bss}
in the isothermal case and \cite{BB,RR1,RoRoESTC14,HK} for temperature-dependent systems.
\par
Energy can be dissipated not only  by  damage growth, but also by
viscosity and heat, both phenomena having a rate-dependent nature.
Rate-independent processes coupled with viscosity, inertia, and also
temperature have first been analyzed in the two pioneering papers
\cite{Roub09RIPV,Roub10TRIP},  cf.\ also  \cite[Chap.\ 5]{MR3380972}.   
Under the assumption of small strains,
the momentum equation is linearized and is formulated using
Kelvin-Voigt rheology and inertia. The nonlinear heat equation is
coupled with the momentum balance through a thermal expansion term:
this reflects the fact that temperature changes produce additional
stresses. Here, we extend Roub\'i\v cek's ansatz for the
temperature-dependent setting to a unidirectional process, thus
dealing with a discontinuous rate-independent dissipation potential,
cf.\ \eqref{dissip-potential} below. Existence results for an
Ambrosio-Tortorelli-type  system with unidirectional damage,
inertia, and damping were already provided in \cite{LarOrtSue} in
the isothermal case.
\par
\paragraph{The PDE system.}
More precisely, we address the analysis of the following PDE system:
\begin{subequations}
\label{ourPDE}
\begin{alignat}{3}
\label{eq:u}
&&&\rho \ddot{u}-\div\left(\DD(z,\theta)e(\dot u) + \CC(z)e( u) -\theta\,\BB \right)=\fv \quad
& \text{in } (0,T)\times\Omega \,,
\\
\label{eq:z}
&&&
\partial \mathrm{R}_1 (\dot z)
+ \mathrm{D}_z G(z,\nabla z) - \div(\mathrm{D}_\xi G(z,\nabla z) )
+\tfrac12 \CC'(z)e(u)\psm e(u)  \ni 0 \quad & \text{in } (0,T)\times\Omega \,,
\\
\label{eq:theta} &&&
\dot\theta-\div(\KK(z,\theta)\nabla\theta)=\mathrm{R}_1(\dot z)+
\DD(z,\theta ) e(\dot u)\psm e(\dot u) - \theta\,\BB\psm e(\dot u)
+ \hv \quad & \text{in } (0,T)\times\Omega \,,
\end{alignat}
\end{subequations}
where the unknowns are the displacement vector field $u$, the damage
variable $z$, and the absolute temperature $\theta$, all the three
being functions of the time $t\in(0,T)$ and of the position $x$ in
the reference configuration of a material $\Om$, a bounded subset of
$\R^d$, with $d\in\{2,3\}$. Here,
$e(u):=\frac{1}{2}(\nabla u+\nabla u^\top)$ denotes the  linearized   strain
tensor.
\par
In \eqref{eq:u}, the constant $\rho>0$ is the mass density. Moreover,
$\DD(z,\theta)$ and $\CC(z)$ are the viscous and the elastic stress tensors
and are both bounded, symmetric, and positive definite   on symmetric matrices,   uniformly in $z$ and $\theta$.
This reflects two hypotheses of the model, motivated by analytical reasons:
first, we cannot renounce the presence of some damping in the  momentum balance; 
second, we restrict ourselves to the case of partial damage,
assuming that even in its most damaged state the material keeps some elastic properties.
In order to account for the phenomenological effect that an increase of damage reduces
the stored elastic energy, see e.g.\ \cite{LemDes05EDM}, it is assumed that
the elastic tensor $\CC(z)$ depends monotonically on
the internal variable $z$, cf.\ also \cite{FreNed96DGDP,Fre02NST,MieHofWel}.
\par
According to the rate-independent and unidirectional nature of the damage process,
$\mathrm{R}_1$ is a 1-homogeneous dissipation potential of the form
\begin{equation}
\label{dissip-potential}
\mathrm{R}_1 (v):=
\begin{cases}
 |v| & \text{if } v \leq 0 \,, 
\\
 +\infty &  \text{otherwise,}
\end{cases}
\end{equation}
which enforces the internal variable $z$ to be nonincreasing in time.
Indeed, we assume that $z=1$ marks the sound material and $z=0$ the most damaged state.
\par
The gradient term $G(z,\nabla z)$ is needed to regularize damage;
in particular, this term also  allows for a nonconvex dependence on $z$
as in  many  phase-field models. 
Moreover, for suitable choices we retrieve
the Modica-Mortola term appearing in
the Ambrosio-Tortorelli functional, see Remark \ref{rmk:a-t}.
The flow rule \eqref{eq:z}
is given as a subdifferential inclusion, where  $\partial$
denotes the subdifferential in the sense of convex analysis 
of $\mathrm{R}_1$ while $\mathrm{D}_z$ and $\mathrm{D}_\xi$ stand for 
the G\^ateaux derivatives of $G(\cdot,\xi)$ and $G(z,\cdot)$, respectively. 
This is a compact way to write a (semi)-stability condition of Kuhn-Tucker type.
\par
The term $\theta\,\BB$, where $\BB$ is a fixed symmetric matrix,
derives from thermodynamical considerations and
is a coupling term between the momentum \eqref{eq:u} and the heat equation \eqref{eq:theta}.
The information on the heat conductivity of the material is contained in the symmetric matrix $\KK(z,\theta)$.
We suppose that $\KK(z,\cdot)$ satisfies subquadratic growth conditions uniformly in $z$,
which are borrowed from  \cite{RoRoESTC14} and  which are  in the same spirit as in  \cite{FePeRo09ESPT}.
These conditions are fundamental
in the proof of some a priori estimates; see the discussion below \eqref{freeEn}
for appropriate examples from materials science.
\par
All the aforementioned quantities are independent of time and space,
whilst the external force $\fv$ and the heat source $H$ are functions of both.
The system is complemented with the natural boundary conditions
\begin{subequations}
\label{bc}
\begin{alignat}{3}
\label{bc-a}
\left(\DD(z,\theta)e(\dot u)+ \CC(z)e(u)- \theta\,\BB  \right) \nu &= \fs
&&\quad\text{on } (0,T)\times\partial_{\mathrm{N}}\Om \,,\\
\label{bc-b}
u&=0 &&\quad\text{on } (0,T)\times\partial_\mathrm{D}\Om \,,\\
\label{bc-c}
\mathrm{D}_\xi G(z,\nabla z) \, \nu&=0 &&\quad\text{on } (0,T)\times\partial\Om\,,\\
\label{bc-d}
\KK(z,\theta)\nabla\theta\ps\nu&=\hs &&\quad\text{on } (0,T)\times\partial\Om \,,
\end{alignat}
\end{subequations}
where $\partial_\mathrm{D}\Om$ and $\partial_{\mathrm{N}}\Om:=\partial\Om\backslash\partial_{\mathrm{D}}\Om$ are
the Dirichlet and the Neumann part of the boundary,
$\nu$ denotes the outer unit normal vector to $\partial\Om$,
and $\fs$ and $h$ are prescribed external data depending on time and space.
As for the Dirichlet data, we restrict to homogeneous boundary conditions,
see Remark \ref{rmk:time-dependent-loading} for a discussion on this choice.
Moreover, Cauchy conditions are given on $u(0)$, $\dot u(0)$, $z(0)$, and $\theta(0)$.
We refer to Section \ref{s:ass} for the precise assumptions on the domain and the given data.
\par
\paragraph{The energetic formulation.}
Due to the rate-independent character of the flow rule \eqref{eq:z}
and to the nonconvexity of the underlying energy, proving the
existence of solutions to the PDE system \eqref{ourPDE}  in its
pointwise form seems to be out of reach. As customary in
rate-independent processes, we will resort to a weak solvability
concept, based on the notion of \emph{energetic solution}, see
\cite{Miel05ERIS} and references therein. For fully rate-independent
systems, governed (in the classical PDE-formulation) by the
static momentum balance for $u$ and the rate-independent flow rule
for $z$, the energetic formulation consists of two properties:
\begin{compactitem}
\item \emph{global stability:} at each time $t$ the configuration $(u(t),z(t))$ is a global minimizer
of the sum of energy and dissipation;
\item \emph{energy-dissipation balance:}
the sum of the energy at time $t$ and of the dissipated energy in $[0,t]$ 
equals the initial energy plus the work of external loadings.
\end{compactitem}
Over the last decade, this approach has been extensively applied to several mechanical problems
and in particular to fracture, see e.g.\ \cite{FraLar,DFT05,DMLaz}, and damage,
see e.g.\ \cite{MiRou06,ThoMie09DNEM,Thom11QEBV}.
\par
However, in a context where other rate-dependent phenomena are present,
the global stability condition is too restrictive.
Following \cite{Roub09RIPV,Roub10TRIP} we will replace it
with a \emph{semistability} condition, where
the sum of energy and dissipation is minimized with respect to the internal variable $z$ only,
while the displacement $u(t)$ is kept fixed,
see also \cite{RosRou,BarRou,RouNCTVE}.
Accordingly, we will weakly formulate system \eqref{ourPDE} by means of
\begin{compactitem}
\item semistability,
\item the (dynamic) momentum equation in a weak sense,
\item a suitable energy-dissipation balance,
\item the heat equation in a weak sense.
\end{compactitem}
\par
\paragraph{Existence result.}
Theorem \ref{thm:main} states the existence of energetic solutions
to the initial-boundary value problem for system \eqref{ourPDE}.
For the proof we rely on a well-established method for showing
existence for rate-independent processes \cite{Miel05ERIS},
adjusted to the coupling with viscosity, inertia, and temperature in  \cite{Roub10TRIP}.
Although we follow the approach of the latter paper,
let us point out that the results therein do not account for some properties of our model, namely,
\begin{compactitem}
\item the unidirectionality of damage, see \eqref{dissip-potential},
\item the dependence of the viscous tensor $\DD(z,\theta)$ on damage and temperature.
\end{compactitem}
These features are important for the modeling of volume-damage, as well as for
the phase-field approximation of fracture and surface damage models,
see also Remark \ref{rmk:a-t}, and cause some analytical difficulties.
\par
As in many works on rate-independent systems,
our existence proof is based on time-discretization and approximation by means of solutions
to incremental problems. Differently from \cite{Roub10TRIP}, in our discrete scheme
the approximate flow rule is decoupled from the other two equations,
which may produce more efficient numerical simulations.
Moreover, the assumption of a constant heat capacity allows us to avoid a so-called
enthalpy transformation and, together with the subquadratic growth of the heat conductivity,
to deduce a priori estimates
and the positivity of the temperature by carefully adapting the methods developed in
\cite{FePeRo09ESPT,RoRoESTC14}.
\par
When 
 taking the time discrete-to-continuous limit, 
we first pass 
to the limit in the weak momentum balance.
From this we also deduce a (time-continuous)
mechanical energy \emph{inequality} 
by lower semicontinuity arguments.
Next we pass to the limit in the
semistability inequality using so-called mutual
recovery sequences. As a further step we verify that 
the mechanical energy balance is satisfied as an \emph{equality}:
this follows 
from the momentum balance and the semistability so far
obtained. This result allows  us  to conclude the convergence of the
viscous dissipation terms, which, in turn, is crucial for the limit
passage in the heat equation.
See Sections \ref{Mombal}--\ref{EEHeat}.  
\paragraph{Some remarks on the thermal properties of system \eqref{ourPDE}
and its applicability.}
For the thermodynamical derivation of the PDE system \eqref{ourPDE}
one may follow the thermomechanical modeling
by Fr\'emond in \cite[Chapter 12]{Fre02NST} or
Roub\'i\v cek in \cite{Roub10TRIP}. In particular, the free energy density
associated with \eqref{ourPDE} is given by
\begin{equation}
\label{freeEn}
F(e(u),z,\nabla z,\theta):=\tfrac{1}{2}\CC(z)e(u):e(u)+G(z,\nabla z)+\varphi(\theta)-\theta\,\BB:e(u)\,,
\end{equation}
which leads to the entropy density $S$ and the internal energy density $U$ of the form
\begin{align*}
S(e(u),z,\nabla z,\theta)&=-\partial_\theta F
=\BB:e(u)-\varphi'(\theta)\,,\\
U(e(u),z,\nabla z,\theta)&=F+\theta\, S=\tfrac{1}{2}\CC(z)e(u):e(u)+G(z,\nabla z)
+\varphi(\theta)-\theta\,\varphi'(\theta)\,,
\end{align*}
where $\varphi$ is a function such that $c_\mathrm{V}(\theta):=\partial_\theta U=-\theta\,\varphi''(\theta)$
is the specific heat capacity,
and $S$ and $U$ satisfy a Gibbs' relation: $\partial_\theta U=\theta\,\partial_\theta S$.
Starting from the entropy equation, which balances the changes of entropy with the heat flux
and the heat sources given by the dissipation rate and  the  external sources~$H$,
\begin{equation*}
\theta \, \partial_\theta S \,  \dot\theta  +\mathop{\mathrm{div}}j=
\mathrm{R}_1(\dot z)+ \left( \DD(z,\theta)e(\dot u)  -\theta\,\BB \right)  :e(\dot u)+H\,,
\end{equation*}
and then invoking Fourier's law $j=  -  \KK(z,\theta)\nabla\theta$ as well as the above Gibbs' relation,
the choice $\varphi(\theta)=\theta(1-\log\theta)$ indeed results
in the heat equation \eqref{eq:theta} with $c_\mathrm{V}(\theta)=\mathrm{const.}=1$.
\par
In fact, the temperature dependence of the heat capacity can be described
by the classical Debye model, see e.g.\
\cite[Sect.\ 4.2, p.\ 761]{Wed97LPC}.
In a first approximation it predicts a cubic growth of $c_\mathrm{V}$ with respect to temperature up to a
certain, material-specific temperature, the so-called Debye temperature $\theta_\mathrm{D}$, whereas
for $\theta\gg\theta_\mathrm{D}$ it can be approximated by $c_\mathrm{V}\equiv\mathrm{const}$.
Thus, the use of \eqref{eq:theta} with $c_\mathrm{V}(\theta)=\mathrm{const.}$
(normalized to $c_\mathrm{V}(\theta)=1$ for shorter presentation) is justified
if the temperature range of application is assumed to be above Debye temperature,
i.e., $\theta\gg\theta_\mathrm{D}$. Indeed, our main existence Theorem \ref{thm:main},
see also Proposition~\ref{prop:exist-discrete-problem}, contains
an enhanced positivity estimate, which ensures that the temperature $\theta$, as a
component of an energetic solution $(u,z,\theta)$,
always stays above a tunable threshold (to be tuned to~$\theta_\mathrm{D}$),
provided that the initial temperature and
the heat sources $H$ are suitably large, see \eqref{teta-pos+}.
\par
In this context, let us here also allude to our hypothesis on the heat
conductivity tensor $\KK(z,\theta)$,
which is assumed to have subquadratic growth in $\theta$, see \eqref{ass-K-b}.
According to experimental findings, cf.\ \cite{Eie64MDWH,Kle12LWP}, polymers
such as e.g.\  polymethylmethacrylate  (PMMA), exhibit
such a subquadratic growth of the heat conductivity.
In contrast, for metals the heat conductivity is ruled by the electron thermal conductivity.
For this, the Wiedemann-Franz law  states a linear dependence on the temperature,
cf.\ \cite[Chapter 17]{CalRet12FMSE}.
Moreover, let us mention that the analytical results in \cite{FePeRo09ESPT}
are obtained under the assumption of superquadratic growth,
which is justified by the examples  on nonlinear heat conduction given in \cite{ZelRaj02PSWH},
that are related to radiation heat conduction or electron/ion heat conduction in a plasma.
Thus, in conclusion, the thermal properties of our model rather comply with polymers than with metals.
\paragraph{Vanishing viscosity and inertia.}
 Finally, in Section  \ref{s:6}  ahead we will address the    analysis of system \eqref{ourPDE}
as the rates of the external load and of the heat sources become slower and slower. Therefore, we will rescale time by a factor
$\eps$ and perform the asymptotic analysis as $\eps \downarrow 0$ of the rescaled  system, i.e.\
with   vanishing viscosity and inertia in the momentum equation, and  vanishing viscosity in the heat equation. Before entering into the
details of our result, let us briefly overview some related literature.
  \par  
On the one hand, the asymptotic analysis for vanishing viscosity and inertia of the sole  momentum balance  has been the subject of earlier
work: we refer, e.g., to \cite{MaSiGaMM} for study of   the   purely elastic limit of dynamic viscoelastic solutions 
to a frictional contact problem, in terms of a graph solution notion. This problem was approached from a more abstract viewpoint in \cite{MaReSo},
with applications to finite-dimensional mechanical systems featuring elastic-plastic behavior with linear hardening in \cite{MaPeMM}. 
On the other hand, a well-established approach to  fully rate-independent systems consists in viscously regularizing the rate-independent flow rule 
for the internal variable (typically coupled with a purely elastic equilibrium equation for the displacements), and taking the vanishing-viscosity limit. 
This  leads to \emph{parameterized}/$\BV$ solutions, encoding
information on the energetic behavior of the system at jumps, see e.g.\ \cite{EM,MRS09,MRS10,DDS}, as well as e.g.\ 
\cite{KMZ,LT,KnRoZa11} for applications to fracture and damage.  We also mention \cite{MR2851894,MR3651608} for finite-dimensional singularly perturbed second order potential-type equations.  
The convergence of kinetic variational inequalities to rate-independent quasistatic variational inequalities was tackled in \cite{MiPeMa}. 
 \par
Let  us point out that our analysis is substantially
 different from the   ``standard'' vanishing-viscosity  approach to rate-independent systems, since in our context viscosity (and inertia for the momentum equation) vanish in the heat and momentum balances, only, while
 we keep the flow rule for the damage parameter rate-independent. 
  In fact, our study is akin to the vanishing-viscosity
and inertia  analysis
 that has been  addressed, in the momentum equation only, for isothermal, rate-independent processes with  dynamics
in 
\cite{Roub09RIPV,Roub13ACVB},
 leading to an energetic-type notion of solution.
 We also refer to \cite{DMScQEPP13,Scal14LVDP} for a combined vanishing-viscosity limit in the momentum equation and in the flow rule, in the cases
  of perfect plasticity and delamination, respectively
  \par
The coupling with the temperature equation attaches 
   an additional difficulty   to   our own  vanishing-viscosity analysis. 
   Because of this,  it will be essential to assume an appropriate scaling of the tensor
of heat conduction coefficients:
in fact, we shall require that the conductivity matrix ($\KK$ in \eqref{eq:theta}) diverges as inertia and viscosity vanish.  
This reflects the
fact that in the slow-loading regime  heat propagates at infinite speed. 
Thus, in the  slow-loading  limit  we will  obtain that the temperature is  spatially constant and its evolution is  fully decoupled
from the one of the mechanical variables.   Indeed, in  Theorem \ref{Thm6.1}.
 we will
prove convergence  as $\eps \downarrow 0$ of energetic solutions $ (u_\eps, z_\eps,\theta_\eps)$ of the rescaled system to a triple $(u,z,\Theta)$ such that  
\begin{compactitem}
\item[-]
$(u,z)$ 
is 
\emph{local solution}  (according to the notion introduced in 
\cite{Miel08?DEMF, Roub13ACVB}) to the (fully rate-independent)  system consisting of the
static momentum balance and of the rate-independent flow
rule for damage;
\item[-]  under a suitable scaling condition on the heat sources, the spatially constant function $\Theta$  satisfies an ODE that involves 
a nonnegative defect measure arising from 
 the limit of the viscoelastic dissipation term.
\end{compactitem} 
\par
\paragraph{Plan of the paper.}
The assumptions on the material quantities
and the statement of the existence results for energetic solutions are given in Section \ref{s:3}.
In Section \ref{s:4} we present the properties of time-discrete solutions,
hence in Section \ref{s:5} we prove the main theorem by passing
to the time-continuous limit by variational convergence techniques.
Finally,  Section \ref{s:6} is devoted to  the asymptotics
for vanishing viscosity and inertia.
\par
%
%
%
%
%
%
%
%
%
%
%
\section{Setup and  main result}
\label{s:3}
%
\noindent {\bf Notation: } Throughout this paper, for a given Banach
space $X$ we will denote by $\pairing{}{X}{\cdot}{\cdot}$ the
duality pairing between $X^*$ and $X$, and by $\BV([0,T];X), $
resp.\ $\rmC_\mathrm{weak}^0 ([0,T];X)$, the
space of the bounded variation, resp.\ weakly continuous, functions
 with values in $X$. Notice
that we shall consider any $v\in \BV([0,T];X) $ to be defined
\emph{at all} $t\in [0,T].$
We also mention that the symbols $c,\,C, \, C' \ldots $ will be used
to denote a positive constant depending on given data, and possibly
varying from line to line. Furthermore in proofs, the symbols $I_i$,
$i=1,\ldots$, will be place-holders for several integral terms
popping up in the various estimates. We warn the reader that we will
not be self-consistent with the numbering so that, for instance, the
symbol $I_1$ will occur in several proofs with different meanings.
%
\subsection{Assumptions}
\label{s:ass}
%
We now specify the assumptions on the domain $\Omega$,
on the nonlinear functions  featured
in  \eqref{ourPDE}, on the initial data, and on the loading and
source terms, under which our existence result, Theorem \ref{thm:main}, holds.
Let us mention in advance that, in order to
simplify the exposition in Sections \ref{s:3}--\ref{s:5}, and
in view of the analysis  for vanishing viscosity and inertia
in Section\ \ref{s:6}, cf.\ \eqref{pistar},
we will suppose that   the matrix of thermal expansion coefficients
is a given symmetric matrix $\BB \in \Rsym$.
 We  instead  allow the elasticity and viscosity
tensors to depend on the state variables $z$ and $(z,\theta)$,
respectively, thus  we need to impose suitable growth and coercivity
conditions.  We will also make growth assumptions for the matrix of
heat conduction coefficients, which are suited for our analysis and
which are in the line of \cite{FePeRo09ESPT, RoRoESTC14}.
These growth conditions  will play a key role in the derivation
of estimates for the temperature $\theta$, in that it will allow us
to cope with the quadratic right-hand side of \eqref{eq:theta}.
 Before detailing the standing assumptions of this paper, let us mention that, 
to  ease the presentation, we will assume the functions of the temperature featuring in the model to be defined also for nonpositive values of 
$\theta$. At any rate,
later on we will prove the existence of solutions such that the temperature is bounded from below by a positive constant, see \eqref{strict-pos}--\eqref{teta-pos+}.
\paragraph{\bf Assumptions on the domain.}
We assume that
\begin{equation}
\begin{split}
\label{ass-dom}
&\Omega\subset\R^d\,,\;d\in\{2,3\}\,,\text{ is a bounded domain with Lipschitz-boundary
$\partial\Omega$ such that\ }\\
&\partial_\mathrm{D}\Om\subset\partial\Om\text{ is nonempty and relatively open and }
\partial_{\mathrm{N}}\Om:=\partial\Om\backslash\partial_\mathrm{D}\Om\,.
\end{split}
\end{equation}
 Moreover, we will use the following notation for the state spaces for $u$ and $z$:
\begin{equation}
\begin{split}
\label{statesp}
H_{\mathrm{D}}^1(\Omega;\R^d)&:=
\{v\in H^1(\Omega;\R^d)\colon \ v=0 \text{ on } \partial_{\mathrm{D}}\Om \text{ in  the  trace sense} \}\,,\\
\calZ&:=\{ z \in W^{1,q}(\Omega)\colon z \in [0,1] \ \aein \Omega \}\,,
\end{split}
\end{equation}
with fixed $q>1$, cf.\ \eqref{G-growth}.
Analogous notation will be employed for the Sobolev spaces $W^{1,\gamma}_\mathrm{D}$, $\gamma\geq1$.
\paragraph{\bf Assumptions on the material tensors.}
We require that the tensors $\BB\in\R^{d \times d}$,  $\CC\colon\R\to \R^{d \times d \times d \times d}$,
and $\DD\colon\R\times\R\to \R^{d \times d \times d \times d}$ fulfill
\begin{subequations}
\label{ass-CD}
\begin{align}
&\label{ass-B} \BB \in \Rsym \text{ and set }
C_\BB:=\mod{\BB}\,,
\\
&
\label{conti}
 \CC\in \rmC^{0,1} (\R;  \R^{d \times d \times d \times d})
\text{ and } \DD \in \rmC^0 (\R\times\R;  \R^{d \times d \times d \times d}) \,,
\\
\label{assCD-1}
&
\CC(z), \, \DD(z,\theta) \in\Tsym \text{ and are positive definite for all }
z \in \R\,,\ \theta\in\R\,,
\\
& \label{conti-c} \exists\, C_\CC^1, \, C_\CC^2 >0  \ \ \forall\, z
\in \R \ \ \forall\, A\in \Rsym \colon \quad
C_\CC^1 \mod{A}^2 \le \CC(z)A: A\le C_\CC^2 \mod{A}^2\,,\\
&
\label{assCD-3}
\exists\, C_\DD^1, \, C_\DD^2 >0  \ \ \forall\, z \in \R \ \ \forall\, \theta \in \R \ \
\forall\, A\in \Rsym \colon \quad
C_\DD^1 \mod{A}^2 \le \DD(z\,,\theta)A: A\le C_\DD^2 \mod{A}^2\,.
\end{align}
\end{subequations}
In the expressions above, $\Rsym$ denotes the subset of
symmetric matrices in $\R^{d \times d}$ and $\Tsym$ is the subset of
symmetric tensors in $\R^{d \times d \times d \times d}$. In
particular,
$$
\CC(z)_{ijkl}{=}\CC(z)_{jikl}{=}\CC(z)_{ijlk}{=}\CC(z)_{klij}\;\text{
and }\;
\DD(z,\theta)_{ijkl}{=}\DD(z,\theta)_{jikl}{=}\DD(z,\theta)_{ijlk}{=}\DD(z,\theta)_{klij}\,.
$$
In addition to \eqref{ass-CD}, we impose that $\CC(\cdot)$ is
monotonically nondecreasing, i.e.,
\begin{equation}
\label{mono}
\forall\, A\in \Rsym \;\;\forall\,0\leq z_1\leq z_2\leq1 \colon \quad
\CC(z_1)A:A\leq \CC(z_2)A:A\,.
\end{equation}

\paragraph{\bf Assumptions on the damage regularization.}
We require that $G\colon\R\times\R^d\to  \R\cup\{\infty\}$  fulfills
\begin{subequations}
\label{assG}
\begin{eqnarray}
\label{Gind}
&&\text{Indicator: For every } (z,\xi)\in\R\times\R^d\colon\quad G(z,\xi)<\infty\;\Rightarrow\;z\in[0,1]\,;\\
\label{Gcont}
&&\text{Continuity: } 
G \text{ is continuous on its domain } \mathrm{dom}(G)\,, \quad  G\ge0\,,  \quad
\text{and} \quad  
G(0,0)=0\,;\\
\label{Gconv}
&&\text{Convexity:  For every }z\in\R,\; G(z,\cdot)\text{ is convex;}\\
\nonumber
&&\text{Growth: There exist constants }q>1 \text{ and }C_G^1,C_G^2>0
\text{ such that for every }(z,\xi)\in\mathrm{dom}(G)\qquad\qquad\\
\label{G-growth}
&&\hspace*{4cm}
C_G^1 (\mod{\xi}^q-1) \le G(z,\xi) \le C_G^2(\mod{\xi}^q+1)\,.
\end{eqnarray}
\end{subequations}
\begin{rmk}[Properties of the regularizing term]
\label{LSCG} 
Since we are encompassing the feature that
$z(\cdot,x)$ is decreasing for almost all $x \in \Omega$,
starting from an initial datum $z_0 \in [0,1]$ a.e.\ in $\Omega$,
the $z$-component of any energetic solution to \eqref{ourPDE} will fulfill
$z(t,x) \leq 1$ a.e.\ in $\Omega$. Therefore, we could
weaken \eqref{Gind} and just require that the domain of  $G$ is a subset of $[0,\infty)$.

Furthermore, we may require the  third  of  \eqref{Gcont} without loss of generality,
since adding a constant to $G$ shall not affect our analysis.

Further observe that the above assumptions \eqref{assG} ensure that the integral functional
\begin{equation*}
\calG\colon L^r(\Omega)\times L^q(\Omega;\R^d)\to\R\cup\{\infty\}\,,\quad
\calG(z,\xi):=\int_\Omega G(z,\xi)\,\mathrm{d}x
\end{equation*}
is lower semicontinuous with respect to strong convergence in $L^r(\Omega)$ for any $r\in[1,\infty)$
and weak convergence in $L^q(\Omega;\R^d)$, cf.\ e.g.\ \cite[Theorem 7.5, p.\ 492]{FoLeo07}.
In addition, $\calG$ is continuous with respect to strong convergence
in $\left( L^r(\Omega)\times L^q(\Omega;\R^d) \right) \cap\mathrm{dom}(G)$.
\end{rmk}
\begin{rmk}[Example: Phase-field approximation  of fracture]
\label{rmk:a-t}
Starting from the work of Ambrosio and Tortorelli \cite{AmbTor},
gradient damage models have been extensively used in recent years to predict crack propagation
in brittle or quasi-brittle materials, by means of phase-field approximation \cite{BoFrMa08VAF}. 
In this approach, a sharp crack is regularized by defining an internal variable
that interpolates continuously between sound and fractured material.
In the mathematical literature,
evolutionary problems for phase-field models were considered
for instance in the fully quasistatic case \cite{Giac05ATAQ},
in viscoelasticity as a gradient flow \cite{BabMil},
and in dynamics \cite{LarOrtSue}, always for isothermal systems.
A thermodynamical model for regularized fracture with inertia
was proposed and treated numerically e.g.\ in \cite{MieHofWel}.
The passage to the limit from phase-field to sharp crack,
though successfully treated in the quasistatic \cite{Giac05ATAQ} and in the viscous case \cite{BabMil},
is by now an open problem in dynamics and is outside the scope of this contribution.
\par
In this context, typical examples for the regularizing term are functionals of Modica-Mortola type,
$$
\mathcal{G}^{q}_{\rm MM}
(z,\nabla z) = \int_\Omega
G^{q}_{\rm MM}(z,\nabla z) \dd x \qquad \text{with }
  G^q_{\rm MM}(z,\nabla z):=  \mod{\nabla z}^q + W(z)  +  I_{[0,1]}(z)  \,,
$$
where $q>1$, $W$ is a suitable potential, and $I_{[0,1]}(z):=0$ 
if $z\in[0,1]$, $I_{[0,1]}(z):=+\infty$ otherwise.
Such regularization agrees with the above assumptions up to
an additive constant.
\par
Notice that in Section \ref{s:4}, to construct discrete solutions,
we will consider unilateral minimum problems of the type
$$
\min_{z\in\calZ}
\left\{ \integ{\Om}{\tfrac12 \CC(z)e(u) \psm  e(u)}{x} + \integ{\Om}{G(z,\nabla z)}{x}
+ \calR_1(z-\bar z) \right\}
$$
for given $u\in H^1_\mathrm{D}(\Om;\R^d)$ and a given  $\bar z\in\calZ$ defined in \eqref{statesp}.
Setting $\CC(z):=(z^2+\delta)\,I$ with $\delta>0$, and $G:=G^{2}_{\rm MM}$ with $W(z):=\frac12(1+z^2)$,
the minimum problem is equivalent to
$$
\min_{0\le z\le  \bar z  } \left\{ \integ{\Om}{(\tfrac12 (z^2+\delta)\mod{e(u)}^2}{x}
+  \integ{\Om}{ \tfrac1{2\delta}  (1-z)^2}{x} +  \integ{\Om}{\delta \mod{\nabla z}^2}{x}   \right\} \,,
$$
that is the classical minimization of the Ambrosio-Tortorelli functional, see \cite{AmbTor,Giac05ATAQ}.
The generalization to $G=G^{q}_{\rm MM}$ with $q>1$ was considered in \cite{Iur13-ACV}.
In this case one may want an effective dependence of the viscous tensor on $z$,
choosing $\DD(z,\theta)=\CC(z)$ as in \cite{LarOrtSue}.
\end{rmk}
\paragraph{Assumptions on the heat conductivity.}
On   $\KK \colon \R \times \R \to \R^{d \times d}$ we assume that
\begin{subequations}
\label{ass-K}
\begin{align}
& \label{ass-K-a}
\mathbb{K}  \in \rmC^0 (\R \times \R;\R^{d \times d}) \,,
\quad \KK(z,\theta) \in \Rsym \text{ for all } z \in \R\,,\ \theta\in\R\,,
\\
& \label{ass-K-b} \exists\,  \kappa\in(1,\kappa_d) \ \ \exists\,
c_1, \, c_2 >0 \ \ \forall\, (z,\theta) \in \R \times \R \ \
\forall \, \xi \in \R^d  \colon
\quad
\begin{cases}
 c_1 (\mod{\theta}^{\kappa} +1) |\xi|^2 \leq   \mathbb{K}(z,\theta)\xi \cdot \xi\,,
 \\
 | \mathbb{K}(z,\theta) | \leq c_2 (\mod{\theta}^{\kappa} +1)\,,
   \end{cases}
\end{align}
\end{subequations}
where $\kappa_d=5/3$ for $d{=}3$ and $\kappa_d=2$ for $d{=}2$.
\par
The bound $\kappa_d$ essentially comes into play 
in the derivation of the \emph{Fifth a priori estimate} (cf.\ the proof of Proposition \ref{Apriori}), 
and when passing from
time-discrete to continuous in the heat equation, cf.\ Proposition \ref{Heat}.
Essentially,  it arises as a consequence of the enhanced integrability
of the approximating temperature variables obtained by interpolation in \eqref{apthetaLp}.
%
\paragraph{Assumptions on the initial data.} 
We impose that
\begin{subequations}\label{assu-init}
\begin{equation}\label{hyp-init}
u_0 \in\, H_{\mathrm{D}}^1(\Omega;\R^d)\,,\quad\dot u_0 \in\, L^2(\Omega;\R^d)\,,\quad z_0\in \calZ\,,
\end{equation}
\be
\label{wzero}
\quad \theta_0 \in\, L^{1} (\Omega)\,,
\quad\text{and }\theta_0
\geq \theta_*>0 \ \ \aein \Omega\,,
\end{equation}
\end{subequations}
where the state spaces $H^1_\mathrm{D}(\Omega;\R^d)$ and $\calZ$ are defined in \eqref{statesp}.
\paragraph{Assumptions on the loading and source  terms.}
On the data $\fv,\, \fs,\, \hv$, and $\hs$ we require that
\begin{subequations}
\label{ass-data}
\begin{align}
&
\label{ass-forces}
\fv\in H^1(0,T; H^1_\mathrm{D}(\Om;\R^d)^*)\,, \qquad
\fs\in H^1(0,T; L^2(\partial_{\mathrm{N}}\Om;\R^d))\,,
\\
\label{ass-dataHeat}
&
\begin{aligned}
&
\hv\in L^1(0,T; L^1(\Om)) \cap L^2(0,T; H^1(\Omega)^*)\,, \quad \hv \geq 0   \text{ a.e.\ in }
(0,T) \times \Omega \,, \\
& \hs\in L^1(0,T; L^2(\partial\Om))\,, \quad  \hs \geq 0 \text{
a.e.\ in } (0,T) \times \partial \Omega \,,
\end{aligned}
\end{align}
\end{subequations}
For later convenience, we also introduce $\bigF\colon [0,T] \to  H^1_\mathrm{D}(\Om;\R^d)^*$ defined by
\begin{equation}
\label{bigF}
\pairing{}{H^1_\mathrm{D}(\Om;\R^d)}{ \bigF(t)}{v}:=
 \pairing{}{H^1_\mathrm{D}(\Om;\R^d)}{ \fv(t)}{v} + \integ{\partial_{\mathrm{N}}\Om}{\fs\ps v}{\acca^{d-1}(x)}
\qquad \text{for all } v \in  H^1_\mathrm{D}(\Om;\R^d)\,.
\end{equation}
It follows from \eqref{ass-forces} that $\bigF \in  H^1(0,T;
H^1_\mathrm{D}(\Om;\R^d)^*)$.
%
%
%
\subsection{Weak formulation and main existence result}
%
As already mentioned, following \cite{Roub10TRIP},
the \emph{energetic} formulation
of (the initial-boundary value problem associated with) system \eqref{ourPDE} consists of the
 variational formulation of the momentum  and of the heat equations \eqref{eq:u} and \eqref{eq:theta},
with suitable test functions,
 and of a semistability condition joint with a \emph{mechanical energy} balance,
providing the weak formulation of the damage equation \eqref{eq:z}. The latter relations  feature
the mechanical  (quasistatic) energy associated with \eqref{ourPDE}, i.e.,
\begin{equation*}
\E(t,u,z):=\integ{\Om}{(\tfrac12 \CC(z)e(u) \psm  e(u)
+G(z,\nabla z))}{x} - \pairing{}{H^1_\mathrm{D}(\Om;\R^d)}{\bigF(t)}{u}\,,
\end{equation*}
as well as the rate-independent dissipation potential,
given as the integrated version of \eqref{dissip-potential}
\begin{equation}
\label{integrated-1-dissip}
\calR_1(\dot z):=\int_\Omega \mathrm{R}_1(\dot z)\,\mathrm{d}x\,.
\end{equation}
\par
In Definition \ref{def4} below, the choice of the test functions for the weak momentum equation
reflects  the regularity \eqref{reguu1} required for $u$, which in
turn will derive from the standard energy estimates that can be
performed on system~\eqref{ourPDE}. As we will see, such estimates
only yield $\theta \in L^\infty (0,T;L^1(\Omega))$. In fact, the
further regularity \eqref{reguw} for $\theta$ shall result from a
careful choice of test functions for the time-discrete version of
\eqref{eq:theta}, and from refined interpolation arguments, drawn
from \cite{FePeRo09ESPT}. Finally, the $\BV ([0,T];
W^{2,d+\delta}(\Omega)^*)$-regularity for $\theta$ follows from a
comparison argument. The choice of the test functions in
\eqref{weak-heat} is the natural one in view of \eqref{reguu}.
\begin{defin}[Energetic solution \eqref{reguu}--\eqref{energetic-formulation}]
\label{def4}
\upshape
 Given a quadruple of initial data
$(u_0,\dot u_0,z_0,\theta_0)$  satisfying \eqref{assu-init},
we call a triple $(u,z,\w)$ an \emph{energetic solution} of
the Cauchy problem for the  PDE system
\eqref{ourPDE} complemented with the boundary conditions
\eqref{bc} if
\begin{subequations}
\label{reguu}
\begin{align}
\label{reguu1}
& u\in
H^1(0,T;H_{\mathrm{D}}^{1}(\Omega;\R^d)) \cap 
W^{1,\infty}(0,T;L^2(\Omega;\R^d)) \,,
\\
\label{reguz}
&
\begin{aligned}
&\!
z \in  L^\infty (0,T;W^{1,q}(\Omega)) \cap L^\infty ((0,T)\times \Omega) \cap
\BV([0,T];L^1(\Omega))\,, \\
&\!
z(t,x) \in [0,1] \ \foraa (t,x) \in
(0,T)\times\Omega\,,
\end{aligned}
\\
\label{reguw}  &
\theta \in L^2(0,T; H^1(\Omega)) \cap L^\infty (0,T; L^1(\Omega)) \cap
\BV ([0,T]; W^{2,d+\delta}(\Omega)^*) \,,
\end{align}
\end{subequations}
such that the triple $(u,z,\w)$ complies
with the initial conditions
\begin{equation*}
u(0)=u_0\,, \quad \dot u(0)=\dot u_0\,, \quad  z(0)=z_0\,, 
\quad \theta(0)=\theta_0 \quad \aein  \Omega\,,
\end{equation*}
and with the following properties:
\begin{compactitem}
\item
\emph{unidirectionality}: for a.a.\ $x\in\Om$, the function
$z(\cdot,x)\colon[0,T]\to[0,1]$ is nonincreasing;
\item
\emph{semistability}:
for every $t\in [0,T]$
\begin{subequations}\label{energetic-formulation}
\begin{equation}
 \label{semistab-general}
\forall\, \tilde{z}\in \calZ
\colon\quad
\E(t,u(t),z(t))\le\E(t,u(t),\tilde z)+{\mathcal R}_1(\tilde z-z(t))\,,
\end{equation}
where $\calZ$ is defined in \eqref{statesp};
 \item
\emph{weak formulation of the momentum equation}:  for all $t\in[0,T]$
\begin{equation}
\label{e:weak-momentum-variational}
\begin{aligned}
& \rho \integ{\Omega}{\dot u(t) \ps v(t)}{x}
- \rho\integlin{0}{t}{\integ{\Omega}{\dot u \ps \dot v}{x}}{s}
+\integlin{0}{t}{\!\integ{\Omega}{\left(\DD(z,\theta)e(\dot u){+}\CC(z) e(u)  
{-}\theta\,\BB \right) \psm e(v)}{x}}{s}
\\
&= \rho\integ{\Omega}{ \dot u_0 \ps v(0)}{x}+
\integlin{0}{t}
{\pairing{}{H^1_\mathrm{D}(\Omega;\R^d)}{\bigF}{v}}{s}
\end{aligned}
\end{equation}
for all test functions $v \in L^2(0,T;H^1_{\mathrm{D}}(\Omega;\R^d))
\cap W^{1,1}(0,T;L^2(\Omega;\R^d))$;
\item
\emph{mechanical energy equality}: for all  $t \in[0,T]$
\begin{equation}
\label{mech-energy-ineq}
\begin{aligned}
 &  \tfrac\rho2{\integ{\Om}{\mod{\dot u(t)}^2}{x}}+
\E(t,u(t),z(t))  + \integ{\Om}{(z_0{-}z(t))}{x}
+ \integlin{0}{t}{\!\integ{\Omega}{ \left(\DD(z,\theta)e(\dot u){-}\theta\,\BB \right)\psm e(\dot u)}{x}}{s}
\\
&= \tfrac\rho2{\integ{\Om}{\mod{\dot u_0}^2}{x}}+
\E(0,u_0,z_0)
+\int_0^t\partial_t\E(s,u(s),z(s))\,\mathrm{d}s
\,,
\end{aligned}
\end{equation}
where $\partial_t \E(t,u,z) = -\pairing{}{H^1_\mathrm{D}(\Om;\R^d)}{\dot{\bigF}(t)}{u}$;
\item
\emph{weak  formulation of the heat equation}: for all $t\in[0,T]$
\begin{equation}
\label{weak-heat}
\begin{aligned}
&
\pairing{}{W^{2,d+\delta}(\Omega)}{\theta(t)}{\testw(t)}
-\integlin{0}{t}{\!\integ{\Om}{\theta\,\dot\testw}{x}}{s}
+\integlin{0}{t}{\!\integ{\Omega}{\KK(\theta,z)\nabla\theta\cdot\nabla \testw}{x}}{s}
\\
&= \integ{\Omega}{\w_0\,\testw(0)}{x}+\integlin{0}{t}{\!\integ{\Om}{\testw\mod{\dot z}}{x}}{s}
+\integlin{0}{t}{\!\integ{\Omega}{\left(\DD(z,\theta)e(\dot u)\psm e(\dot u)
{-} \theta\,\BB \right) \psm e(\dot u)\testw}{x}}{s}\\
&\phantom{=}
+\integlin{0}{t}{\!\integ{\partial\Om}{\hs\testw}{\acca^{d-1}(x)}}{s}
+\integlin{0}{t}{\!\integ{\Om}{\hv\testw}{x}}{s}
\end{aligned}
\end{equation}
for all test functions $\eta\in H^1(0,T;L^2(\Omega))\cap
\mathrm{C}^0([0,T];W^{2,d+\delta}(\Omega))$, for some fixed $\delta>0$.
Here and in what follows, $\mod{\dot z}$ denotes the total variation measure of $z$
(i.e., the heat produced by the rate-independent dissipation), which is defined
on every closed set of the form   $[t_1,t_2]\times C\subset[0,T]\times\overline{\Omega}$  by
\begin{equation*}
 \mod{\dot z}([t_1,t_2]{\times} C):=\int_C \mathrm{R}_1 (z(t_2)-z(t_1))\,\mathrm{d}x\,, 
\end{equation*}
and, for simplicity, we shall write
$\integlin{0}{t}{\integ{\Om}{\testw\mod{\dot z}}{x}}{s}$ instead of
$\iint_{[0,t] \times \Omega} \testw \mod{\dot z}(\dd s \dd x ) $.
\end{subequations}
\end{compactitem}
\end{defin}
 Since $z$ has at most $\BV$-regularity as a function of time, 
it may have (at most countably many) jump points, where the  left and right limits $z(t_-),\, z(t_+) \in L^1(\Omega)$  differ. Indeed, from $z\in L^\infty (0,T;W^{1,q}(\Omega)) \cap \BV ([0,T];L^1(\Omega))$ it is  immediate to deduce that, at every $t\in [0,T]$ (with the standard conventions
$z(0_-): = z(0)$ and $z(T_+): = z(T)$),
 both  $z(t_-)$ and $z(t_+)$ 
are elements in    $W^{1,q}(\Omega)$, with $z(t_-) =\lim_{s\uparrow t} z(s)$
 and  $z(t_+) =\lim_{s\downarrow t} z(s)$ w.r.t.\ the weak topology of $W^{1,q}(\Omega)$. In particular, the right limit $z(0_+)$ exists, and it may be $z(0_+) \neq z(0) =z_0$ (observe that, by \eqref{assu-init} the initial condition is fulfilled as an equality in $W^{1,q}(\Omega)$).
  In that case, the mechanical energy balance \eqref{mech-energy-ineq} records the jump of the stored/dissipated energies at the initial time. 
 \begin{rmk}[Total energy balance]
Summing up the mechanical energy inequality \eqref{mech-energy-ineq}
and the weak heat equation \eqref{weak-heat} tested by $\eta\equiv 1$, 
yields the total energy balance
\begin{equation*}
\begin{split}
\label{totenbal}
&\int_\Omega\tfrac{\rho}{2}|\dot u(t)|^2\,\mathrm{d}x
+\calE(t,u(t),z(t))+ \int_\Omega \theta(t)  \,\mathrm{d}x 
=\int_\Omega\tfrac{\rho}{2}|\dot u_0|^2\,\mathrm{d}x
+\calE(0,u_0,z_0)+\int_\Omega \theta_0  \,\mathrm{d}x 
\\
&\phantom{=}+\int_0^t \partial_t\calE (s,u(s),z(s))\,\mathrm{d}s
+\int_0^t\!\int_\Omega \hv\,\mathrm{d}x\,\mathrm{d}s
+\int_0^t\!\int_{\partial\Omega} \hs\,\mathrm{d}\calH^{d-1}(x)\,\mathrm{d}s\,.
\end{split}
\end{equation*}
\end{rmk}
\begin{rmk}[Improved regularity on $\ddot u$]
\label{rmk:more-regularity}
From the definition of energetic solution we can gain improved regularity
for the time derivatives of the displacement.
Indeed, let $(u,z,\w)$ be as in \eqref{reguu} and such that
the weak momentum equation \eqref{e:weak-momentum-variational} holds.
Then \eqref{eq:u} holds in the sense of distributions and
$$
\rho \norm{\ddot{u}}{L^2(0,T; H^1_{\mathrm{D}}(\Om;\R^d)^*)}
= \sup_{\norm{v}{}\le1}
\integlin{0}{T}{\!\integ{\Omega}{\left(\DD(z,\theta) e(\dot u)
+  \CC(z) e(u)   -\theta\,\BB \right) \psm e(v)}{x}}{t}
-\integlin{0}{T}{\pairing{}{H^1_\mathrm{D}(\Om;\R^d)}{\bigF}{v}}{t} \,,
$$
where the supremum is taken over all functions such that
$\norm{v}{L^2(0,T; H^1_{\mathrm{D}}(\Om;\R^d))}\le1$. The left-hand
side of the previous equality is uniformly bounded thanks to
\eqref{ass-CD}, \eqref{bigF}, and \eqref{reguu}, thus we deduce that
$\ddot u\in L^2(0,T;H^1_\mathrm{D}(\Om;\R^d)^*)$. Since the spaces
$H^1_\mathrm{D}(\Om;\R^d) \subset L^2(\Om;\R^d) \subset
H^1_\mathrm{D}(\Om;\R^d)^*$ form a Gelfand triple, in view of e.g.\ 
\cite[Chap.\ 1, Sec.\ 2.4, Prop.\ 2.2]{LioMag72NHBV}, we
conclude that
\begin{equation}
\label{gelfand} \begin{aligned} &
\integlin{t_1}{t_2}{\pairing{}{H^1_\mathrm{D}(\Om;\R^d)}{\ddot u}{\dot
u}}{t}
\\ &
=\tfrac 1 2 \pairing{}{H^1_\mathrm{D}(\Om;\R^d)}{\dot u(t_2)}{\dot
u(t_2)}-\tfrac 1 2 \pairing{}{H^1_\mathrm{D}(\Om;\R^d)}{\dot u(t_1)}{\dot
u(t_1)}=\tfrac 1 2\|\dot u(t_2)\|^2_{L^2(\Om;\R^d)}-\tfrac 1 2\|\dot
u(t_1)\|^2_{L^2(\Om;\R^d)}
\end{aligned}
\end{equation}
for every $t_1,t_2\in[0,T]$. Hence, $\dot u$  can be used as a test
function in \eqref{e:weak-momentum-variational}.
\end{rmk}
We are now in a position to state the main result of this paper.
The last part of the assertion concerns  the strict positivity of the
absolute temperature $\theta$. In particular, under \eqref{h-pos+} below we are able to specify,
in terms of the given data, the constant
which bounds $\theta$ from below.
\begin{thm}[Existence of energetic solutions \eqref{reguu}--\eqref{energetic-formulation}]
\label{thm:main}
Under assumptions \eqref{ass-dom}--\eqref{mono}, \eqref{assG}, and \eqref{ass-K}, and
\eqref{ass-data} on the data $\fv,$ $\fs,$ $\hv,$ and $\hs$, for every
qua\-druple $(u_0,\dot{u}_0,z_0,\theta_0)$ fulfilling
\eqref{assu-init}  with $z_0$ satisfying \eqref{semistab-general},
there exists an energetic solution $(u,z,\theta)$ to the Cauchy
problem  for system~\eqref{ourPDE}.
\par
Moreover, there exists $\widetilde{\theta} >0$ such that
\begin{equation}
\label{strict-pos}
\theta (t,x)\geq \widetilde{\theta} >0 \qquad \foraa
(t,x)\in (0,T) \times \Omega\,.
\end{equation}
Furthermore, if in addition
\begin{equation}
\label{h-pos+}
\exists\, H_*>0 \colon\ 
\hv(t,x) \geq H_*  \ \foraa (t,x) \in (0,T) \times \Omega\;\text{ and }\;\,  
\theta_0(x) \geq \sqrt{H_*/\bar c} \ \ \foraa x \in \Omega\,,
\end{equation}
where $\bar{c} := \tfrac{ {(C_\BB)^2}}{2C_\mathbb{D}^1}$, then
\begin{equation}
\label{teta-pos+}
\theta (t,x)\geq \max\left\{ \widetilde{\theta} , \sqrt{H_*/\bar{c}} \right\} 
\quad \text{for a.a.}\ (t,x)\in (0,T) \times \Omega \,. 
\end{equation}
\end{thm}
The proof of Theorem \ref{thm:main} will be developed in
Sections \ref{s:4} and \ref{s:5} by time-discretization (see
Propositions \ref{Conv}--\ref{prop:existence-limit}).
\begin{rmk}[Time-dependent Dirichlet loadings]
\label{rmk:time-dependent-loading} The existence of energetic
solutions can be proven also when
\emph{time-dependent Dirichlet loadings} are considered
for  the displacement $u$ instead of the homogeneous Dirichlet
condition \eqref{bc}, in the case the viscous tensor $\DD$ is
\emph{independent} of $z$ and $\theta$. This restriction is due to technical
reasons, related to the derivation of suitable estimates for the
approximate solutions to \eqref{ourPDE}.
\par
An alternative damage
model, that still features a $(z,\theta)$-dependence of $\DD$, is discussed in \cite{MURPHYS-LRTT},
where a
time-dependent loading for $u$ can be encompassed in the analysis,
albeit under  suitable stronger conditions.
\end{rmk}
%
\begin{rmk}[Failure of ``entropic'' solutions]
\label{rmk:comp-betty} \upshape As already mentioned, the regularity
for the temperature $\theta \in L^2(0,T; H^1(\Omega))\cap\BV ([0,T];
W^{2,d+\delta}(\Omega)^*) $ results from careful estimates on the
heat equation \eqref{eq:theta}, tailored on the quadratic character
of its right-hand side and drawn from \cite{FePeRo09ESPT}. There,
the analysis of the \emph{full system} for phase transitions
proposed by Fr\'emond \cite{Fre02NST}, featuring a heat equation
with an $L^1$ right-hand side, was carried out.

The techniques from \cite{FePeRo09ESPT} have been recently extended
in \cite{RoRoESTC14} to analyze a model for \emph{rate-dependent} damage in thermo-viscoelasticity.
Namely, in place of the $1$-homogeneous dissipation potential $\mathrm{R}_1$
from \eqref{dissip-potential}, the flow rule for the damage parameter in
\cite{RoRoESTC14} features the quadratic dissipation
$\mathrm{R}_2 (\dot z) = \frac12 |\dot z|^2$ if $\dot z \leq 0$, and $\mathrm{R}_2 (\dot z) = \infty$ else.
Consequently, the heat equation  in  \cite{RoRoESTC14} is of the type
\begin{equation}
\label{heat-rr14}
\dot\theta-\div(\KK(z,\theta)\nabla\theta)=|\dot z|^2+ \DD(z) e(\dot u)\psm e(\dot u)
- \theta\,\BB\psm e(\dot u) + H \qquad \text{in } (0,T) \times \Omega\,.
\end{equation}
In \cite{RoRoESTC14}, under a weaker growth condition on $\KK$ than
the present \eqref{ass-K}, it was possible to prove an existence
result for a weaker formulation of \eqref{heat-rr14}, consisting of
an entropy inequality and of a total energy inequality. The
resulting notion of ``entropic'' solution, originally proposed in
\cite{FePeRo09ESPT}, indeed  reflects the strict positivity of
the temperature, and the fact that the entropy increases along
solutions. Without going into details, let us mention that this
entropy inequality is (formally) obtained by testing \eqref{heat-rr14} by
$\varphi\,\theta^{-1}$, with $\varphi$ a smooth test function, and
integrating in time. This procedure is fully justified because
$\theta$ can be shown to be bounded away from zero by a positive
constant, hence $\varphi(t)\,\theta^{-1}(t) \in L^\infty(\Omega)$ for
almost all $t\in (0,T)$, and the integrals $\int_0^T \int_\Omega
|\dot z|^2 \varphi\,\theta^{-1}  \dd x  \dd t $ and $\int_0^T \int_\Omega
\DD(z) e(\dot u)\psm e(\dot u) \varphi\,\theta^{-1}  \dd x \dd t  $
resulting from the first  and second terms on the right-hand side of
\eqref{heat-rr14}  are well-defined.

In the present \emph{rate-independent} context,
proving an existence  result for the entropic formulation  of \eqref{eq:theta}
seems to be out of reach. Indeed, in such formulation the term
$\int_0^T \int_\Omega |\dot z|^2 \varphi\,\theta^{-1}  \dd x  \dd t $
would have to be replaced by  $\int_{[0,T] \times  \Omega} \varphi\,\theta^{-1}  |\dot z| (\dd x  \dd t) $,
with $|\dot z|$ the total variation measure of $z$, cf.\ \eqref{weak-heat},
but the above integral is not well defined since  $\varphi\,\theta^{-1}$ is not a continuous function.
\end{rmk}
%
\section{Time-discretization}
\label{s:4}
\subsection{The time-discrete scheme}
\label{ss:4.1}
Given a partition
\[
0=t_n^0<\dots<t_n^n=T \qquad  \text{with}  \qquad t\kk-t\km=\tfrac{T}{n}=:\taun\,,
\]
we construct a family of  discrete solutions ${(u\kk,z\kk,\theta\kk)}_{k=1,\ldots,n}$ 
by solving recursively the time-discretization scheme
\eqref{time-discrete-scheme} below,
where the data $\bigF$, $\hv$, and $\hs$ are approximated by \emph{local means} as follows
\begin{equation}
\label{local-means} \bigF\kk:=
\tfrac{1}{\taun}\int_{t\km}^{t\kk} \bigF(s)\dd s\,,
\qquad   \hv\kk:= \tfrac{1}{\taun}\int_{t\km}^{t\kk} \hv(s)
\dd s\,, \qquad  \hs\kk:= \tfrac{1}{\taun}\int_{t\km}^{t\kk} \hs(s)
\dd s\,,
\end{equation}
and the above  integrals need to be understood in the
Bochner sense.
\par
Let us mention in advance that we have to add the  regularizing term
$-\taun\div(|e(u\kk)|^{\gamma-2} e(u\kk))$ in the discrete momentum equation, with $\gamma >4$.
 Basically, the reason  for this is that we need to compensate
the quadratic term in $e(u\kk)$  on the right-hand
side of the discrete heat equation \eqref{DSw3}.   In practice,  the term $-\taun\div(|e(u\kk)|^{\gamma-2} e(u\kk))$  
will have a key role in proving that the pseudomonotone operator in terms of which the (approximate) discrete system 
can be reformulated  is coercive, and thus such system admits solutions. 
Because of this additional regularization, it will be necessary to further approximate the initial datum $u_0$
from \eqref{hyp-init} by a sequence (cf.\ \cite[p.\ 56, Corollary 2]{Bure98SSOD})
\begin{equation}
\label{tilde-uzero}
{(u_n^0)}_n \subset
W_{\mathrm{D}}^{1,\gamma}(\Omega;\R^d) \quad \text{such that } u_n^0
\to u_0 \quad \text{in } H_{\mathrm{D}}^1 (\Omega;\R^d) \text{ as }
n \to \infty\,,
\end{equation}
where $W_{\mathrm{D}}^{1,\gamma}(\Omega;\R^d)=\{v\in W^{1,\gamma}(\Omega;\R^d)\colon v=0
\text{ on }\partial_\mathrm{D}\Omega \text{ in the trace sense}\}$.

For  the weak formulation of the discrete heat equation,
we  also need to introduce the function space appropriate for $\theta$,
dependent on a given $\bar{z} \in L^\infty(\Omega)$
\begin{equation*}
\Xteta_{\bar z}:= \big\{ \vartheta \in H^1 (\Omega)\colon
\int_\Omega \mathbb{K}(\bar z ,\vartheta)\nabla \vartheta \ps \nabla v \dd x
\; \text{is well defined for all } v \in H^1(\Omega) 
\big\}\,.
\end{equation*}
 In fact, the above space encodes   the sharpest property that we will be  able to obtain for our discrete solutions ${(u\kk,z\kk,\theta\kk)}_{k=1}^n$. This will  be proven by  approximating 
system \eqref{time-discrete-scheme} by truncations, so that in  the truncated system the heat equation is standardly weakly formulated in $ H^1(\Omega)^*$. Passing to the limit
as the truncation parameter tends to  infinity, with a careful comparison argument in the discrete heat equation   (cf.\ the proof of   \cite[Lemma 4.4]{RoRoESTC14} 
for all details),   it is possible to prove that $\theta \kk \in  \Xteta_{z \kk} $. 
\par
We consider the following weakly-coupled discretization scheme
(in fact, only the momentum and the heat equation are coupled,
while the discrete equation for $z$ is decoupled from them):
\begin{problem}
\label{prob-discrete}
Starting from
\[
u_n^0\,, 
\quad z_n^0:=z_0\,, \quad \w_n^0:=\theta_0\,,
\]
and setting $u_n^{-1}:=u_n^0 - {\taun} \dot u_0$, find
${(u\kk,z\kk,\theta\kk)}_{k=1}^n \subset
W^{1,\gamma}_{\mathrm{D}}(\Omega;\R^d) \times W^{1,q}(\Omega) \times
\Xteta_{z\kk}$ such that  the following hold:
\begin{subequations}
\label{time-discrete-scheme}
\begin{compactitem}
\item[-] {Minimality of }$z\kk$:
\begin{align}
\label{DSw1}
z\kk \in \argmin \left\{ \calR_1(z-z_n^{k-1}) + \E(t\kk,u\km,z) \colon
z\in \calZ \right\}\,;
\end{align}
\item[-] {Time-discrete weak formulation of the coupled momentum balance and the heat equation}:\\
Find $u\kk \in W^{1,\gamma}_{\mathrm{D}} (\Omega;\R^d)$
and $\theta\kk \in \Xteta_{z\kk}$ such that
\begin{align}
\label{DSw2}
&
\begin{aligned}
&
\rho\integ{\Om}{\tfrac{u\kk-2u\km+u\kmm}{\taun^2}\ps v}{x}\\
& \phantom{=} +\integ{\Om}{\left( \DD(z\km,\theta\km) \,e\left(\tfrac{u\kk-u\km}{\taun}\right)  
+ \CC(z\kk) e(u\kk)
-\theta\kk\,\BB +\taun |e(u\kk)|^{\gamma-2} e(u\kk) \right)\psm e(v)}{x}\\
&=
 \pairing{}{H^1_\mathrm{D}(\Omega;\R^d)}{\bigF\kk}{v}
\hspace*{19em} \text{for all } v \in W^{1,\gamma}_{\mathrm{D}} (\Omega;\R^d) \,,
\end{aligned}
\\[3mm]
\label{DSw3}
&
\begin{aligned}
&
\integ{\Om}{\tfrac{\w\kk-\w\km}{\taun}\testw}{x}+
\integ{\Om}{\ \KK(z\kk,\w\kk)\nabla \w\kk \ps\nabla\testw}{x} \\
&= \integ{\Om}{\tfrac{z\km-z\kk}{\taun}\,\testw}{x}
+\integ{\Om}{\left(\DD(z\km,\theta\km)\,e\left(\tfrac{u\kk-u\km}{\taun}\right)-\theta\kk\,\BB\right)\psm 
e\left(\tfrac{u\kk-u\km}{\taun}\right)
\testw}{x}
\\
&\phantom{=}
+ \integ{\partial\Om}{\hs\kk\testw}{\acca^{d-1}(x)} + \pairing{}{H^1(\Omega)}{\hv\kk}{\eta}
\hspace*{13em} \text{for all }\eta \in H^1(\Omega)\,.
\end{aligned}
\end{align}
\end{compactitem}
\end{subequations}
\end{problem}
\par
The above time-discrete problem has been carefully designed in such
a way as to be  weakly-coupled in that, for each
$k\in\{1,\ldots,n\}$, it can be solved successively starting from
\eqref{DSw1} and then solving the system \eqref{DSw2}--\eqref{DSw3}.
 See \cite[Remark 4.3]{RoRoESTC14} for similar ideas.
\par
Our existence result for Problem
\ref{prob-discrete} reads:
\begin{prop}
\label{prop:exist-discrete-problem}
Let the assumptions of Theorem~\ref{thm:main} hold true. Then 
there exists a solution
$${(u\kk,z\kk,\theta\kk)}_{k=1}^n \subset
W^{1,\gamma}_{\mathrm{D}}(\Omega;\R^d)\times W^{1,q}(\Omega) \times
H^1(\Omega)$$ 
to Problem \ref{prob-discrete},
%
 satisfying the following properties: 
There exists $\widetilde{\theta}>0$ such that
\begin{equation}
\label{tetak-pos}
\theta\kk \geq \widetilde{\theta}>0 \qquad \text{for all } k=1,\ldots,n\,, \quad \text{for all } n \in \N\,.
\end{equation}
Furthermore, if in addition \eqref{h-pos+} holds, then
\begin{equation}
\label{tetak-pos+h}
\theta\kk \geq \max\left\{ \widetilde{\theta}, \sqrt{H_*/\bar c} \right\}>0
\qquad \text{for all } k=1,\ldots,n\,, \quad \text{for all } n \in \N\,,
\end{equation}
with $H_*$ and $\bar c$ from \eqref{h-pos+}. 
\end{prop}
While the existence of solutions for \eqref{DSw1} follows from the
direct method of the calculus of variations in a straightforward
manner, the existence proof for system \eqref{DSw2}--\eqref{DSw3} is
more involved, due to the quasilinear character of the discrete heat
equation.  This is due to the fact that the viscous dissipation
$\DD(z\km,\theta\km)e\big(\tfrac{u\kk-u\km}{\tau_{n}}\big):e\big(\tfrac{u\kk-u\km}{\tau_{n}}\big)$
as well as the thermal stresses $\theta\km\,\BB \psm
e\big(\tfrac{u\kk-u\km}{\tau_{n}}\big)$ only happen to be of
 $L^1$-summability  as a consequence of \eqref{DSw2}. Observe in
particular that $\CC(z_n^k),\DD(z\km,\theta\km)\in
\big(L^\infty(\Omega)\cap W^{1,q}(\Omega)\big)^{d \times d \times d
\times d}$, and we do not impose the assumption $q>d$, which would
guarantee the continuity of the coefficients. As it is demonstrated
by the counterexample in \cite{NecSti76PTLE}, in absence of
continuous coefficients, it is not ensured that the solution of
\eqref{DSw2} enjoys elliptic regularity. Because of this expected
lack of additional regularity, the existence of solutions for the
coupled system \eqref{DSw2}--\eqref{DSw3} will be verified by means
of an approximation procedure, in which the $L^1$ right-hand side in
\eqref{DSw3} is replaced by a sequence of truncations. 
For this we proceed along the lines of \cite{RoRoESTC14} where the
analysis of a time-discrete system analogous to
\eqref{DSw1}--\eqref{DSw3} was carried out. The existence of
solutions to the approximate discrete system in turn follows from an
existence result for a wide class of elliptic equations, in the
framework of  the Leray-Schauder theory of pseudo-monotone
operators. We will then conclude the existence of solutions to
\eqref{DSw2}--\eqref{DSw3} by passing to the limit with the
truncation parameter. In such a step, we shall exploit the strict
positivity of the approximate discrete temperatures, cf.\
\eqref{tetakm-pos} below. This property and the convergence of the
approximate discrete temperatures clearly imply the strict
positivity \eqref{tetak-pos}. 
 Arguing directly on the non-truncated discrete heat equation,  we will also obtain 
the enhanced positivity property   \eqref{tetak-pos+h} which,
unlike \eqref{not-tunable}, in fact provides a tunable threshold
from below to the discrete temperatures.

In the forthcoming proof, we will  use that for any
\emph{convex} (differentiable) function $\psi: \R \to
(-\infty,+\infty]$
\begin{equation}
\label{cvx-ineq} \psi(x) -\psi(y) \leq \psi'(x) (x{-}y) \quad
\text{for all } x,y\in \dom(\psi)\,.
\end{equation}
\begin{proof}
{\bf Existence of a minimizer to \eqref{DSw1}: } We first verify the coercivity of the
functional
$z\mapsto \calE(t_n^k, u\km,z)+\calR_1(z-z_n^{k-1}) \colon W^{1,q}(\Omega)\to\R\cup\{\infty\}$,
where $\calR_1$ is the dissipation potential
\eqref{integrated-1-dissip}. Indeed, by the positivity of
$\calR_1(\cdot)$ and assumption \eqref{G-growth} on the density $G$
we have
\begin{equation*}
\calE(t_n^k,u\km,z)+\calR_1(z-z_n^{k-1})\geq\int_\Omega G(z,\nabla z)\,\mathrm{d}x -C
\geq C_G^1\|z\|_{W^{1,q}(\Omega)}^q-C_G^1\calL^d(\Omega)-C\,,
\end{equation*}
where we also used that $G(z(x),\nabla z(x))<\infty$ implies
$z(x)\in[0,1]$, cf.\ \eqref{Gind}. By the convexity   and the
continuity assumptions \eqref{Gcont}--\eqref{Gconv} on $G$ and by
the properties of $\calR_1$ we conclude that the functional
$$\calE(t_n^k,u\km,\cdot)+\calR_1((\cdot)-z_n^{k-1})\colon W^{1,q}(\Omega)\to\R\cup\{\infty\}$$
is weakly sequentially lower semicontinuous. Since
$\calZ=\{ z \in W^{1,q}(\Omega)\colon z \in [0,1] \ \aein \Omega \}$, see
\eqref{statesp}, is a closed subset of a reflexive Banach space, the
direct method of the calculus of variations ensures the existence of
a minimizer $z_n^k\in \calZ$.
\par
{\bf Existence of  an approximate  solution to system
\eqref{DSw2}--\eqref{DSw3}: }
As in  \cite[proof of Lemma 4.4]{RoRoESTC14}, we approximate
\eqref{DSw2}--\eqref{DSw3} by a suitable truncation of the heat
conductivity matrix $\KK$, in such a way as to reduce to an elliptic
operator with \emph{bounded} coefficients in the discrete heat
equation.  In a similar manner we treat the $L^1$ right-hand sides
in order to improve their integrability. Accordingly, we truncate
all occurrences of $\theta\kk$ in the respective terms of system
\eqref{DSw2}--\eqref{DSw3}.  We show that the approximate system
thus obtained admits solutions by resorting to an existence result
from the theory of elliptic systems featuring pseudo-monotone
operators drawn from \cite{Roub05NPDE}. Hence, we pass to the limit
with the truncation parameter and conclude the existence of
solutions to \eqref{DSw2}--\eqref{DSw3}.

Let   $z\kk$ be a  solution of  \eqref{DSw1}.
In what follows,  we  shall  denote by $\overline{\KK}= \overline{\KK}(x,\theta)$ the function
${\KK}(z\kk(x),\theta)$. Let $M>0$. We introduce   the truncation operator
\begin{equation*}
\mathcal{T}_M(\theta) :=  \begin{cases}
0 & \text{if }\theta<0,
\\
\theta & \text{if } 0 \leq \theta \leq M,
\\
M & \text{if } \theta>M,
\end{cases}
\end{equation*}
and we set 
\begin{equation*}
\overline{\KK}_M : \Omega \times \R \to \R^{d \times d}\,, \qquad  \overline{\KK}_M (x,\theta)
:=\overline{\KK}(x,\mathcal{T}_M(\theta)).
\end{equation*}
Since $\KK \in \rmC^0 (\R \times \R; \R^{d \times d})$ and  $0 \leq
z\kk (x) \leq 1 $ for almost all $x \in \Omega$, it is immediate to
check that there exists a positive constant $C_M$ such that
$|\overline{\KK}_M (x,\theta)| \leq C_M$ for almost all $x \in
\Omega$ and $\theta \in \R$. The truncated version of system
\eqref{DSw2}--\eqref{DSw3} thus reads:  find $(u,\theta) \in
W^{1,\gamma}_{\mathrm{D}}(\Omega;\R^d)\times H^1(\Omega)$ such that
\begin{subequations}
\label{system-truncated}
\begin{align}
\label{DSw2-trunc}
&
\begin{aligned}
&\rho\integ{\Om}{\tfrac{u-2u\km+u\kmm}{\taun^2}\ps v}{x}\\
&\phantom{=}+\integ{\Om}{\left(\DD(z\km,\theta\km)\,e\left(\tfrac{
u-u\km}{\taun}\right) + \CC(z\kk)e( u)
-\mathcal{T}_M(\theta)\,\BB +\taun |e( u)|^{\gamma-2} e( u) \right)\psm e(v)}{x}\\
&=
\pairing{}{H^1_\mathrm{D}(\Omega;\R^d)}{\bigF\kk}{v}
\hspace*{22em} \text{for all } v \in W^{1,\gamma}_{\mathrm{D}} (\Omega;\R^d) \,,
\end{aligned}
\\[2mm]
\label{DSw3-trunc}
&
 \begin{aligned}
&\integ{\Om}{\tfrac{\w-\w\km}{\taun}\testw}{x}+
\integ{\Om}{\overline\KK_M(x,\w)\nabla \w \ps\nabla\testw}{x} \\
&= \integ{\Om}{\tfrac{z\km-z\kk}{\taun}\,\testw}{x}
+\integ{\Om}{\left(\DD(z\km,\theta\km)\,e\left(\tfrac{
u-u\km}{\taun}\right) -\mathcal{T}_M(\theta)\,\BB \right)\psm e\left(\tfrac{ u-u\km}{\taun}\right)
\testw}{x}
\\
&\phantom{=}+  \integ{\partial\Om}{\hs\kk\testw}{\acca^{d-1}(x)}
+ \pairing{}{H^1(\Omega)}{\hv\kk}{\eta}
\hspace*{16em}
\text{for all }\eta \in H^1(\Omega)\,.
\end{aligned}
\end{align}
\end{subequations}
Observe that system
\eqref{system-truncated} rewrites as
\begin{subequations}
\label{more-explicit}
\begin{align}
\label{more-explicit1}
&
\begin{aligned}
&\rho\integ{\Om}{ u\ps v}{x}
\\
&\phantom{=}+\taun\!\integ{\Om}{\left(
\DD(z\km,\theta\km) e( u)+ \taun \CC(z\kk)e( u) -\taun\mathcal{T}_M(
\theta)\,\BB +\taun^2 |e( u)|^{\gamma-2} e( u) \right)\psm e(v)}{x}
\\
&=
\rho\integ{\Om}{(2u\km -u\kmm) \ps v}{x}
+\taun\!\integ{\Om}{ \DD(z\km,\theta\km)e(u\km) \psm e(v)  }{x}
+ \taun^2\pairing{}{H^1_\mathrm{D}(\Omega;\R^d)}{\bigF\kk}{v}
\\
&
\hspace*{30.8em}
\text{for all } v \in W^{1,\gamma}_{\mathrm{D}} (\Omega;\R^d) \,,
\end{aligned}
\\[2mm]
\label{more-explicit2}
&
\begin{aligned}
&\integ{\Om}{ \theta  \, \testw}{x}+ \taun
\integ{\Om}{ \overline\KK_M(x, \theta )\nabla  \theta
 \ps\nabla\testw}{x} -\tfrac{1}{\taun}
\integ{\Om}{\DD(z\km,\theta\km)e( u)\psm e( u) \testw}{x}
\\
&\phantom{=}
 +
\integ{\Om}{\mathcal{T}_M( \theta)\,\BB\psm e( u)\testw}{x}
+\tfrac2 \taun \! \integ{\Om}{\DD(z\km,\theta\km)e( u)\psm
e(u\km)\testw}{x}   -  \integ{\Om}{\mathcal{T}_M(
\theta)\,\BB\psm e(u\km)\testw}{x}
\\
&= \integ{\Om}{\w\km\testw}{x} +
\tfrac1\taun\!\integ{\Om}{\DD(z\km,\theta\km)e(u\km)\psm e(u\km)
\testw}{x}
 \\
&\phantom{=}+ \integ{\Om}{(z\km-z\kk)\testw}{x}
+\taun \! \integ{\partial\Om}{\hs\kk\testw}{\acca^{d-1}(x)}
+ \taun \pairing{}{H^1(\Omega)}{\hv\kk}{\eta}
\hspace*{4.8em} \text{for all }\eta \in H^1(\Omega)\,,
\end{aligned}
\end{align}
\end{subequations}
which in turn can be recast in the form
\begin{equation*}
\mathcal{A}_{k,M} ( u,\theta)= B_{k-1}\,.
\end{equation*}
Here, $\mathcal{A}_{k,M} : W^{1,\gamma}_{\mathrm{D}} (\Omega;\R^d)
\times H^1(\Omega) \to W^{1,\gamma}_{\mathrm{D}} (\Omega;\R^d)^*
\times H^1(\Omega)^*$ is the elliptic operator, acting on the
unknown $(u, \theta)$, defined by the left-hand sides of
\eqref{more-explicit1} and  \eqref{more-explicit2}, while $B_{k-1}$
is the vector defined by the right-hand side terms in system
\eqref{more-explicit}. It can be verified that $\mathcal{A}_{k,M}$
is a pseudo-monotone operator in the sense of \cite[Chapter II,
Definition 2.1]{Roub05NPDE}:  without entering into details, we may in fact observe that 
 $\mathcal{A}_{k,M}$ is given by the sum of either bounded, radially continuous, monotone operators, 
  or totally continuous operators, cf.\  \cite[Chapter II,
Definition 2.3, Lemma 2.9, Cor.\ 2.12]{Roub05NPDE}. 
 Furthermore,
 crucially exploiting the presence of the regularizing  term $-\taun\div(|e(u)|^{\gamma-2} e(u))$, with $\gamma>4$,
  in the discrete momentum balance, 
 we may show that 
  $\mathcal{A}_{k,M}$ is
coercive on $ W^{1,\gamma}_{\mathrm{D}} (\Omega;\R^d) \times
H^1(\Omega)$. This can be checked directly on system
\eqref{more-explicit}, testing \eqref{more-explicit1} by $ u$ and
\eqref{more-explicit2} by $ \theta$ and adding the resulting
equations: it is then sufficient to deduce from these calculations
an estimate for $\| u\|_{W_\mathrm{D}^{1,\gamma} (\Omega;\R^d) }$
and $\| \theta \|_{H^1(\Omega)}$. We refer to \cite[proof of Lemma
4.4]{RoRoESTC14} for all the detailed calculations, 
 which show that, since $\gamma>4$,  the term
$-\taun\div(|e(u)|^{\gamma-2} e(u))$
can absorb
 the
quadratic terms in $e( u)$ on the right-hand side of
\eqref{DSw3-trunc}. In this way, it is possible to
carry out the test of \eqref{more-explicit2} by $ \theta$ and 
   obtain the bound for $\|
\theta\|_{H^1(\Omega)}$: for this, one  also exploits   that the operator with
coefficients $\overline\KK_M$ is uniformly elliptic thanks to
\eqref{ass-K-b}. Since $\mathcal{A}_{k,M}$ is pseudo-monotone and
coercive, we are in a position to apply \cite[Chapter II, Theorem
2.6]{Roub05NPDE} to system \eqref{more-explicit},  for every
$M\in\N$  thus deducing the existence of a solution
$(u,\theta)$  which shall be hereafter denoted as
$(u\kkm,\theta\kkm)$.
\par
 {\bf Positivity of $\theta\kkm$:}  First of all, we show that $\theta \kkm \geq 0$ a.e.\ in $\Omega$.  To this end, we test the 
(approximate) discrete heat equation \eqref{DSw3-trunc} by $-(\theta\kkm)^- = \min\{ \theta \kkm, 0\}$. We thus obtain 
\[
\begin{aligned}
&
\integ{\Om}{\tfrac{1}{\tau_n} |(\theta\kkm)^-|^2}{x}  +  \integ{\Om}{\tfrac{1}{\tau_n}  \theta \km  (\theta\kkm)^-}{x} 
+ 
\integ{\Om}{\overline\KK_M(x,\theta \kkm)\nabla( \theta\kkm)^-  \ps \nabla( \theta\kkm)^-   }{x} \\
&= -  \integ{\Om}{\tfrac{z\km-z\kk}{\taun}\,\theta\kkm}{x}
- \integ{\Om}{\DD(z\km,\theta\km)\,e\left(\tfrac{
u-u\km}{\taun}\right) \psm e\left(\tfrac{ u-u\km}{\taun}\right) \theta\kkm}{x}
\\
&
\phantom{=}
- \integ{\Om}{\mathcal{T}_M(\theta)\,\mathbb{B} \psm e\left(\tfrac{ u-u\km}{\taun}\right) \theta\kkm}{x} 
+  \integ{\partial\Om}{\hs\kk\theta\kkm }{\acca^{d-1}(x)}
+ \pairing{}{H^1(\Omega)}{\hv\kk}{\theta\kkm }
\end{aligned}
\]
Now, the second term on the left-hand side is non-negative, since we may suppose, by induction,  that  $ \theta \km  \geq 0$ a.e.\ in $\Omega$
(in fact, for $k=0$
the \emph{strict} positivity \eqref{tetak-pos} holds with $\widetilde\theta= \theta_*$, thanks to \eqref{wzero}). The third term is also non-negative, by ellipticity of $\overline\KK_M$. As for the right-hand side,  the first, second, fourth, and fifth terms are negative, 
since $z \km \geq z\kk$ a.e.\ in $\Omega$, and by the positivity properties of the data $\mathbb{D}$, $H$, and $h$. The very definition of the truncation operator $\mathcal{T}_M$ does ensure that 
the third term  is  null. 
All in all, we conclude that 
$\integ{\Om}{|(\theta\kkm)^-|^2}{x} \leq 0$, whence $(\theta\kkm)^- = 0 $ a.e.\ in $\Omega$, i.e.\  the desired positivity.  
Let us now prove that $\theta\kkm$ fulfills  \eqref{tetak-pos}, namely
\begin{equation}
\label{teta-kk-m-pos}
\theta\kkm \geq \widetilde{\theta}>0 \quad \aein\, \Omega. 
\end{equation} 
 Following the lines  of \cite[proof of Lemma 4.4]{RoRoESTC14}
we develop a comparison argument drawn from \cite{FePeRo09ESPT}.
In this context, we will use the following estimate
\begin{equation}
\label{pos1}
\begin{split}
&\DD(\bar z, \bar \theta)\bar e\psm\bar e -\mathcal{T}_M( \bar\theta)\,\BB \psm\bar
e  \geq C_\mathbb{D}^1 |\bar e|^2 -|\bar e|C_\BB |\bar\theta| \geq
\tfrac{C_\mathbb{D}^1}{2}|\bar e|^2-\tfrac{ {(C_\BB)^2}
}{2C_\mathbb{D}^1}|\bar\theta|^2\,.
\end{split}
\end{equation}
Exploiting \eqref{pos1} and also using
 that $z_{k-1} \geq z_k$ a.e.\ in $\Omega$, 
 the positivity \eqref{ass-dataHeat} of the data $\hv$ and $\hs$
and of $\theta\km$, we deduce from \eqref{DSw3} that
$\theta\kkm$ fulfills
\begin{equation}
\label{crucial-for-pos}
\int_\Omega \theta\kkm \eta \dd x
+ \taun\!\int_\Omega
\overline\KK_M(z_n^k,\theta\kkm) \nabla \theta\kkm \ps \nabla \eta \dd x
\geq \int_\Omega \theta\km \eta \dd x
-\taun \bar{c} \int_\Omega \left(\theta\kkm\right)^2 \eta  \dd x \qquad
\end{equation}
for all   $\eta \in H^1(\Omega) \cap  L^\infty(\Omega) $   with  $\eta \geq 0$  a.e.\ in  $\Omega$,
with the constant $\bar{c} = \tfrac{ {(C_\BB)^2}}{2C_\mathbb{D}^1}$ independent of $k$.
Hence, we compare $\theta\kkm$ with the solution $v_k\in\R$ of the finite difference equation
\begin{equation}
\label{diffeq} v_k= v_{k-1} -\taun \bar{c} \,v_k^2, \quad
k=1,\ldots,n,  \qquad \text{with } v_0:= \theta_*>0\,.
\end{equation} 
Now, it is possible to show that
 \begin{equation}
\label{not-tunable} v_k \geq \widetilde\theta:= \left(\bar c T+
\frac1{\theta_*}\right)^{-1}.
\end{equation}
We test the difference of
\eqref{crucial-for-pos} and \eqref{diffeq} by the function $L_\varepsilon(v_k {-}\theta\kkm)$, with 
\[
L_\varepsilon(x) :=  \begin{cases}
0 & \text{if }x\leq 0,
\\
\tfrac x\eps & \text{if } 0 < x<\eps,
\\
1 & \text{if } x\geq \eps,
\end{cases}
\]
and 
we conclude that 
\begin{equation}
\label{crucial-H-eps}
\integ{\Om}{(v_k{-}v_{k-1}) {-} (\theta\kkm {-} \theta\km) H_\eps(v_k{-} \theta\kkm )}{x}  = \taun \bar{c}  \integ{\Om}{\left( \left(\theta\kkm\right)^2{-} v_k^2\right) H_\eps(v_k{-} \theta\kkm )}{x}    \leq 0\,.
\end{equation}
Observe that, in order to conclude that the above integral is negative, it was essential to preliminarily show that $\theta\kkm \geq 0$ a.e.\ in $\Omega$. 
Assume now that $\theta\km \geq v_{k-1}$ (which is true for $k=0$, cf.\ \eqref{wzero}). 
Letting $\eps \downarrow  0$ in \eqref{crucial-H-eps} yields that $\theta\kkm \geq v_k$ a.e.\ in $\Omega$. Hence, in view of 
\eqref{not-tunable} we conclude  the desired \eqref{teta-kk-m-pos}.  
\par
{\bf  Passage to the limit as $M{\to}\infty$: } 
We now consider a family ${(u\kkm,\theta\kkm)}_M$ of solutions to
the truncated system \eqref{system-truncated}: we shall derive some
a priori estimates on ${(u\kkm,\theta\kkm)}_M$ which will allow us
to extract a (not relabeled) subsequence converging as
$M{\to}\infty$ to a solution of system  \eqref{DSw2}--\eqref{DSw3}.
For the ensuing calculations, it is crucial to observe that
\begin{equation}
\label{tetakm-pos}
\exists\,\widetilde{\theta} \quad \text{such that}\quad
\theta\kkm \geq \widetilde{\theta}>0 \quad \text{for all } M>0\,.
\end{equation}
This follows from the very same arguments as for \eqref{tetak-pos}:
indeed, notice that $\widetilde{\theta}$ does not depend on $M$.

Hence, let us first  test \eqref{DSw2-trunc} by $(u\kkm {-}
u\km)/\taun$, \eqref{DSw3-trunc} by $1$, and add the resulting
relations. Taking into account the cancelation of the coupling terms
between \eqref{DSw2-trunc} and \eqref{DSw3-trunc},  by convexity,
cf.\ \eqref{cvx-ineq},  we obtain
\[
\begin{aligned}
& \tfrac{\rho}{2\taun^3}\int_\Omega |u\kkm - u\km|^2 \dd x
+\tfrac1{2\taun}\int_\Omega \CC( z\kk) e(u\kkm) \psm e(u\kkm) \dd x
+\tfrac1{\gamma} \int_\Omega |e(u\kkm)|^{\gamma} \dd x +\tfrac1{\taun}
\int_\Omega \theta \kkm \dd x
\\
&
 \leq \tfrac{\rho}{2\taun^3}\int_\Omega
|u\km - u\kmm|^2 \dd x +\tfrac1{2\taun}\int_\Omega \CC( z\kk)e(u\km)\psm
e(u\km) \dd x +\tfrac1{\gamma} \int_\Omega |e(u\km)|^{\gamma} \dd x
+\tfrac1{\taun} \int_\Omega \theta \km \dd x
\\
& \phantom{\leq}
 + \pairing{}{H^1_\mathrm{D}(\Omega;\R^d)}{\bigF\kk}{\tfrac{u\kkm-u\km}\taun}
  + \int_{\Omega} \left(
\tfrac{z\km-z\kk}{\taun} + \hv\kk \right)\dd x  + \int_{\partial\Omega}
\hs\kk \dd \mathcal{H}^{d-1}(x) \leq   C_{k,n}\,, 
\end{aligned}
\]
where the constant $ C_{k,n} $ is uniform with respect to the truncation
parameter $M$
 (but depends on $k$ and $n$).
Therefore, also on account of \eqref{tetakm-pos} we
infer that
\begin{equation}
\label{estimate-1} \| u\kkm \|_{W^{1,\gamma} (\Omega;\R^d) } + \|
\theta \kkm\|_{L^1(\Omega)} \leq   C_{k,n}\,, 
\end{equation}
 for a (possibly different) constant  $ C_{k,n}$ uniform w.r.t.\ $M$ but depending on $k$ and $n$.
From now till the end of the discussion  of the limit passage $M\to\infty$, we will omit 
the dependence of such constants on
$k$ and $n$.  
 As a  straightforward consequence of \eqref{estimate-1}, if we define
\[
\mathcal{S}_M= \{ x\in \Omega \, : \ \theta\kkm \leq M \}\,,
\]
using Markov's inequality, it is not difficult to infer from
\eqref{estimate-1} that
\begin{equation}
\label{trunca-set}
 |\Omega\backslash \mathcal{S}_M|
 \to 0
\quad \text{as } M \to \infty\,.
\end{equation}

Secondly, we test \eqref{DSw3-trunc} by $\mathcal{T}_M(\theta \kkm)$.
Using that
\begin{equation*}
\theta \, \mathcal{T}_M(\theta) \geq |\mathcal{T}_M(\theta)|^2
\quad\text{and}\quad
\overline{\KK}_M(x,\theta)\nabla \theta \cdot \nabla
\mathcal{T}_M(\theta) =
\overline{\KK}(x,\mathcal{T}_M(\theta))\nabla
\mathcal{T}_M(\theta)\cdot \nabla \mathcal{T}_M(\theta),
\end{equation*}
we obtain
\begin{equation}
\begin{split}
\label{to-close}
&\tfrac1{2\taun} \int_\Omega
|\mathcal{T}_M(\theta\kkm)|^2 \dd x +\int_\Omega
\overline{\KK}(x,\mathcal{T}_M(\theta\kkm)) \nabla
\mathcal{T}_M(\theta\kkm) \ps \nabla  \mathcal{T}_M(\theta\kkm) \dd x \\
&\leq \tfrac1{2\taun} \int_\Omega |\theta\km|^2 \dd x +
I_1+I_2+I_3+I_4\,,
\end{split}
\end{equation}
where, taking into account \eqref{assCD-3} and the previously
obtained \eqref{estimate-1}, we have
\begin{align*} \displaybreak[0]
&
\begin{aligned}
I_1 :=& \left| \int_{\Omega} \DD(z\km,\theta\km)e
\Big( \tfrac{u\kkm -u\km}{\taun}\Big)\psm
e\Big( \tfrac{u\kkm- u\km}{\taun}\Big) \mathcal{T}_M(\theta\kkm)\dd x  \right|
\\
  \leq&\ C \left\| e\Big( \tfrac{u\kkm -u\km}{\taun}\Big)\right\|_{L^4(\Omega;\R^{d \times d})}^4
+\tfrac1{8\taun} \int_\Omega |\mathcal{T}_M(\theta\kkm)|^2 \dd x\,,
\end{aligned}
\\[2mm] \displaybreak[0]
&
\begin{aligned}
 I_2  :=& \left|  \int_{\Omega} \mathcal{T}_M(\theta \kkm) \mathbb{B}:
e\Big( \tfrac{u\kkm - u\km}{\taun}\Big)  \mathcal{T}_M(\theta \kkm)\dd x  \right| \\
 \leq&\ C  \left\| e\Big( \tfrac{u\kkm -u\km}{\taun}\Big)\right\|_{L^2(\Omega;\R^{d \times d})}
\|\mathcal{T}_M(\theta\kkm) \|_{L^4(\Omega)}^2 
\\
  \leq& C    \|\mathcal{T}_M(\theta\kkm) \|_{L^4(\Omega)}^2  \leq \tfrac{c_1}4 \int_\Omega
|\nabla \mathcal{T}_M(\theta\kkm) |^2 \dd x
 +  
\|\mathcal{T}_M(\theta\kkm)\|_{L^1(\Omega)}^2\,,
\end{aligned}
\\[2mm] \displaybreak[0]
&
\begin{aligned}
I_3 :=& \left|  \int_{\Omega} \tfrac{z\kk - z\km}{\taun}
\mathcal{T}_M(\theta\kkm) \dd x \right| \leq C + \tfrac1{8\taun}
\int_\Omega |\mathcal{T}_M(\theta\kkm)|^2 \dd x\,,
\end{aligned}
\\[2mm] \displaybreak[0]
&
\begin{aligned}
I_4   :=& \left|
\pairing{}{H^1(\Omega)}{\hv\kk}{\mathcal{T}_M(\theta\kkm) } +
\int_{\partial\Omega} \hs\kk \, \mathcal{T}_M(\theta \kkm)  \dd
\mathcal{H}^{d-1}(x)\right| \\  \leq&\ \tfrac1{16\taun} \int_\Omega
|\mathcal{T}_M(\theta\kkm)|^2 \dd x + \tfrac{c_1}2 \int_\Omega
|\nabla \mathcal{T}_M(\theta\kkm)|^2 \dd x +C\,.
\end{aligned}
\end{align*}
 where in the estimate for $I_2$ we have used the
previously obtained bound \eqref{estimate-1}, the
Gagliardo-Nirenberg inequality
$\|v\|_{L^4{(\Omega)}} \leq C \| v\|_{H^1(\Omega)}^{\sigma} \| v\|_{L^1(\Omega)}^{1-\sigma} $
for $\sigma=9/10$,
 and the Young
inequality.
 As by \eqref{ass-K-b} it is
$\overline{\KK}_M\xi\cdot\xi \geq c_1|\xi|^2$, combining the above
estimates with \eqref{to-close}  and taking into account
\eqref{estimate-1},  we conclude that
\begin{equation*}
\|\mathcal{T}_M( \theta\kkm) \|_{L^2(\Omega) }
+\int_\Omega \overline{\KK}(x,\mathcal{T}_M( \theta\kkm)) \nabla
\mathcal{T}_M( \theta\kkm) \ps \nabla \mathcal{T}_M( \theta\kkm) \dd
x  \leq C\,.
\end{equation*}
Now, the  coercivity \eqref{ass-K-b} implies
\[
\begin{aligned}
& \int_\Omega \overline{\KK}(x,\mathcal{T}_M( \theta\kkm)) \nabla
\mathcal{T}_M( \theta\kkm) \ps \nabla \mathcal{T}_M( \theta\kkm) \dd
x \\ & \geq c_1 \int_{\Omega } |\mathcal{T}_M( \theta\kkm)|^{\kappa}
|\nabla \mathcal{T}_M( \theta\kkm)|^2 \dd x = c \int_\Omega |\nabla
(\mathcal{T}_M( \theta\kkm))^{(\kappa+2)/2}|^2 \dd x\,.
\end{aligned}
\]
From this, recalling the continuous embedding $H^1{\subset}L^6$
we infer
\begin{equation}
\label{estimate-3} \| \mathcal{T}_M(\theta\kkm) \|_{H^1(\Omega) } +
\| \mathcal{T}_M(\theta\kkm) \|_{L^{3\kappa+6}(\Omega) } \leq C\,.
\end{equation}

Thirdly, we test \eqref{DSw3-trunc} by $\theta \kkm $. 
Relying on estimate \eqref{estimate-3} to bound the second term on the right-hand side of
\eqref{DSw3-trunc} and mimicking the
above calculations, we obtain
\begin{equation}
\label{estimate-4} \| \theta\kkm \|_{H^{1}(\Omega)} + \| \theta\kkm
\|_{L^{3\kappa+6}(\mathcal{S}_M) }\leq C\,.
\end{equation}
With estimates \eqref{estimate-1}, \eqref{estimate-3}, and
\eqref{estimate-4}, combined with  well-known compactness arguments, we find a
pair $(u,\theta)$  such that, along a not relabeled subsequence,   
$(u\kkm,\theta\kkm)\rightharpoonup (u,\theta)$ in
$W^{1,\gamma}_\mathrm{D}(\Omega;\R^d)\times H^1(\Omega)$.
The argument for passing to the limit as $M{\to}\infty$ in \eqref{system-truncated},  
also based on \eqref{trunca-set},
is completely analogous to the one developed in the proof of \cite[Lemma 4.4]{RoRoESTC14},
therefore we refer to the latter paper for all details.
\par
 {\bf Positivity of the discrete temperature, ad \eqref{tetak-pos}: }  The strict positivity \eqref{tetak-pos} is now inherited by $\theta\kk$ in the  limit passage, as $M\to\infty$, in \eqref{tetak-pos}. 
\par
{\bf Refined positivity estimate for the discrete temperature, ad \eqref{tetak-pos+h}: }
Under the additional strict positivity
\eqref{h-pos+} of $\hv$, arguing as in the above lines we infer that
 $\theta\kk$ fulfills
\begin{equation*}
\int_\Omega \theta\kk \eta \dd x +
\taun\!\int_\Omega\KK(z_n^k,\theta\kk) \nabla \theta\kk \ps \nabla
\eta \dd x
\geq \int_\Omega \theta\km \eta \dd x
+\int_\Omega \taun
\big(H_* -\bar{c}\left(\theta\kk\right)^2\big) \eta \dd x
\end{equation*}
for all $\eta \in L^\infty(\Omega)$  with  $\eta \geq 0$  a.e.\ in
$\Omega$, with  $\bar{c}>0$ the same constant as in
\eqref{crucial-for-pos}. Hence,
we compare $\theta\kk$ with the solution $\widetilde{v}_k\in\R$
\begin{equation}
\label{diffeq-tilde}
\widetilde{v}_k= \widetilde{v}_{k-1} + \taun(H_* -\bar{c}\,\widetilde{v}_k^2)\,, \quad
k=1,\ldots,n\,, \quad \text{with }\  \widetilde{v}_0:= \max\left\{ \theta_*,  \sqrt{H_*/\bar{c}}\right\}>0\,,
\end{equation}
The very same arguments from \cite[proof of Lemma 4.4]{RoRoESTC14},   cf.\ also the previous discussion,   allow us to show
for all $k=0,\ldots, n $ that
$\theta \kk(x) \geq \widetilde{v}_k$ for almost all $x\in\Omega$.
Since $\widetilde{v}_k > \widetilde{v}_{k-1} - \taun \bar{c}\,\widetilde{v}_k^2$,
and $\widetilde v_0\geq v_0=\theta_*$,
a comparison with the solution $v_k$ of the finite-difference
equation \eqref{diffeq} and induction over $k$ yield that
$\widetilde{v}_k \geq v_k$.
Hence $\tilde{v}_k \geq \widetilde\theta>0$.
We now aim to  prove  that
\begin{equation}
\label{to-show-tilde}
\widetilde{v}_k \geq \sqrt{H_*/\bar c} \qquad \text{for all } k=1,\ldots, n\,.
\end{equation}
We proceed by contradiction and suppose that $ H_* > \bar{c}\, \widetilde{v}_{\bar k}^2$
for a certain $\bar k \in \{ 1, \ldots, n\}$. Then, we read from
\eqref{diffeq-tilde}  that
$\widetilde{v}_{\bar k} > \widetilde{v}_{\bar k-1}$.
Since $\widetilde v_{\bar k-1}>0$, we then conclude that
$ H_* >\bar{c}\, \widetilde{v}_{\bar{k}}^2> \bar{c}\, \widetilde{v}_{\bar{k}-1}^2$.
Proceeding by induction, we thus conclude that
$H_* > \bar{c}\,\widetilde{v}_0^2 $,
which is a contradiction to \eqref{diffeq-tilde}.
Therefore, \eqref{to-show-tilde} ensues.
This concludes the existence proof for system \eqref{DSw2}--\eqref{DSw3}.
\end{proof}
%
%
\subsection{Time-discrete version of the energetic formulation}
\label{ss:4.2}
%
We now define the approximate solutions to the energetic
formulation of the initial-boundary value problem for
system  \eqref{ourPDE} by suitably interpolating
the discrete solutions ${(u\kk,z\kk,\theta\kk)}_{k=1}^n$ from Proposition~\ref{prop:exist-discrete-problem}.
Namely, for $t\in (t\km,t\kk]$, $k=1,\dots,n$, we set
\begin{subequations}
\label{interpolants}
\begin{alignat}{3}
& \baru(t):=u\kk\,,\ && \barw(t):=\w\kk\,,\ && \barz(t):=z\kk \,,\\
& \underu (t) := u \km \,,\quad && \underth(t):=\theta\km \,,\quad && \underz(t):=z\km \,,
\end{alignat}
and we also consider the piecewise linear interpolants, defined by
\begin{equation}
u_n(t):=\tfrac{t-t\km}{\taun} u\kk + \tfrac{t\kk-t}{\taun} u\km \,, \;\;
z_n(t):=\tfrac{t-t\km}{\taun} z\kk + \tfrac{t\kk-t}{\taun} z\km  \,,\;\;
\theta_n (t) :=    \tfrac{t-t\km}{\taun} \theta\kk + \tfrac{t\kk-t}{\taun} \theta\km.\quad
\end{equation}
In what follows, we shall understand the time derivative of the piecewise linear
interpolant $u_n$ to be defined  also  at the nodes of the partition by
\begin{equation}
\dot{u}_n(t\kk):= \tfrac{u\kk -u\km}{\taun}\,,\quad
\quad \text{for } k=1,\ldots, n\,.
\end{equation}
 This will allow us, for instance, to state \eqref{discrete-solution} for all $t\in[0,T]$.
\end{subequations}
We also introduce the piecewise constant and linear interpolants of
the discrete data ${(\bigF\kk,\hv\kk,\hs\kk)}_{k=1}^n$ in~\eqref{local-means}
by setting for $t\in (t\km,t\kk]$
\[
\barbigF(t):=\bigF\kk\,, \qquad  \barH(t):=\hv\kk \,,\qquad   \barh(t):=\hs\kk \,,
\]
and $\bigF_n(t):=\tfrac{t-t\km}{\taun} \bigF\kk +
\tfrac{t\kk-t}{\taun} \bigF\km $  with time derivative
$\dot\bigF_n(t):=\frac{\bigF\kk-\bigF\km}{\taun}$.  It follows from~\eqref{ass-data}
that, as $n{\to}\infty$,
\begin{subequations}
\begin{alignat}{4}
\label{ass-forces-n}
\barbigF &\to \bigF \quad &&\text{in } L^p (0,T; H^1_\mathrm{D}(\Omega;\R^d)^*) \text{ for all } 1 \leq p<\infty\,, \ \ 
&\barbigF \wtos \bigF \quad &\text{in } L^\infty (0,T; H^1_\mathrm{D}(\Omega;\R^d)^*)\,,
\\
\label{ass-forces-n-bis}
\barbigF(t) &\to \bigF(t) \quad &&\text{in } H^1_\mathrm{D}(\Omega;\R^d)^* \text{ for all } t\in [0,T]\,, & &
  \\
\label{ass-forces-n-dot}
\bigF_n &\weakto \bigF \quad &&\text{in } H^1 (0,T; H^1_\mathrm{D}(\Omega;\R^d)^*)\,, \\
\label{ass-dataHeat-n}
\barH&\to \hv \quad &&\text{in } L^1(0,T; L^1(\Omega)) \cap L^2(0,T; H^1(\Omega)^*)\,, \quad
&\barh \to \hs  \quad &\text{in } L^1(0,T; L^2(\partial\Omega))\,.
\end{alignat}
\end{subequations}
Finally, we consider the piecewise constant interpolants associated with the partition, i.e.,
$$
\bartau(t):=t\kk \quad \text{and} \quad \undertau(t):=t\km \quad \text{for } t\in (t\km,t\kk] \,.
$$

In Proposition \ref{EFdiscr} we show that the approximate solutions
introduced above indeed fulfill the discrete version of the energetic formulation
from Definition \ref{def4}.
In order to check the discrete momentum equation \eqref{disc-momentum} and \eqref{disc-heat},
we shall make use of the following \emph{discrete by-part integration} formula, for every
${(r_k )}_{k=1}^{n} \subset X$ and ${( s_k  )}_{k=1}^{n}
\subset X^*$, with $X$ a given Banach space:
\begin{equation}
\label{discrete-by-part}
\sum_{k=1}^{n}  \pairing{}{X} { s_k}{r_k-r_{k-1}}
= \pairing{}{X}{s_{n}}{r_n} -\pairing{}{X}{ { s_0} }{r_0}
-\sum_{k=1}^{n} \pairing{}{X}{s_k -s_{k-1}}{r_{k-1}}.\\
\end{equation}
  In the discrete mechanical  energy inequality
\eqref{disc-mech-energy-ineq} below,  the mechanical energy
$\mathcal{E}$ will be  replaced by
\begin{equation}
\label{discr-mech-energ} \mathcal{E}_n (t,u,z):=
\int_\Omega \left(\tfrac{1}{2}\CC(z)e(u):e(u) +
\tfrac{\taun}{\gamma}|e(u)|^{\gamma} \right) \dd x +\calG(z,\nabla z) -
\pairing{}{H^1_\mathrm{D}(\Omega;\R^d)}{\barbigF(t)}{u}  \; \text{
with } \taun =\tfrac{T}{n}\,.
\end{equation}
\begin{prop}[Time-discrete version of  the  energetic formulation
\eqref{energetic-formulation}  \& total energy inequality]
\label{EFdiscr}
Let the assumptions of Theorem~\ref{thm:main} hold
true. Then the interpolants of the time-discrete solutions
$(\baru,\underu,u_n,\barz,\underz,  z_n,  \barth,\underth,\theta_n)$
obtained via Problem \ref{prob-discrete} and \eqref{interpolants} 
satisfy the following properties:
\begin{subequations}\label{discrete-solution}
\begin{compactitem}
\item
unidirectionality: \; for a.a.\ $x\in\Om$, the functions
$\barz(\cdot,x)\colon[0,T]\to[0,1]$ are nonincreasing;
\item
discrete semistability: for all $t\in [0,T]$
\begin{equation}
\label{disc-semistab} \forall\, \tilde{z}\in\calZ \colon\quad
\E_n(t,\underu(t),\barz(t))\le\E_n(t,\underu(t),\tilde z)
+\calR_1(\tilde z-\barz(t))\,;
\end{equation}
\item
discrete formulation of the momentum equation:
 for all $t\in[0,T]$ and
for every $(n+1)$-tuple ${(v\kk)}_{k=0,\dots,n}\subset
W^{1,\gamma}_{\mathrm{D}}(\Omega;\R^d)$, setting $\barv(s):=v\kk$
and $v_n(s):=\tfrac{s-t\km}{\taun} v\kk + \tfrac{t\kk-s}{\taun}
v\km$ for $s\in (t\km,t\kk]$,
\begin{equation}
\label{disc-momentum}
\begin{aligned}
&\rho\integ{\Om}{\left(\dot u_n(t)\ps \barv(t)-\dot u_0\ps
v_n(0)\right)}{x} -\rho\integlin{0}{\bartau(t)}{\integ{\Om}{\dot u_n(s{-}\taun)\ps \dot v_n(s)}{x}}{s}
\\
&\phantom{=}+\integlin{0}{\bartau(t)}{\!\integ{\Om}{\left( \DD(\underz,\underth) e(\dot u_n)+
\CC(\barz) e(\baru)
-\barw\,\BB +\taun |e(\baru)|^{\gamma-2} e(\baru) \right)\psm e(\barv)}{x}}{s}\\
&= \integlin{0}{\bartau(t)}{\pairing{}{H^1_\mathrm{D}(\Omega;\R^d)}{\barbigF}{\barv}}{s}
\,,
\end{aligned}
\end{equation}
where
we have extended $u_n$ to $(-\taun,0]$ by setting $u_n(t):= u_n^0 + t \dot u_0$;
\item
discrete mechanical energy inequality: for all $t \in[0,T]$
\begin{equation}
\label{disc-mech-energy-ineq}
\begin{aligned}
&\tfrac\rho2{\integ{\Om}{\mod{\dot u_n(t)}^2}{x}}+\E_n (t,\baru(t),\barz(t))
+ \integ{\Om}{(z_0{-}\barz(t))}{x}\\
&\phantom{=}+\integlin{0}{\bartau(t)}{\!\integ{\Om}{\left(\DD(\underz,\underth)e(\dot u_n)
{-}\barw\,\BB
\right)\psm e(\dot u_n)}{x}}{s}\\
&\le \tfrac\rho2{\integ{\Om}{\mod{\dot u_0}^2}{x}}+
\E_n(0,u_n^0,z_0)
- \integlin{0}{\bartau(t)}
{\pairing{}{H^1_\mathrm{D}(\Omega;\R^d)}{\dot{\bigF}_n}{\underu}}{s}\,; 
\end{aligned}
\end{equation}
 \item
discrete total energy inequality: for all  $t \in[0,T]$
\begin{equation}
\label{disc-tot-energy-ineq}
\begin{aligned}
 &\tfrac\rho2{\integ{\Om}{\mod{\dot u_n(t)}^2}{x}}
+\E_n(t,\baru(t),\barz(t))
+\integ{\Om}{\barw(t)}{x}
\\
&\le\tfrac\rho2{\integ{\Om}{\mod{\dot u_0}^2}{x}}+
\E_n(0,u_n^0,z_0) +\integ{\Om}{\w_0}{x}
\\
&\phantom{=}  - \integlin{0}{\bartau(t)}
{\pairing{}{H^1_\mathrm{D}(\Omega;\R^d)}{\dot{\bigF}_n}{\underu}}{s}   
+\integlin{0}{\bartau(t)}{ \! \left[
\integ{\partial\Om}{\barh}{\acca^{d-1}(x)} +\integ{\Om}{\barH}{x}
\right] }{s} \,;
\end{aligned}
\end{equation}
\item
discrete formulation of the heat equation:
 for all $t\in[0,T]$ and
for every $(n+1)$-tuple ${( \eta \kk)}_{k=0}^n \subset H^1(\Omega)$, setting
$\bareta(s):=\eta\kk$ and
$\eta_n(s):=\tfrac{s-t\km}{\taun} \eta\kk + \tfrac{t\kk-s}{\taun} \eta\km$
for $s\in (t\km,t\kk]$,
\begin{equation}
\label{disc-heat}
\begin{aligned}
&\integ{\Om}{\barw(t) \bareta(t)}{x} -
\integ{\Om}{\theta_0\eta_n(0)}{x} -
\integlin{0}{\bartau(t)}{\!\integ{\Om}{\underth(s)\dot
\eta_n(s)}{x}}{s}
\\
& \phantom{=}
+
\integlin{0}{\bartau(t)}{\!\integ{\Om}{\left( \KK(\barz,\barw)\nabla \barw \right)\ps\nabla\bareta}{x} }{s}\\
&=
\integlin{0}{\bartau(t)}{\!\integ{\Om}{\bareta\,|\dot z_n |}{x}}{s}
\integlin{0}{\bartau(t)}{\!\integ{\Om}{\left(\DD(\underz,\underth)e(\dot
u_n) - \barw\,\BB \right) \psm e(\dot u_n) \,\bareta}{x} }{s}
\\
& \phantom{=}
+  \integlin{0}{\bartau(t)}{ \!\left[ \integ{\partial\Om}{\barh\,\eta_n}{\acca^{d-1}(x)}
+ \pairing{}{H^1(\Om)}{\barH}{\eta_n} \right] }{s}\,.
\end{aligned}
\end{equation}
\end{compactitem}
\end{subequations}
\end{prop}
%
%
%
%
\begin{proof}
The discrete momentum and heat equations   \eqref{disc-momentum} and \eqref{disc-heat}
follow from testing \eqref{DSw2} and \eqref{DSw3}
by the discrete test functions ${(v\kk)}_{k=0}^n\subset W^{1,\gamma}_{\mathrm{D}}(\Omega;\R^d)$
and ${( \eta \kk)}_{k=0}^n \subset H^1(\Omega)$, respectively,
and applying the discrete by-part integration formula \eqref{discrete-by-part}. From
the discrete minimum problem \eqref{DSw1} we infer
\[
\E(t\kk,u\km,z\kk) \leq  \E(t\kk,u\km,\tilde z )
+\int_{\Omega}(z\km - \tilde z) \dd x
-  \int_{\Omega}(z\km - z\kk) \dd x
\leq  \E(t\kk,u\km,\tilde z )  +\int_{\Omega}(z\kk - \tilde z) \dd x
\]
for all  $ \tilde z \in \mathcal{Z}$  with  $ \tilde z \leq z\km$.
By \eqref{DSw1} and the definition of the dissipation $\calR_1$ we
have $z\kk\le z\km$, whence  the unidirectionality and
the discrete semistability \eqref{disc-semistab} hold.

To deduce the mechanical energy inequality
\eqref{disc-mech-energy-ineq} we choose  $z\km$ as  a competitor in
\eqref{DSw1} and get
\begin{equation}
\label{times-zdot}
\begin{aligned}
& \integ{\Om}{(z\km-z\kk)}{x} + \integ{\Om}{\left(\tfrac12  \CC(z\kk)e(u\km)\psm e(u\km)
+G(z\kk,\nabla z\kk)\right)}{x} \\
&\le
\integ{\Om}{\left( \tfrac12 \CC(z\km)e(u\km):e(u\km)+G(z\km,\nabla z\km)\right)}{x} \,.
\end{aligned}
\end{equation}
Moreover, we  test  \eqref{DSw2} by $v=u\kk-u\km$.  To this aim, we
observe that by convexity \eqref{cvx-ineq}
\begin{subequations}
\begin{align}
&
\rho \int_\Omega \tfrac{u\kk-2u\km+u\kmm}{\taun^2} \ps (u\kk{-}u\km) \dd x
\geq\ \rho\integ{\Om}{\!\!\left( \tfrac12\tfrac{\mod{u\kk-u\km}^2}{\taun^2}
- \tfrac12\tfrac{\mod{u\km-u\kmm}^2}{\taun^2}\right)}{x}\,,\label{inertia}
\\
&
\label{due-convexity-1}
\int_\Om \CC(z\kk) e(u\kk)\,{:}\,(e(u\kk){-}e(u\km))\dd x
\geq
 \int_\Om  \tfrac12\Big(
\CC(z\kk)e(u\kk) \,{:}\, e(u\kk) \,{-}\, \CC(z\kk)e(u\km)\,{:}\,e(u\km)
\Big) \dd x\,,
\\
\label{due-convexity-2}
&
\integ{\Om}{
\taun |e(u\kk)|^{\gamma-2} e(u\kk) \,{:}\, (e(u\kk){-}e(u\km)) }{x}
\geq \integ{\Om} {\Big(
\tfrac{\taun}{\gamma} |e(u\kk)|^{\gamma} - \tfrac{\taun}{\gamma} |e(u\km)|^{\gamma}
\Big) }{x} \,.
\end{align}
\end{subequations}
 Further,  let $t \in (0,T]$ be fixed, and let $1\leq  j\leq n $
fulfill $t\in (t_n^{j-1}, t_n^j]$. We  sum
\eqref{inertia}--\eqref{due-convexity-2} over the index $k=1,
\ldots, j$. Applying the by-part integration formula
\eqref{discrete-by-part}  we  conclude that
\begin{equation}
\label{discrete-by-part-F}
\begin{aligned}
&
\sum_{k=1}^{j} \pairing{}{H^1_\mathrm{D}(\Omega;\R^d)}{\bigF\kk}{u\kk-u\km} =
\integlin{0}{\bartau(t)} {\pairing{}{H^1_\mathrm{D}(\Omega;\R^d)}{\barbigF}{\dot{u}_n}}{s}
\\ &=
\pairing{}{H^1_\mathrm{D}(\Omega;\R^d)}{\barbigF(t)}{\baru(t)}
- \pairing{}{H^1_\mathrm{D}(\Omega;\R^d)}{\bigF(0)}{u_0} 
- \integlin{0}{\bartau(t)}
{\pairing{}{H^1_\mathrm{D}(\Omega;\R^d)}{\dot{\bigF}_n}{\underu}}{s}\,. 
\end{aligned}
\end{equation}
All in all we infer
\[
\begin{aligned}
&
\tfrac{\rho}2\integ{\Om}{|\dot{u}_n(t)|^2}{x}
+\int_0^{\bartau(t)}
\integ{\Om}{\left( \DD(\underz,\underth) \,e\left(\dot{u}_n\right){-}\underth\,\BB    \right)
\psm e(\dot{u}_n)} {x} \dd s
\\
& \phantom{=} + \integ{\Om} {\tfrac12 \CC(\barz(t))e(\baru(t)) \psm
e(\baru(t))  }{x} +\integ{\Om} {\tfrac{\taun}{\gamma}
|e(\baru(t))|^{\gamma} }{x}
-\pairing{}{H^1_\mathrm{D}(\Om;\R^d)}{\barbigF(t)}{\baru(t)}
\\
&
\leq
\tfrac{\rho}2\integ{\Om}{|\dot{u}_0|^2}{x}
+\integ{\Om} {\tfrac{\tau_{n}}{\gamma} |e(u_0)|^{\gamma} }{x}
 - \pairing{}{H^1_\mathrm{D}(\Omega;\R^d)}{\bigF(0)}{u_0}
 - \integlin{0}{\bartau(t)}
{\pairing{}{H^1_\mathrm{D}(\Omega;\R^d)}{\dot{\bigF}_n}{\underu}}{s} 
 \\
 & \phantom{=}
+
\sum_{k=1}^j \integ{\Om} {\tfrac12 \CC(z\kk)e(u\km)  \psm e(u\km)  }{x}\,.
\end{aligned}
\]
We add the above inequality to \eqref{times-zdot}, summed over $k=1,
\ldots, j$. Observing the cancelation of the term $\sum_{k=1}^j
\integ{\Om} {\tfrac12 \CC(z\kk)e(u\km)\psm e(u\km)  }{x}$, we
conclude \eqref{disc-mech-energy-ineq}.

Finally, the discrete total energy inequality ensues from adding the
discrete mechanical energy inequality \eqref{disc-mech-energy-ineq}
with the discrete heat equation \eqref{DSw3}, tested for
$\testw=\tau_{n}$ and added up over $k=1, \ldots, j$. We observe the
cancelation of some terms, and readily conclude
\eqref{disc-tot-energy-ineq}.
\end{proof}
%
\subsection{A priori estimates}
%
The following result collects a series of a priori estimates on
the approximate solutions, uniform with respect to $n\in \N$.
Let us mention in advance that, in its proof
we will
start from the discrete total energy inequality
\eqref{disc-tot-energy-ineq} and derive estimates
\eqref{apDu}, \eqref{apDu-gamma}, \eqref{ap-u-lin-linfty},
\eqref{ap-z-w1q}, for $\baru, \,
\dot{u}_n, \,   \barz$, as well as estimate \eqref{aptheta1} below
for $\|\barth\|_{L^\infty(0,T;L^1(\Omega))}$. The next crucial step
will be to obtain a bound for the $L^2(0,T;H^1(\Omega))$-norm of
$\barth$. For this, we will make use of a technique developed in
\cite{FePeRo09ESPT}, cf.\ also \cite{RoRoESTC14}. Namely, we will test
the discrete heat equation \eqref{DSw3} by $(\theta\kk)^{\alpha-1}$,
with $\alpha \in (0,1)$. Exploiting the concavity of the function
$F(\theta) =\theta^{\alpha}/\alpha$, we will deduce that
$$
\int_0^{T}\!\!\!\!  \int_\Omega \KK \big(\barz,\bartheta\big)\nabla \big(\bartheta^{\,\alpha/2}\big)
\ps \nabla \big(\bartheta^{\,\alpha/2} \big)  \dd x \dd t
+\int_\Omega\tfrac{\theta_0^\alpha}\alpha \dd x
\leq \int_\Omega {\tfrac{\bartheta^\alpha( T)}{\alpha} }\dd x +
C  \int_0^{T}\!\!\!\!
\int_\Omega \bartheta^{\,\alpha+1} (t) \dd x \dd t\,,
$$
where the positive  and quadratic terms on the right-hand side of
\eqref{DSw3} have been confined to the \emph{left-hand side} and
thus can be neglected. Hence, relying on the growth \eqref{ass-K-b}
of $\KK$, we  will  end up with an estimate for $\barth^{\,\alpha/2}$ in
$L^2(0,T;H^1(\Omega))$, from which we  will    ultimately infer the desired
bound \eqref{apthetaH12}, whence \eqref{apthetaLp} by interpolation.
We  will be then in a position to exploit the mechanical energy
inequality in order to recover the \emph{dissipative} estimate
\eqref{ap-u-lin-inter}.
Estimate \eqref{apthetaBV}   will  finally ensue from a comparison in
\eqref{DSw3}.

In the following proof we will also use the concave counterpart to
inequality \eqref{cvx-ineq}, namely that
for any \emph{concave} (differentiable) function $\psi: \R \to (-\infty,+\infty]$
\begin{equation}
\label{concave-ineq}
\psi(x) -\psi(y) \leq \psi'(y) (x{-}y) \quad \text{for all } x,y\in \dom(\psi)\,.
\end{equation}

\begin{prop}[A priori estimates]
\label{Apriori}
Let the assumptions of Theorem~\ref{thm:main} hold true and
consider a sequence
${(\baru,\underu,u_n,\barz,\underz,\barth,\underth,\theta_n)}_n$
complying with Proposition \ref{EFdiscr}. Then there exists a
constant $C>0$ such that the following estimates hold uniformly with
respect to $n\in \N$:
\begin{subequations}
\label{apriori}
\begin{align}
\label{apDu}
\|\baru\|_{L^\infty(0,T;H_{\mathrm{D}}^1 (\Omega;\R^d))}&\leq C\,,\\ \displaybreak[0]
\label{apDu-gamma}
\taun^{1/\gamma}\|\baru\|_{L^\infty(0,T;W^{1,\gamma}_\mathrm{D} (\Omega;\R^d))}&\leq C\,,\\ \displaybreak[0]
\label{ap-u-lin-inter}
\|u_n\|_{H^1(0,T;H_{\mathrm{D}}^1 (\Omega;\R^d))}&\leq C\,,\\ \displaybreak[0]
\label{ap-u-lin-linfty}
\|\dot{u}_n \|_{L^\infty (0,T; L^2(\Omega;\R^d))} &\leq C\,,\\ \displaybreak[0]
\label{apddu}
 \|\dot{u}_n \|_{\BV([0,T];W^{1,\gamma}_\mathrm{D}(\Omega;\R^d)^*)}&\leq  C\,,\\ \displaybreak[0]
\label{apR1}
\mathcal{R}_1(\barz(T)-z_0)&\leq C\,, \\ \displaybreak[0]
\label{apzLinfty}
 \| \barz \|_{L^\infty ((0,T) \times \Omega)  } & \leq 1
 \,,   \\ \displaybreak[0]
\label{ap-z-w1q}
 \| \barz \|_{L^\infty (0,T; W^{1,q}(\Omega)) } &  \leq C 
\,,\\ \displaybreak[0]
\label{aptheta1}
\big\|\barth\big\|_{L^\infty(0,T;L^1(\Omega))}&\leq  C\,, \\ \displaybreak[0]
\label{apthetaH12}
\big\|\barth \big\|_{L^2(0,T;H^1(\Omega))}&\leq  C\,, \\ \displaybreak[0]
\label{apthetaLp}
\big\|\barth \big\|_{L^p((0,T)\times\Omega)}&\leq C
\quad\text{for any }p\in\left\{
\begin{array}{ll}
[1,8/3]&\text{if }d{=}3\,,\\
{[1,3]}&\text{if }d{=}2\,,
\end{array}
\right.
\\ \displaybreak[0]
\label{apthetaBV} \big\|\barth
\big\|_{\BV([0,T]; W^{1,\infty}(\Omega){^*})}&\leq C 
\,,
\end{align}
\end{subequations}
where $\mathcal{R}_1$ is from \eqref{integrated-1-dissip}.
\end{prop}
 Observe that estimate \eqref{ap-u-lin-inter}  implies \eqref{apDu}, and that \eqref{apthetaLp} is a consequence  of \eqref{aptheta1} and 
\eqref{apthetaH12}. Nonetheless, we have chosen to highlight  \eqref{apDu} and   \eqref{apthetaLp}  for ease of exposition, both in the proof of
Prop.\ \ref{Apriori} and  for the compactness arguments of Prop.\ \ref{Conv}. 
\begin{proof}
Estimate \eqref{apR1} follows from \eqref{Gind}, \eqref{hyp-init},
the definition of $\mathcal{R}_1$, and the monotonicity of $\barz$
and $\underz$. We divide the proof of the other estimates in
subsequent steps.
\par
{\bf First a priori estimates, ad \eqref{apDu}, \eqref{apDu-gamma}, \eqref{ap-u-lin-linfty}, \eqref{apzLinfty},
\eqref{ap-z-w1q}, \eqref{aptheta1}: }
We start from the discrete total energy inequality \eqref{disc-tot-energy-ineq}.
For its left-hand side, we
observe that
the first and the third term are nonnegative. For the second one, we use that,
in view of \eqref{conti-c}, \eqref{G-growth}, and
\eqref{ass-forces}, we have
\begin{equation}
\label{est-mechen}
\begin{aligned}
{\E_n(t,\baru(t),\barz(t)) }
 & \geq C_{\CC}^1 \int_\Omega |e(\baru(t))|^2 \dd x
+ C^1_G \int_\Omega |\nabla \barz(t)|^q \dd x
+\tfrac\taun\gamma \int_\Omega |e(\baru(t))|^\gamma \dd x
\\ & \quad
 - \big\|
\barbigF\big\|_{L^\infty (0,T;H^1_\mathrm{D}(\Omega;\R^d)^*)}
\| \baru(t)\|_{H^1_\mathrm{D}(\Omega;\R^d)} -C
\\
& \geq C \left( \| \baru(t) \|_{H^1_\mathrm{D}(\Omega;\R^d)}^2 + \taun \| \baru(t)
\|_{W^{1,\gamma}_\mathrm{D}(\Omega;\R^d)}^\gamma
+ \|    \barz(t) \|_{W^{1,q}(\Omega)}^q
\right) - C\,,
\end{aligned}
\end{equation}
for almost all $t \in (0,T)$,
where we have also used Poincar\'e's and  Korn's inequalities. Concerning the right-hand
side of \eqref{disc-tot-energy-ineq},   we use that
$| \partial_t \E_n (t,\underu(t),\underz(t))  |
\leq \| \dot{\bigF}_n \|_{H^1_\mathrm{D}(\Omega;\R^d)^*}
\| \underu(t)\|_{H^1_\mathrm{D}(\Omega;\R^d)}$
for almost all $t \in (0,T)$.
The remaining terms on the
right-hand side  are bounded, uniformly with respect to $n\in \N$,
in view of the properties of the initial and given data \eqref{assu-init}
and \eqref{tilde-uzero}, and of \eqref{ass-dataHeat-n}.
All in all,  from  \eqref{disc-tot-energy-ineq} we deduce
\[
C  \| \baru(t) \|_{H^1_\mathrm{D}(\Omega;\R^d)}^2
\leq C+ \tfrac12 \int_0^{\bartau(t)}  \| \underu(s)\|_{H^1_\mathrm{D}(\Omega;\R^d)}^2 \dd s
+  \tfrac12 \int_0^{\bartau(t)} \big\| \dot{\bigF}_n \big\|_{H^1_\mathrm{D}(\Omega;\R^d)^*}^2 \dd s\,.
\]
Also in view of the bounds on $\dot f_n$ by \eqref{ass-forces-n-dot},
estimate \eqref{apDu} then follows from the Gronwall Lemma. As a by-product, we conclude that
\begin{equation}
\label{by-product}
\int_0^{\bartau(t)}| \partial_t \E_n (s,\underu(s),\underz(s))  | \dd s
\leq C \int_0^{\bartau(t)} \big\| \dot{\bigF}_n (s) \big\|_{H^1_\mathrm{D}(\Omega;\R^d)^*} \dd s
\leq C\,.
\end{equation}
Inserting this into \eqref{disc-tot-energy-ineq} we also  infer
estimates \eqref{ap-u-lin-linfty},  \eqref{aptheta1}, and that
$|\E_n (t,\baru(t),\barz(t))| \leq C$ for a constant independent of
$n\in \N$ and $t\in (0,T)$. This implies \eqref{apDu-gamma} and the
first estimate in \eqref{ap-z-w1q}  via \eqref{est-mechen}. Then the
second estimate in \eqref{ap-z-w1q} immediately follows from the
very definition of the interpolants \eqref{interpolants}. Moreover,
\eqref{apzLinfty} is a direct consequence of the boundedness of the energy,
which implies $\barz ,\underz \in[0,1]$ a.e.\ in $\Omega$, for a.e.\
$t\in(0,T)$.
\par
{\bf Second a priori estimate: }
We fix $\alpha\in(0,1)$. Exploiting that
$\theta_n^k\geq\widetilde{\theta}>0$,
we may test \eqref{DSw3} by $(\theta_n^k)^{\alpha-1}$, thus obtaining
\begin{equation}
\label{tricky}
\begin{aligned}
&
\tfrac{4(1-\alpha)}{\alpha^2} \int_\Omega  \KK(z_n^k,\theta_n^k) \nabla (\theta\kk)^{\alpha/2}
\ps \nabla (\theta\kk)^{\alpha/2}  \dd x +
\int_\Omega \DD(z_n^k)e\big(\tfrac{u_n^k-u_n^{k-1}}{\tau}\big)
\psm e\big(\tfrac{u_n^k-u_n^{k-1}}{\tau}\big) (\theta_n^k)^{\alpha-1}
\dd x
\\ & \phantom{=}
+\int_\Omega \tfrac{z_n^{k-1}-z\kk}{\tau} (\theta_n^k)^{\alpha-1} \,\mathrm{d}x
+ \pairing{}{H^1(\Om)}{\hv\kk}{(\theta_n^k)^{\alpha-1}}
+\int_{\partial\Omega}\hs\kk(\theta_n^k)^{\alpha-1} \,\mathrm{d}\mathcal{H}^{d-1}
\\
&
=
\int_\Omega\tfrac{\theta_n^k-\theta_n^{k-1}}{\tau}(\theta_n^k)^{\alpha-1}\,\mathrm{d}x
+ \int_\Omega \theta_n^{k} \,\BB  \psm
e\big(\tfrac{u_n^k-u_n^{k-1}}{\tau}\big) (\theta_n^k)^{\alpha-1}\,\mathrm{d}x
\doteq I_1 +I_2\,,
\end{aligned}
\end{equation}
where we used that
\[
\KK(z_n^k,\theta_n^k) \nabla\theta_n^k \ps
\nabla(\theta_n^k)^{\alpha-1} =
(\alpha-1)(\theta_n^k)^{\alpha-2}\KK(z_n^k,\theta_n^k) \nabla\theta_n^k \ps
\nabla\theta_n^k
=
\tfrac{4(\alpha-1)}{\alpha^2}
\KK(z_n^k,\theta_n^k) \nabla (\theta_n^k)^{\alpha/2}\ps\nabla (\theta_n^k)^{\alpha/2}
\]
 and moved the term $\int_\Omega \KK(z_n^k,\theta_n^k)
\nabla\theta_n^k    \nabla(\theta_n^k)^{\alpha-1} \dd x  $ to the
opposite side. It follows from \eqref{concave-ineq} with $\psi(x) :
= \tfrac{x^\alpha}{\alpha}$ that
\[
I_1 \leq  \int_\Omega  \psi(\theta \kk ) \dd x -  \int_\Omega  \psi(\theta \km) \dd x\,,
\]
whereas we estimate $I_2$ by
\[
I_2 \leq \tfrac{C_{\DD}^1}2
\int_\Omega \left| e\big(\tfrac{u_n^k-u_n^{k-1}}{\tau}\big) \right|^2(\theta_n^k)^{\alpha-1}  \dd x
+C \int_\Omega |\theta_n^{k}|^2 (\theta_n^k)^{\alpha-1}  \dd x  \doteq I_3+I_4\,, 
\]
where $C_{\DD}^1$ from \eqref{assCD-3} is such that
$\displaystyle\int_\Omega \DD(z_n^k) e\big(\tfrac{u_n^k-u_n^{k-1}}{\tau}\big)
\psm e\big(\tfrac{u_n^k-u_n^{k-1}}{\tau}\big)
(\theta_n^k)^{\alpha-1} \dd x$ on the left-hand side of
\eqref{tricky} is bounded from below by $\displaystyle C_{\DD}^1 \int_\Omega
\left| e\big(\tfrac{u_n^k-u_n^{k-1}}{\tau}\big) \right|^2
(\theta_n^k)^{\alpha-1}  \dd x $,  which in turn dominates
$I_3$.  Taking into account that the second,  the third and the
fourth integrals on the left-hand side of \eqref{tricky} are
nonnegative also thanks to \eqref{ass-dataHeat} and summing up over
the index $k$, we end up with
\begin{equation}
\label{tricky-bis}
\tfrac{4(1-\alpha)}{\alpha^2}\!
 \int_0^{\bartau(t)}\!\!\!\!  \int_\Omega\!\! \KK \big(\barz,\bartheta\big)\nabla \big(\bartheta^{\alpha/2}\big)
\ps \nabla \big(\bartheta^{\alpha/2}\big)   \dd x \dd s
+\int_\Omega\!\!\tfrac{\theta_0^\alpha}\alpha \dd x
\leq \int_\Omega\!\! \tfrac{\bartheta(t)^{\alpha}}\alpha \dd x +  C\!
\int_0^{\bartau(t)}\!\!\!\!
\int_\Omega\!\! \bartheta(t)^{\alpha+1}  \dd x \dd s\,.
\end{equation}
Since $\alpha \in (0,1)$  and $\theta^k_n \geq \widetilde \theta>0$,  we have 
\[
 \int_\Omega \tfrac{\bartheta(t)^\alpha}\alpha \dd x
\leq \tfrac1{\alpha} \int_\Omega  \bartheta(t) \dd x + C
\leq C\,,
\]
  where the latter estimate follows  by \eqref{aptheta1}.
 From  \eqref{ass-K-b} we deduce that
 \begin{equation}
\label{additional-info}
\begin{aligned}
& \int_0^{\bartau(t)}  \!\!\! \int_\Omega \KK \big(\barz,\bartheta\big)
\nabla  \big(\bartheta^{\alpha/2}\big) \ps  \nabla \big(\bartheta^{\alpha/2}\big)
\dd x \dd s \geq c_1 \int_0^{\bartau(t)} \!\!\! \int_\Omega
\big(\bartheta\big)^\kappa  |\nabla \big(\bartheta^{\alpha/2}\big)|^2 \dd x \dd s
\\
& = { C} \int_0^{\bartau(t)} \!\!\!  \int_\Omega \big|\big(\bartheta\big)^{\kappa+\alpha
- 2}\big| \big|\nabla  \bartheta\big|^2   \dd x \dd s = { C}
\int_0^{\bartau(t)} \!\!\! \int_\Omega  |\nabla
\big(\bartheta^{(\kappa+\alpha)/2}\big)|^2   \dd x \dd s \,.
\end{aligned}
\end{equation}
In order to clarify the estimate for the second term on the right-hand
side of \eqref{tricky-bis},
we now use the placeholder 
\begin{equation*}
 w_n:= (\bartheta)^{(\kappa+\alpha)/2}\,,
\end{equation*}
 so that
$(\bartheta)^{\alpha+1}  = (w_n)^{2(\alpha+1)/(\alpha+\kappa)}$.  
Hence, neglecting the (positive) second term on the left-hand side of
\eqref{tricky-bis}, we infer
\begin{equation}
\label{calc2.2}
 \int_0^{\bartau(t)} \!\!\! \int_\Omega  |\nabla w_n|^2   \dd x \dd s  \leq C  + C  \int_0^{\bartau(t)} \!\!\!
\int_\Omega | w_n|^\om \dd x \dd s \qquad \text{with }  \om = 2\tfrac{\alpha+1}{\alpha+\kappa}  \,.
\end{equation}
We now proceed exactly in the same way as in  \cite{FePeRo09ESPT}, cf.\ also
\cite{RoRoESTC14}. Namely, the Gagliardo-Nirenberg inequality  for $d{=}3$
(for $d{=}2$ even better estimates hold true) yields
\begin{equation*}
\| w_n \|_{L^\om(\Omega)} \leq C  \| \nabla w_n\|_{L^2(\Omega;\R^d)}^\sigma
\| w_n \|_{L^r(\Omega)}^{1-\sigma} + C'   \| w_n \|_{L^r(\Omega)}
\end{equation*}
for suitable constants $C$
and $C'$, and  for
 $ 1 \leq  r  \leq  \om $ and $\sigma $ satisfying $1/\om= \sigma/6
+ (1-\sigma)/r$. Hence $\sigma= 6 (\om-r)/\om (6-r)$. Observe that
$\sigma \in (0,1)$  since $\om = 2(\alpha+1)/(\alpha+\kappa)<6$, 
which is satisfied because $\kappa>1$.  
Hence we transfer the Gagliardo-Nirenberg estimate into \eqref{calc2.2} and use
Young's inequality in the estimate of the term
 \[
 C \int_0^{\bartau(t)} \| \nabla w_n\|_{L^2(\Omega;\R^d)}^{\om\sigma}
\| w_n \|_{L^r(\Omega)}^{\om(1-\sigma)} \dd s
\leq  \tfrac{1}2  \int_0^{\bartau(t)} \| \nabla w_n\|_{L^2(\Omega;\R^d)}^2 \dd s
+  C'  \int_0^{\bartau(t)}
\| w_n  \|_{L^r(\Omega)}^{2\om(1-\sigma)/(2-\om\sigma)}\dd s\,.
\]
 In the previous inequality we have used the fact that $\omega\sigma<2$, which holds since $\omega<2$ and $\sigma<1$ by \eqref{calc2.2}.
The term $\tfrac{1}2 \int_0^{\bartau(t)} \| \nabla w_n\|_{L^2(\Omega;\R^d)}^2 \dd s$
may be absorbed into the 
left-hand side of \eqref{calc2.2}.
All in all, we conclude
\begin{equation}
\label{calc2.3}
 \int_0^{\bartau(t)} \!\!\! \int_\Omega  |\nabla w_n|^2   \dd x \dd s  \leq
 C   + C  \int_0^{\bartau(t)}  \| w_n  \|_{L^r(\Omega)}^{2\om
(1-\sigma)/(2-\om\sigma)}\dd s +  C'  \int_0^{\bartau(t)}
\| w_n\|_{L^r(\Omega)}^\om  \dd s \,.
\end{equation}
Now,  let us choose 
 \begin{equation*}
 1\leq r \leq  2/(\alpha{+}\kappa).
 \end{equation*}
 Then, we have
for almost all $t\in (0,T)$ that
\begin{equation}
\label{calc2.3-2}
\| w_n(t) \|_{L^r(\Omega)} = \left(\int_\Omega     \big(\bartheta(t)\big)^{r(\kappa+\alpha)/2}
\dd x\right)^{1/r}  =  \left(\int_\Omega \bartheta(t) \dd x\right)^{1/r} \leq  C\,,
\end{equation}
for a constant independent of $t$, where again we have used estimate  \eqref{aptheta1}. 
 Observe that, 
  since we have previously imposed $\kappa +\alpha-2 \geq 0$, we ultimately find that \eqref{calc2.3-2}  must hold for $r=1$ and that, moreover,
 $\alpha=2-\kappa\in(2-\kappa_d,1)$, 
 with $\kappa_d=5/3$ if $d{=}3$ and
$\kappa_d=2$ if $d{=}2$, so that
$ w_n = \bartheta$. 
 From \eqref{calc2.3}--\eqref{calc2.3-2}  we then infer 
\begin{equation}\label{171207}
\int_0^{\bartau(t)} \int_\Omega \big|\nabla \bartheta\big|^2 \dd x \dd s
\leq C\,.
\end{equation}
\par
{\bf Third a priori estimate, ad \eqref{apthetaH12} and  \eqref{apthetaLp}: }
 From \eqref{171207} we deduce  \eqref{apthetaH12}     in view of the previously obtained
\eqref{aptheta1} via Poincar\'e's inequality. Estimate \eqref{apthetaLp}  ensues
by interpolation between $L^2(0,T;H^1(\Omega))$
and $L^\infty(0,T;L^1(\Omega))$, relying on \eqref{apthetaH12} and \eqref{aptheta1}
 and exploiting the Gagliardo-Nirenberg inequality.
For later convenience, let us
also point out that,  we indeed recover the following bound
\begin{equation}
\label{est-alpha-later}\big\| (\bartheta)^{(\kappa+\alpha)/2}
\big\|_{L^2(0,T;H^1(\Omega))} \leq C 
\end{equation}
for arbitrary $\alpha \in (0,1)$.
For this, it is sufficient to observe that second term on the right-hand side of 
\eqref{tricky-bis} now fulfills 
$
\int_0^{\bartau(t)}
\int_\Omega \bartheta(t)^{\alpha+1}  \dd x \dd s \leq C$ thanks to estimate \eqref{apthetaLp}. 
Then,  by 
\eqref{additional-info} we find that 
$\int_0^{\bartau(t)}  \int_\Omega  |\nabla
\big(\bartheta^{(\kappa+\alpha)/2}\big)|^2   \dd x \dd s \leq C$, whence \eqref{est-alpha-later} via Poincar\'e's inequality. 
\par
{\bf Fourth a priori estimate, ad \eqref{ap-u-lin-inter}  and \eqref{apddu}: } 
From the discrete mechanical energy inequality \eqref{disc-mech-energy-ineq} we infer
\begin{equation}
\label{calc-4.1}
\begin{aligned}
C_{\DD}^1\integlin{0}{\bartau(t)}{ \!\!\! \integ{\Om}{|e(\dot u_n)|^2}{x}}{s}
\leq C + \integlin{0}{\bartau(t)}{ \!\!\! \integ{\Om} {\barw \,\BB  \psm e(\dot u_n)}{x} }{s}
\end{aligned}
\end{equation}
where we have used \eqref{est-mechen}, \eqref{by-product}, and the fact that the terms
$\int_\Omega |\dot{u}_0|^2 \dd x$ and $\E(0, u_n^0, z_0)$ are bounded, uniformly with respect to $n\in \N$,
in view of \eqref{hyp-init}
and \eqref{tilde-uzero}.
Exploiting the previously obtained estimate~\eqref{apthetaH12} we find
\[
\begin{aligned}
\integlin{0}{\bartau(t)}{ \!\!\! \integ{\Om} {\barw \,\BB  \psm
e(\dot u_n)}{x} }{t}  & \leq
\tfrac{C_{\DD}^1}2\integlin{0}{\bartau(t)}  {\integ{\Om}{|e(\dot
u_n)|^2}{x}}{t}
+ C \int_0^{\bartau(t)}  \!\!\! \int_\Omega |\barw|^2 \dd x \dd s \\
& \leq \tfrac{C_{\DD}^1}2\integlin{0}{\bartau(t)}{ \!\!\! \integ{\Om}{|e(\dot u_n)|^2}{x}}{t}
+ C\,.
\end{aligned}
\]
Inserting this into  \eqref{calc-4.1} we conclude \eqref{ap-u-lin-inter}
via Korn's inequality, again exploiting the definition of the
interpolants~\eqref{interpolants}. 
Finally, estimate  \eqref{apddu} ensues from a comparison argument in \eqref{DSw2},  
taking into account the previously proven \eqref{apDu-gamma},
\eqref{ap-u-lin-inter}, \eqref{apthetaH12}, as well as \eqref{ass-forces-n}.
\par
{\bf Fifth a priori estimate, ad \eqref{apthetaBV}: } Let $\kappa$
be as in \eqref{ass-K}. In \eqref{DSw3} we use a test function
 $\eta\in W^{1,\infty}(\Omega)$, 
thus we  find
\begin{equation}
\label{rhs-side}
\left| \int_\Omega\tfrac{\theta_n^k-\theta_n^{k-1}}{\taun}\,\eta\,\mathrm{d}x \right|
\leq\left|\int_\Omega\KK(z_n^k,\theta_n^k)\nabla\theta_n^k\ps\nabla\eta\,\mathrm{d}x\right|
+ \left|\pairing{}{ W^{1,\infty}(\Omega)}{\mathrm{RHS}_n^k}{\eta}\right|\,,
\end{equation}
where the terms on the right-hand side of \eqref{DSw3} are summarized in $\mathrm{RHS}_n^k$.
It follows from assumptions \eqref{ass-CD} and \eqref{ass-dataHeat} that
\begin{equation}\label{est-5.1}
\begin{aligned}
&
\left|\pairing{}{ W^{1,\infty}(\Omega)}{\mathrm{RHS}_n^k}{\eta}\right|
\\
& \leq { C} \left( \left\|
e\left(\tfrac{u_n^k-u_n^{k-1}}{\taun}\right)\right\|_{L^2(\Omega;\R^{d \times
d})}^2 + \| \theta\kk\|_{L^2(\Omega)}^2 + \left\| \tfrac{z\kk -
z\km}{\taun}\right\|_{L^1(\Omega)} + \|\hs\kk\|_{L^2(\partial\Omega)} + \| \hv\kk\|_{L^1(\Omega)}
 \right)
 \|\eta\|_{L^\infty(\Omega)}
\\
&
\doteq \Lambda \kk  \|\eta\|_{L^\infty(\Omega)}\,.
\end{aligned}
\end{equation}
Furthermore,  with \eqref{ass-K} we find for every $\alpha\in(1/2,1)$
\begin{align}
\nonumber
& \left|\int_\Omega\KK(z_n^k,\theta_n^k)\nabla\theta_n^k \ps
\nabla\eta\,\mathrm{d}x\right|
\\
\label{est-5.2}
&\leq \|\nabla\eta\|_{L^\infty(\Omega;\R^d)}c_2
\|((\theta_n^k)^\kappa+1)\nabla\theta_n^k\|_{L^1(\Omega;\R^d)}\\
\nonumber
&\leq \|\nabla\eta\|_{L^\infty(\Omega;\R^d)} c_2\left( \|
(\theta_n^k)^{(\kappa -\alpha+2)/2} \|_{L^2(\Omega)}
\|(\theta_n^k)^{(\kappa +\alpha-2)/2}\nabla\theta_n^k\|_{L^2(\Omega;\R^d)}
+\calL^d(\Omega)^{1/2}
\|\nabla\theta_n^k\|_{L^2(\Omega;\R^d)} \right)\,.
\end{align}
Inserting \eqref{est-5.1} and \eqref{est-5.2} into \eqref{rhs-side}
and summing over the index $k=1,\ldots,n$, we find for every
\emph{time-dependent} function $\eta\in
\rmC^0([0,T]; W^{1,\infty}(\Omega))$ that
\begin{align}
\nonumber
& \left| \int_0^{\bartau(t)} \int_\Omega \dot{\theta}_n \, \eta \dd x
\dd s \right|
\\
\nonumber
& \leq C \|\nabla\eta\|_{L^\infty((0,T)\times\Omega;\R^d)}
\left(\big\|\barth\big\|_{L^{\kappa-\alpha+2}((0,T)\times\Omega)}^{(\kappa-\alpha+2)/2}\|
\big(\barth\big)^{(\kappa+\alpha)/2} \|_{L^2(0,T; H^1(\Omega))} +
\big\|\nabla\barth\big\|_{L^2((0,T)\times\Omega;\R^d)} \right)
\\
\label{che-fatica}
&\phantom{\leq}
+ \| \eta\|_{L^\infty ((0,T)\times\Omega)}
\int_0^{\bartau(t)} \overline{\Lambda}_n \dd s\,,
\end{align}
where $\overline{\Lambda}_n$ denotes the piecewise constant
interpolant of the values ${(\Lambda \kk)}_k$.
Note that the estimate on
$\|(\theta_n^k)^{(\kappa +\alpha-2)/2}\nabla\theta_n^k\|_{L^2(\Omega;\R^d)}$ ensues
from \eqref{additional-info}   and   \eqref{est-alpha-later}. 
Now, observe that
\[
\big\|\barth\big\|_{L^{\kappa-\alpha+2}((0,T)\times\Omega)}^{(\kappa-\alpha+2)/2}
\leq C
\]
thanks to \eqref{apthetaLp}  if $p=\kappa-\alpha+2$ satisfies the
constraints in \eqref{apthetaLp}.   Recall that the parameter  $\alpha$ 
for which   \eqref{est-alpha-later} holds
can be chosen
arbitrarily close to $1$. Therefore,  such constraints for  $p=\kappa-\alpha+2$ are valid since, by \eqref{ass-K-b},
 $\kappa\in(1,\kappa_d)$ with $\kappa_d=5/3$ if $d{=}3$ and
$\kappa_d=2$ if $d{=}2$.  
Finally,
it follows from \eqref{ass-dataHeat-n},
\eqref{ap-u-lin-inter}, \eqref{apR1}, and \eqref{apthetaH12} that
$\int_0^T \overline{\Lambda}_n \dd t\leq C.$ Ultimately,  from
\eqref{che-fatica} we conclude~\eqref{apthetaBV}.
\end{proof}
%
\section{Passage from time-discrete to continuous}
\label{s:5}
%
Based on the a priori bounds deduced in Proposition \ref{Apriori},  exploiting 
compactness results \`a la Aubin-Lions  as well as 
a version of
Helly's selection principle,   we are now in a position
to extract a  subsequence of solutions of the time-discrete problems  converging to a limit triple $(u,z,\theta)$ in   suitable topologies. In \eqref{convs} below we have collected all of these convergences with some redundancies: for example, \eqref{convoz} and \eqref{convuz} imply 
\eqref{convozr} and \eqref{convptz}, but the latter are stated for later reference.
Subsequently, we will verify that the  triple  $(u,z,\theta)$    is an energetic solution of the time-continuous problem
 as
stated in Definition \ref{def4}.  
\begin{prop}[Convergence of the time-discrete solutions]
\label{Conv}
Let the assumptions of Theorem~\ref{thm:main} be satisfied. 
Then, there exists a triple $(u,z,\theta)\colon[0,T]\times\Omega\to\R^d \times\R\times[0,\infty)$
of regularity \eqref{reguu} such that for a.a.\
$x\in\Om$ the function $t\mapsto z(t,x)\in[0,1]$ is nonincreasing,
\eqref{strict-pos} holds, as well as \eqref{teta-pos+} under the assumption \eqref{h-pos+},
and there exists
a subsequence of the time-discrete solutions
${(\baru,\underu,u_n,\barz,\underz,\barth,\underth )}_n$
 from \eqref{interpolants}
such that
\begin{subequations}
\label{convs}
\begin{alignat}{4}
\label{convou}
\baru&\wtos\,\,&& u&&\text{ in }L^\infty(0,T;H^1_\mathrm{D}(\Omega;\R^d))\,,\\ \displaybreak[0]
\label{convu}
u_n&\rightharpoonup&& u &&\text{ in } H^1(0,T;H^1_\mathrm{D}(\Omega;\R^d))\,,\\ \displaybreak[0]
\dot u_n&\wtos&& \dot u&&\text{ in }L^\infty(0,T;L^2(\Omega;\R^d)) \label{convdotu}\,,\\ \displaybreak[0]
\label{convu-new}
\baru(t), \, u_n(t)&\rightharpoonup&& u(t) &&\text{ in } H^1_\mathrm{D}(\Omega;\R^d) \;\text{for all }t\in[0,T]\,,
\\ \displaybreak[0]
\label{ptcdu}
\dot u_n(t)&\rightharpoonup&& \dot u(t)&&\text{ in }L^2(\Omega;\R^d)\;\text{for all }t\in[0,T]\,,
\\ \displaybreak[0]
\label{convouz}
\barz\,,\, \underz&\wtos&& z&&\text{ in }L^\infty(0,T;W^{1,q}(\Omega))\cap L^\infty ((0,T)\times \Omega)\,,
\\ \displaybreak[0]
\label{convoz}
\barz(t)&\rightharpoonup&&{z}(t)&&\text{ in }W^{1,q}(\Omega) \text{ for all }t\in[0,T]\,,\\ \displaybreak[0]
\label{convozr}
\barz(t)&\to&&z(t)&&\text{ in }L^r(\Omega)\text{ for all }r\in[1,\infty)
\text{ and for all }t\in[0,T]\,,
\\ \displaybreak[0]
\label{convuz}
\underz(t)&\rightharpoonup&&{z}(t)&&\text{ in }W^{1,q}(\Omega)
\text{ for all }t\in[0,T]\backslash J\,,
\\ \displaybreak[0]
\label{convptz}
\underz(t)
&\to&&z(t)&&\text{ in }L^r(\Omega)\text{ for all }r\in[1,\infty)\text{ and for all }t\in[0,T]\backslash J\,,
\\ \displaybreak[0]
\label{convotheta}
\barth\,,\,\underth&\rightharpoonup&&\theta&&\text{ in }L^2(0,T;H^1(\Omega))\,,
\\ \displaybreak[0]
\label{convustrtheta}
\barth\,,\,\underth\,,\,\theta_n  &\to&&\theta&&\text{ in }L^2(0,T;Y) \quad \text{for all } Y \text{ such that }
H^1(\Omega) \Subset Y \subset W^{2,d+\delta}(\Omega)^*\,,
\\ \displaybreak[0]
\label{conv-theta-lp}
\barth\,,\,\underth\,,\,\theta_n  &\to&&\theta&&\text{ in }
L^p((0,T) \times \Omega) \quad \text{for all }
p\in\left\{
\begin{array}{ll}
[1,8/3)&\text{if }d{=}3\,,\\
{[1,3)}&\text{if }d{=}2\,,
\end{array}
\right.
\\
\label{ptw-conv-teta}
\theta_n(t) &\rightharpoonup && \theta(t)  &&
\text{ in } W^{2,d+\delta}(\Omega)^* \text{ for all } t \in [0,T]\,,
\end{alignat}
The set $J\subset[0,T]$ appearing  in
\eqref{convuz}--\eqref{convptz} denotes the jump set of
$z\in\BV([0,T];L^1(\Omega))$.
Finally,
\begin{equation}
\label{measure-convergence}
|\dot{z}_n| \to |\dot z| \quad \text{in the sense of measures on }  [0,T] \times \overline\Omega  \,.
\end{equation}
\end{subequations}
\end{prop}
\begin{proof}
{\bf Convergence of the displacements: } The convergences
\eqref{convou}, \eqref{convu}, and \eqref{convdotu} follow by
compactness from \eqref{apDu}, \eqref{ap-u-lin-inter}, and
\eqref{ap-u-lin-linfty}. As $u_n(t)-\baru(t)=(t-t_n^k)\dot u_n(t)$
and $u_n(t)-\underu(t)=(t-t_n^{k-1})\dot u_n(t)$, we immediately
deduce from \eqref{convu} that the sequences $u_n$, $\baru$, and
$\underu$ have the same limit in
$L^\infty(0,T;H^1_\mathrm{D}(\Omega;\R^d))$,
and the pointwise weak convergences \eqref{convu-new} ensue. 
 Furthermore, due to
estimate \eqref{apddu}, by compactness, there exists a further
subsequence such that $\dot u_n\rightharpoonup \dot u$ in
$\BV([0,T];W^{1,\gamma}_\mathrm{D}(\Omega;\R^d)^*)$ as well as $\dot
u_n(t)\rightharpoonup \dot u(t)$ in $W^{1,\gamma}_\mathrm{D}(\Omega;\R^d)^*$
for all $t\in[0,T]$. Thanks to \eqref{ap-u-lin-linfty}, arguing by
contradiction and using that $L^2(\Omega;\R^d)$ is dense in
$W^{1,\gamma}_\mathrm{D}(\Omega;\R^d)^*$, we may also conclude that $\dot
u_n(t)\rightharpoonup \dot u(t)$ in $L^2(\Omega;\R^d)$ for all
$t\in[0,T]$, i.e.\ \eqref{ptcdu}.
\par
{\bf Convergence of the damage variables: } From estimates
\eqref{apR1}
 on the $\calR_1$-total variation of $(\barz)_n$ (by monotonicity of $\barz$),
 combined with  \eqref{ap-z-w1q},
a generalized version of Helly's
selection principle,   cf.\ e.g.\ \cite[Theorem 6.1]{MieThe04RIHM}, 
allows us to extract a subsequence such
that $\barz(t)\wto z(t)$ and $\underz(t)\wto \underline z(t)$ weakly
in $W^{1,q}(\Om)$ for all $t\in[0,T]$, and $z,\underline z\in
L^\infty(0,T;W^{1,q}(\Omega))$. Moreover, the limit functions
$z$ and $\underline z$ inherit the monotonicity in time from
$\barz$ and $\underz$, hence $z,\underline z\in \BV([0,T];L^1(\Om))$,
and their jump sets $J$ and $\underline J$ are
at most countable. Let $t\in[0,T]\backslash (J\cup\underline J)$
fixed. Then, by \eqref{interpolants}, for every $n\in\N$ we have
$\barz(t-\tau_n)=\underz(t)$ and therefore as $n{\to}\infty$ we get
${z}(t)=\underline z(t)$. Let now $t\in J\cup\underline J$  and let
${(t_j^-)}_j\,,{(t_j^+)}_j\subset [0,T]\backslash(J\cup\underline J)$ be
such that $t_j^-\nearrow t$ and $t_j^+\searrow t$. Since $z$ and $\underline z$
coincide on $[0,T]\backslash(J\cup\underline J)$, we
deduce that the left and  the  right limit satisfy
$z^-(t)=\lim_jz(t_j^-)=\lim_j\underline z(t_j^-)=\underline z^-(t)$
and $z^+(t)=\lim_jz(t_j^+)=\lim_j\underline z(t_j^+)=\underline
z^+(t)$. Therefore $J=\underline J$ and the convergences
\eqref{convouz}, \eqref{convoz}, \eqref{convuz} hold. From this,
using \eqref{apzLinfty} we conclude that \eqref{convozr} and
\eqref{convptz} hold true as well.
In this line, we conclude by observing that \eqref{measure-convergence}
follows from the fact that $\int_\Omega (z_n(0)-z_n(T))\,{\mathrm d} x$,
i.e.\ the total variation of
$\dot{z}_n$ on $[0,T]\times\overline{\Omega}$, converges to the total variation
$\int_\Omega (z(0)-z(T))\,{\mathrm d} x$ of $\dot z$, also relying on
the argument from \cite[Proposition 4.3, proof of (4.80)]{Roub10TRIP}.
\par
{\bf Convergence of the temperature variables: } Due to estimate
\eqref{apthetaH12} we have $\barth\rightharpoonup\theta$ 
in $L^2(0,T;H^1(\Omega))$. Exploiting the definition of the
interpolants \eqref{interpolants}, similarly to the arguments for
the damage variables, we conclude that also
$\underth\rightharpoonup\theta$ in $L^2(0,T;H^1(\Omega))$, thus
\eqref{convotheta} is proven. From this, convergences
\eqref{convustrtheta} and \eqref{conv-theta-lp} for
$(\barth,\underth)_n$ follow by a
generalized Aubin-Lions Lemma, cf.\ \cite[Corollary 7.9, p.\ 196]{Roub05NPDE}, 
making use of the estimates \eqref{apthetaH12},
\eqref{apthetaLp}, and \eqref{apthetaBV}. Taking into account
that $|\theta_n (t,x)|\leq \max\{ |\barth (t,x)|,|\underth (t,x)|
\}$ for almost all $(t,x)\in (0,T) \times \Omega$,  (a generalized
version of) the Lebesgue Theorem yields convergence
\eqref{conv-theta-lp} for $(\theta_n)_n$ as well. All in all, we
conclude the weak convergence \eqref{convotheta}, as well as
\eqref{convustrtheta}, for $(\theta_n)_n$. 
 Convergence \eqref{ptw-conv-teta} is a consequence of
\cite[Theorem 6.1]{MieThe04RIHM}. The positivity properties
\eqref{strict-pos} and \eqref{teta-pos+} (under the additional
\eqref{h-pos+}) then follow from their discrete analogues
\eqref{tetak-pos} and \eqref{tetak-pos+h}, respectively, combined
with \eqref{apthetaLp}.
\end{proof}
The fact that the limit triple $(u,z,\theta)$ is an energetic
solution of the limit problem will be verified in Sections
\ref{Mombal}--\ref{EEHeat} right below.  For this, in Section
\ref{Mombal}, we first pass from time-discrete to continuous in the
weak momentum balance \eqref{disc-momentum} using suitably chosen
time-discrete test functions and deduce a time-continuous limit
\emph{inequality} for the mechanical energy balance
\eqref{discr-mech-energ} by lower semicontinuity arguments.
Secondly, in Section \ref{Semistab} we pass to the limit in the
semistability inequality \eqref{disc-semistab} using mutual
recovery sequences. As a further step in Section \ref{EEHeat} it has
to be verified that the limit triple $(u,z,\theta)$ indeed satisfies
the mechanical energy balance as an \emph{equality} by deducing the reverse
inequality from the momentum balance and the semistability so far
obtained. This result allows  us  to conclude the convergence of the
viscous dissipation terms, which, in turn, is crucial for the limit
passage in the heat equation \eqref{disc-heat}.

Altogether, these steps amount to the following
\begin{prop}[Energetic solution of the limit problem]\label{prop:existence-limit} 
Let the assumptions of Theorem~\ref{thm:main} be satisfied 
and let $(u,z,\theta)$ be a triple of regularity \eqref{reguu} 
obtained as a limit, in the sense of convergences \eqref{convs}, 
of a sequence of solutions to Problem  \ref{prob-discrete}. 
Then,  $(u,z,\theta)$ is an energetic solution of the
time-continuous problem \eqref{ourPDE}, supplemented with the
boundary conditions \eqref{bc}, in the sense of 
Definition \ref{def4}.
\end{prop}
\begin{proof}
The statement of the  proposition follows directly by combining Propositions \ref{prop:Mombal},
\ref{MechEE}, and
\ref{Heat} and Theorem \ref{MRS}.
\end{proof}
%
%
%
\subsection{Limit passage in the momentum balance and the energy inequalities}
\label{Mombal}
%
Based on the convergence properties \eqref{convs} we now pass from time-discrete to time-continuous
in the weak momentum balance. By lower semicontinuity we will then
carry out the limit passage in the mechanical as well as  in  the total energy  inequality
and obtain their analogues for the
limit problem.
\par
 Let us mention in advance that, while  the passage to the limit in most of the terms of the momentum balance  can be treated in a  straightforward
 way by exploiting the convergence properties \eqref{convs}, 
the quadratic terms arising from the
stored elastic energy and the viscous dissipation,  which involve the  state-dependent coefficients
 $\DD(\underz,\underth)$ and $\CC(\barz)$, need special attention.
For these terms the limit will be deduced by exploiting the $L^\infty$-bounds \eqref{ass-CD}
on $\CC$ and $\DD$ and the dominated convergence theorem.
\begin{prop}[Limit passage in the weak momentum balance]
\label{prop:Mombal}
Let the assumptions of Theorem~\ref{thm:main} be satisfied.
Then, a limit triple $(u,z,\theta)$ extracted  as  in Proposition
\ref{Conv} solves the time-continuous momentum balance
\eqref{e:weak-momentum-variational} at every $t\in[0,T]$.  In
particular,   it holds $\dot u\in H^1(0,T;H^1_\mathrm{D}(\Omega;\R^d)^*)
\cap
\rmC^0_{\mathrm{weak}} ([0,T];L^2(\Omega;\R^d))$. 
\end{prop}
\begin{proof} 
Let  $v\in L^2(0,T;H^1_\mathrm{D}(\Omega;\R^d))\cap
H^{1}(0,T;L^2(\Omega;\R^d))$ be a test function for   \eqref{e:weak-momentum-variational}. It follows from, e.g., 
 \cite[p.\ 56, Corollary 2]{Bure98SSOD} 
and \cite[p.\ 189, Lemma 7.2]{Roub05NPDE}, that for every 
$\eps>0$ there exists
\begin{equation}
\label{recseqv}
\begin{split}
&v^\star\in L^2(0,T;\rmC^1(\overline{\Omega};\R^d))\cap
L^2(0,T;H_\mathrm{D}^{1}(\Omega;\R^d))
\cap H^{1}(0,T;L^2(\Omega;\R^d))\colon \\
&\|v-v^\star\|_{L^2(0,T;H_\mathrm{D}^{1}(\Omega;\R^d))\cap H^{1}(0,T;L^2(\Omega;\R^d))}\leq\eps
\;\text{ and }\;v^\star=v\text{ on }\partial_\mathrm{D}\Omega\,\text{ in  the  trace sense.}
\end{split}
\end{equation}
In particular, $v^\star \in L^2(0,T;W^{1,\gamma}(\Omega;\R^d))$, with 
$\gamma>4$ the same exponent as in  the   regularizing  term $-\taun\div(|e(u)|^{\gamma-2} e(u))$
in time-discrete momentum balance \eqref{disc-momentum}. Therefore, the discrete test functions 
$(v^\star)\kk : = \frac1{\tau_n} \int_{t\km}^{t\kk} v^\star(s) \dd s $  for all $k=0,\ldots, n$ fulfill
$(v^\star)\kk  \in W^{1,\gamma}(\Omega;\R^d)$, so that   they are admissible  test functions for 
\eqref{disc-momentum}. 
We now  
consider the 
piecewise constant and linear
interpolants $\barv^\star$ and $v^\star_n$ of the elements $((v^\star)\kk )_{k=0}^n$.  
In view of \eqref{recseqv}, it can be checked that
\begin{subequations}
\label{recseqvconv}
\begin{equation}
\label{recseqvconv-a}
\begin{split}
&\barv^\star \to v^\star\text{ in }L^2(0,T;H^{1}_\mathrm{D}(\Omega;\R^d))
\;\text{ and }\; v_n^\star\to v^\star\text{ in }H^{1}(0,T;L^2(\Omega;\R^d))\,,\\
& \tau_n^{1/\gamma} \|e(\barv^\star)\|_{L^\gamma(0,T;L^\gamma(\Omega;\R^{d \times d}))}\to0\,.
\end{split}
\end{equation}
Observe that \eqref{recseqvconv-a}
implies
\begin{equation}
\label{recseqvconv-b}
v_n^\star(t)\to v^\star\text{ in }L^2(\Omega;\R^d)\text{ for all }t\in[0,T]\,.
\end{equation}
\end{subequations}
Using such sequences ${(\barv^\star,v^\star_n)}_n$ of interpolants of smooth, dense test functions,
we can now carry out the limit passage in  \eqref{disc-momentum}. 
 By the convergence properties of the given data
\eqref{ass-forces-n}  and for the smooth test functions
\eqref{recseqvconv}, together with the convergence results
\eqref{ptcdu},  \eqref{convu} and \eqref{convotheta} we immediately
find
\begin{align*}
\nonumber
&\rho\int_\Omega\big(\dot u_n(t) {\cdot} v^\star_n(t)-\dot u_0{\cdot}
v^\star_n(0)\big)\,\mathrm{d}x
-\int_0^{\bartau(t)}\!\!\!\left(\int_\Omega \big(\rho\dot
u_n(s{-}\tau_n){\cdot}\dot v^\star_n
-\barth\,\BB:e(\barv^\star)\big)\,\mathrm{d}x
- \pairing{}{H^1_\mathrm{D}(\Om;\R^d)}{\overline{\bigF}_n}{\barv^\star}\right)\,\mathrm{d}s\\
&\longrightarrow\;\rho\int_\Omega\big(\dot u(t){\cdot} v^\star(t)-\dot u_0{\cdot} v^\star(0)\big)\,\mathrm{d}x
-\int_0^t\!\!\!\left(\int_\Omega\big(\rho\dot u{\cdot}\dot v^\star
-\theta\,\BB:e(v^\star)\big)\,\mathrm{d}x
-\pairing{}{H^1_\mathrm{D}(\Om;\R^d)}{\bigF}{v^\star}\right)\,\mathrm{d}s \,.
\end{align*}
Moreover, the convergence of the term involving the $\gamma$-Laplacian follows from the estimate
\begin{equation*}
\left|\int_0^t\!\int_\Omega\tau_n
|e(\baru)|^{\gamma-2}e(\baru):e(\barv^\star)\,\mathrm{d}x\,\mathrm{d}s
\right|
\leq
\taun^{\tfrac{\gamma-1}{\gamma}}\|e(\baru)\|^{\gamma-1}_{L^\gamma((0,T)\times\Omega;\R^{d \times d})}
\taun^{\tfrac 1\gamma}\|e(\barv^\star)\|_{L^\gamma((0,T)\times\Omega;\R^{d \times d})}\;\to0\,,
\end{equation*}
due to the uniform bound \eqref{apDu-gamma}
and the convergence of ${(v^\star_n)}_n$ by \eqref{recseqvconv}.
\par
Finally, in order to handle the  remaining  quadratic terms with state-dependent coefficients  in \eqref{disc-momentum}, we will  prove that 
\begin{equation}
\label{provestrong}
\big(\DD(\underz,\underth)+\CC(\barz)\big)
e(\barv^\star)
\to\big(\DD(z,\theta)+\CC(z)\big)e(v^\star)\;\;\text{strongly in }L^2((0,T)\times\Omega;\R^{d \times d})\,.
\end{equation}
Then, the convergence of the quadratic terms with state-dependent
coefficients follows from weak-strong convergence, using that both
$e(\dot u_n)\rightharpoonup e(\dot u)$ and $e(u_n)\rightharpoonup
e(u)$ weakly in $L^2(0,T;L^2(\Omega;\R^{d \times d}))$ by
\eqref{convu}. Now, to verify \eqref{provestrong} we are going to
apply the dominated convergence theorem. For this, we observe that
for a.e.\ $t\in(0,T)$ we have
$|\big(\DD(\underz(t),\underth(t))+\CC(\barz(t))\big)
:e(\barv^\star(t))| \to|\big(\DD(z(t),\theta(t))+\CC(z(t))\big)
:e(v(t))|$ pointwise a.e.\ in $\Omega$, by assumption \eqref{conti}
and since by convergence results \eqref{convptz}  and
\eqref{convustrtheta} we can resort to a subsequence
${(\underz(t),\barz(t),\underth)}_n$ that converges pointwise a.e.\ in
$\Omega$ for a.e.\ $t\in(0,T)$. Moreover, by assumption
\eqref{ass-CD} 
we find an integrable, convergent majorant, i.e.,
\begin{equation*}
\big|\big(\DD(\underz,\underth)+\CC(\barz)\big)
e(\barv^\star)\big|
\leq (C_\DD^2+C_\CC^2)|e(\barv^\star)|\,\to\,(C_\DD^2+C_\CC^2)|e(v^\star)|
\end{equation*}
pointwise a.e.\ in $(0,T)\times\Omega$ and with respect to the strong
$L^2((0,T)\times\Omega))$-topology by \eqref{recseqvconv}.
Hence,  a generalized version of the Dominated Convergence Theorem, cf.\, e.g.,  \cite[Section 4.4, Theorem 19]{Royden},  yields \eqref{provestrong}.
This concludes the limit passage in the momentum balance for smooth test function as in \eqref{recseqv}.
By density this result carries over to all test functions
$v\in L^2(0,T;H_\mathrm{D}^{1}(\Omega;\R^d))\cap H^{1}(0,T;L^2(\Omega;\R^d))$.
As by \eqref{ptcdu} we have $\dot u(t)\in L^2(\Om;\R^d)$ for every $t\in[0,T]$, we immediately deduce that
\eqref{e:weak-momentum-variational} holds true at all $t\in [0,T]$.

The last assertion follows from Remark \ref{rmk:more-regularity}.
\end{proof}
\begin{lemma}[Energy inequalities by lower semicontinuity]\label{lemma:5.5} 
Let the assumptions of Theorem~\ref{thm:main} be satisfied 
and let $(u,z,\theta)$
be a limit triple given by Proposition~\ref{Conv}.
Then for every $t\in[0,T]$ we have
\begin{equation}\label{en-ineq-lsc}
\begin{aligned}
&\tfrac\rho 2\int_\Omega\!|\dot u(t)|^2 \dd x+\E(t,u(t),z(t))
+\int_\Omega (z(t){-}z_0)\dd x+\int_0^t\!\int_\Omega(\DD(z,\theta)e(\dot u){-}\theta\,\BB):e(\dot u)\dd x\dd s\\
&\le\tfrac\rho 2\displaystyle\int_\Omega|\dot u_0|^2\dd x
+\E(0,u_0,z_0)- \int_0^t \pairing{}{H^1_\mathrm{D}(\Om;\R^d)}{\dot{\bigF}}{v}\,\mathrm{d}s \,.
\end{aligned}
\end{equation}
\end{lemma}
\begin{proof}
It is enough to pass to the limit in \eqref{disc-mech-energy-ineq}
taking into account
\eqref{ass-forces-n-bis}, 
\eqref{convu-new}, 
\eqref{ptcdu}, \eqref{convptz}, and  \eqref{convustrtheta}.
\end{proof}
%
\subsection{Limit passage in the semistability inequality}
\label{Semistab}
%
In order to carry out the passage from time-discrete to continuous in the semistability inequality
we follow the well-established method of circumventing a direct passage to the limit on the 
left- and on the  right-hand side of the semistability inequality \eqref{disc-semistab}. 
Instead, it is enough to prove a
limsup inequality for the difference, cf.\ also
\cite{MiRou06,MRS06}, using a so-called mutual recovery sequence.
This procedure, which allows  one  to  take advantage of some
cancelations in the regularizing terms for the internal variable
$\calG(z,\nabla z)$, has been already employed in
\cite{MiRou06,ThoMie09DNEM,Thom11QEBV} in problems concerned with
(fully) rate-indepen\-dent, partial, isotropic and unidirectional
damage,
 featuring  a $W^{1,q}(\Omega)$-gradient regularization, with $q{>}d$ in \cite{MiRou06},
any $q{>}1 $ in \cite{ThoMie09DNEM} as in the
present context, and $q{=}1$ in \cite{Thom11QEBV}.
 In what follows,
we verify that the recovery sequence constructed in \cite{ThoMie09DNEM},
where $\calG(z,\nabla z)=|\nabla z|^q$, is also suited in our setting
of semistability with a general gradient term.

More precisely, let us fix $t\in[0,T]$ in the energy functionals $\calE_n$ 
from \eqref{discr-mech-energ},
and a sequence
${(v_n,\zeta_n)}_n\subset H^1_{\mathrm D}(\Om;\R^d)\times \calZ$ such that
\begin{equation}\label{vnzetan}
\begin{split}
& v_n\wto v\quad\text{ weakly in }H^1_{\mathrm D}(\Om;\R^d)\,,
\quad \zeta_n\wto \zeta \quad\text{ weakly in }W^{1,q}(\Om)\,,
\\
& {\calE}_n(t,v_n,\zeta_n)\le {\calE}_n(t,v_n,\hat\zeta)
+\calR_1(\hat\zeta-\zeta_n)\quad\text{ for all }\hat\zeta\in\calZ\,,
\end{split}
\end{equation}
i.e.,
$\zeta_n$ is semistable for ${\calE}_n(t,v_n,\cdot)$.
Given $\tilde\zeta\in\calZ$
let the recovery sequence ${(\tilde \zeta_n)}_n\subset\calZ$ be defined by
\begin{equation}
\label{constr}
\begin{split}
&\tilde \zeta_n:=\min\big\{\zeta_n,\max\{(\tilde  \zeta-\delta_n,0)\}\big\}=
\begin{cases}
(\tilde  \zeta-\delta_n)&\text{on }A_n=\big\{0\leq(\tilde  \zeta-\delta_n)\leq\zeta_n\big\}\,,\\
\zeta_n&\text{on }B_n=\big\{\tilde  \zeta-\delta_n>\zeta_n\big\}\,,\\
0&\text{on }C_n=\big\{\tilde  \zeta-\delta_n<0\big\}\,,
\end{cases}
\\
&\text{where }\delta_n:=\|\zeta_n- \zeta\|_{L^q(\Omega)}^{1/q}\,.
\end{split}
\end{equation}
The sequence ${(\tilde \zeta_n)}_n$ was introduced in \cite{ThoMie09DNEM} where it was shown that
\begin{equation}
\label{recseqc}
\tilde \zeta_n\rightharpoonup\tilde \zeta\;\text{ in $W^{1,q}(\Omega)$\; for $q\in(1,\infty)$
from \eqref{G-growth} fixed.}
\end{equation}
Note however that strong convergence in $W^{1,q}(\Omega)$ cannot be expected,
since $\zeta_n\rightharpoonup \zeta$ weakly in
$W^{1,q}(\Omega)$, only. This makes it impossible to show directly that
$\calG(\tilde \zeta_n,\nabla\tilde \zeta_n)\to\calG(\tilde \zeta,\nabla\tilde \zeta)$, 
since this would require the
strong convergence of the gradients. Nevertheless the following result holds.

\begin{thm}
\label{MRS} 
Let the assumptions of Theorem~\ref{thm:main} be satisfied. 
Let $t\in[0,T]$ be fixed and
consider a sequence ${(v_n,\zeta_n)}_n\subset H^1_{\mathrm D}(\Om;\R^d)\times \calZ$ such that
\eqref{vnzetan} holds.
Given $\tilde\zeta\in\calZ$,
let ${(\tilde \zeta_n)}_n\subset\calZ$ as in
\eqref{constr}.
Then
\begin{equation}
\label{mrs}
0\le\limsup_{n\to\infty}\Big({\calE}_n(t,v_n,\tilde \zeta_n)-{\calE}_n(t,v_n, \zeta_n)
+\calR_1(\tilde \zeta_n- \zeta_n)\Big)
\leq \calE(t,v,\tilde \zeta)-\calE(t,v,\zeta)+\calR_1(\tilde \zeta-\zeta)\,.
\end{equation}
Therefore the limit $\zeta$ is semistable for $\calE(t,v,\cdot)$.
\end{thm}
\begin{proof}
First of all note that, if $\tilde\zeta\in\calZ$ does not satisfy
$0\le\tilde\zeta\le\zeta$, then \eqref{mrs} trivially holds, since
in this case $\calR_1(\tilde \zeta-\zeta)=+\infty$.

Assume now $0\le\tilde\zeta\le\zeta$ for a.e.\ $x\in\Om$.
Let us estimate the left-hand side of \eqref{mrs} as follows:
\begin{align}
\label{esthelp}
&\limsup_{n\to\infty}
\Big({\calE}_n(t,v_n,\tilde \zeta_n)-{\calE}_n(t,v_n, \zeta_n)
+\calR_1(\tilde \zeta_n- \zeta_n)\Big)\\
\nonumber
&\leq
\limsup_{n\to\infty}\int_\Omega(\CC(\tilde \zeta_n)-\CC( \zeta_n))e(v_n):e(v_n)\,\mathrm{d}x
+\limsup_{n\to\infty}\big(\calG(\tilde \zeta_n,\nabla \tilde \zeta_n)-\calG( \zeta_n,\nabla  \zeta_n)\big)
+\limsup_{n\to\infty}\calR_1(\tilde \zeta_n- \zeta_n)
\end{align}
and then treat each of the terms on the right-hand side of \eqref{esthelp} separately.
Since $ \zeta_n\rightharpoonup \zeta$ in $W^{1,q}(\Omega)$, we may choose a (not relabeled) subsequence
that converges pointwise a.e.\ in $\Omega$.
\par
{\bf Estimation of $\limsup_{n\to\infty}
\big(\calG(\tilde \zeta_n,\nabla\tilde \zeta_n)-\calG( \zeta_n,\nabla \zeta_n)\big)$: }
Note that $G(\tilde \zeta_n,\nabla\tilde \zeta_n)=G( \zeta_n,\nabla \zeta_n)$ on $B_n$.
If $\| \zeta_n-\zeta\|_{L^q(\Omega)}>0$, by Markov's inequality
\begin{equation*}
\calL^d(B_n)\leq\calL^d([\delta_n\leq| \zeta_n-\zeta|])\le \tfrac{1}{\delta_n}\int_\Omega| \zeta_n-\zeta|\dd x
\leq \tfrac{1}{\delta_n}\| \zeta_n-\zeta\|_{L^q(\Omega)}\to0 \,,
\end{equation*}
with $\delta_n$ from \eqref{constr},
while for $\| \zeta_n-\zeta\|_{L^q(\Omega)}=0$
it is indeed $\calL^d(B_n)=0$,
thus
\begin{equation}
\label{convsets}
\calL^d(A_n\cup C_n)\to\calL^d(\Omega)\,.
\end{equation}
In what follows, $\calX_D$ will denote the characteristic function of  a  set $D$.
By \eqref{Gcont}, \eqref{G-growth} and \eqref{constr}, we deduce
\begin{subequations}
\begin{align}
\nonumber 
&\limsup_{n\to\infty}\big(\calG(\tilde \zeta_n,\nabla\tilde \zeta_n)-\calG( \zeta_n,\nabla \zeta_n)\big)\\
\nonumber \displaybreak[0]
&=\limsup_{n\to\infty}\int_{A_n}G((\tilde \zeta-\delta_n),\nabla\tilde \zeta)\,\mathrm{d}x
+\int_{C_n}G(0,0)\,\mathrm{d}x
-\int_{A_n\cup C_n}G( \zeta_n,\nabla \zeta_n)\,\mathrm{d}x\\
\nonumber \displaybreak[0]
&\leq\limsup_{n\to\infty}\Big(
\int_{\Omega}G(\calX_{A_n}(\tilde \zeta-\delta_n),\calX_{A_n}\nabla\tilde \zeta)\,\mathrm{d}x
+\int_{\Omega}G( 0,\calX_{C_n}\nabla\tilde \zeta)\,\mathrm{d}x
-\int_{\Omega}G(\calX_{A_n\cup C_n} \zeta_n,\calX_{A_n\cup C_n}\nabla \zeta_n)
\,\mathrm{d}x\Big)\\
\nonumber \displaybreak[0]
&=\limsup_{n\to\infty}\Big(
\int_{\Omega}G(\calX_{A_n\cup C_n}(\tilde \zeta_n),
\calX_{A_n\cup C_n}\nabla\tilde \zeta)\,\mathrm{d}x
-\int_{\Omega}G(\calX_{A_n\cup C_n} \zeta_n,\calX_{A_n\cup C_n}\nabla \zeta_n)
\,\mathrm{d}x\Big)\\
\label{est1} \displaybreak[0]
&\leq\calG(\tilde \zeta,\nabla\tilde \zeta)
-\liminf_{n\to\infty}\calG(\calX_{A_n\cup C_n} \zeta_n,\calX_{A_n\cup C_n}\nabla \zeta_n)\\
\label{est2}
&\leq\calG(\tilde \zeta,\nabla\tilde \zeta)-\calG(\zeta,\nabla \zeta)\,,
\end{align}
\end{subequations}
 where in the second integral term in the third line we have used the obvious identity $\calX_{C_n} 0=0$.
To obtain \eqref{est1} we have used the dominated convergence theorem, while in order
to prove \eqref{est2} we employed the lower semicontinuity of
$\calG:L^q(\Omega)\times L^q(\Omega;\R^d)\to\R\cup\{\infty\}$,  since,
by \eqref{recseqc} and \eqref{convsets},
we have $\calX_{A_n\cup C_n} \zeta_n\to \zeta$ strongly in $L^q(\Omega)$ and
$\calX_{A_n\cup C_n}\nabla \zeta_n\rightharpoonup\nabla \zeta$ weakly in $L^q(\Omega;\R^d)$.
\par
{\bf Estimation of the remaining terms in \eqref{esthelp}: } Since construction \eqref{constr}
ensures  $\tilde \zeta_n \le  \zeta_n$  for every $n\in\N$, as well as
$\tilde \zeta_n\to\tilde \zeta$
in $L^q(\Omega)$, due to $ \zeta_n\to \zeta$ in $L^q(\Omega)$, we immediately conclude that
$\calR_1(\tilde \zeta_n- \zeta_n)\to\calR_1(\tilde \zeta- \zeta)$.
\par
We now estimate the difference of the quadratic terms in the mechanical energy. As $\tilde \zeta_n\leq \zeta_n$, by the monotonicity assumption~\eqref{mono} we have that
$(\CC(\tilde \zeta_n)-\CC( \zeta_n))e(v_n):e(v_n)\leq 0$. 
Since
both $ \zeta_n\to \zeta$ and $\tilde \zeta_n\to\tilde \zeta$ in $L^q(\Omega)$, 
the Lipschitz-continuity of $\CC$, cf.\ \eqref{conti}, implies that
$\CC(\tilde \zeta_n)-\CC( \zeta_n)\to(\CC( \tilde \zeta)-\CC(\zeta))$ in $L^q(\Omega;\R^{d \times d \times d \times d}_{\mathrm{sym}})$.
Let us consider the auxiliary functional
$\calC:L^q(\Omega)\times L^q(\Omega)\times L^2(\Omega;\R^{d \times d})\to\R$ defined by
\begin{equation*}
\calC(\zeta,\tilde \zeta,e):=\int_\Omega (\CC(\zeta(x))-\CC(\min\{\zeta(x),\tilde\zeta(x)\})) e(x):e(x)\,\mathrm{d}x\,.
\end{equation*}
By e.g.\ \cite[Theorem 7.5, p.\ 492]{FoLeo07} the functional $\calC$ is lower semicontinuous with respect 
to the strong convergence in $L^q(\Omega)\times L^q(\Omega)$ and  the
weak convergence in $L^2(\Omega;\R^{d \times d})$. 
%
Thus, the first term on the right-hand side of \eqref{esthelp} can be
rewritten and estimated  as follows, using \eqref{ap-u-lin-inter} and the lower semicontinuity of $\calC$,
\begin{align*}
\limsup_{n\to\infty}\int_\Omega(\CC(\tilde \zeta_n)-\CC( \zeta_n))e(v_n):e(v_n)\,\mathrm{d}x
&
\leq\int_\Omega(\CC(\tilde \zeta)-\CC( \zeta))e(v):e(v)\,\mathrm{d}x\,.
\end{align*}
\par
Combining the above established estimates for the three terms on the right-hand side of \eqref{esthelp}
shows that condition \eqref{mrs} is satisfied. 
\end{proof}
%
\subsection{Energy equalities and limit passage in the heat equation}
\label{EEHeat}
%
We now show that the limit triple $(u,z,\theta)$ satisfies the
mechanical energy  equality  \eqref{mech-energy-ineq}. The
inequality ($\le$)  has been proven in Lemma \ref{lemma:5.5}. The
opposite inequality is found by approximation with Riemann sums, as
common in existence proofs of rate-independent and rate-dependent
evolutions, see e.g.~\cite{DFT05}.
\begin{prop}[Mechanical energy equality] 
Let the assumptions of Theorem~\ref{thm:main} be satisfied,
let $(u,z,\theta)$ be  a triple given by Proposition~\ref{Conv}, and let $t\in[0,T]$.
Then \eqref{mech-energy-ineq} holds.
\label{MechEE}
\end{prop}
\begin{proof}
 We fix  a sequence of subdivisions ${(s_n^k)}_{0\le k\le k_n}$ of
the interval $[0,t]$, with
$0=s_n^0<s_n^1<\dots<s_n^{k_n-1}<s_n^{k_n}=t$,
$\lim_n\max_k(s_n^k-s_n^{k-1})=0$, and
\be
\label{140522}
\mod{\sum_{k=1}^{k_n} \integlin{s\km}{s\kk}{\!\integ{\Om}{\left[\CC(z(s\kk))
{-}\CC(z(s))\right]e(u(s)):e(\dot u(s))}{x}}{s} }\to0 \,.
\ee
The existence of such a sequence is guaranteed by \cite{Hahn14},
see also \cite[Proposition 4.3, Step 7]{Roub10TRIP}.
Taking $z(s_n^k)$ as test function in  the time-continuous
semistability inequality  \eqref{semistab-general} at time $s\km$ we
get
\begin{align*} \displaybreak[0]
\E(s\km,u(s\km),z(s\km)) &\le \ \E(s\km,u(s\km),z(s\kk))+\integ{\Om}{(z(s\km){-}z(s\kk))}{x} \\
&= \E(s\kk,u(s\kk),z(s\kk))+\integ{\Om}{(z(s\km){-}z(s\kk))}{x}
- \integlin{s\km}{s\kk}{\partial_t\E(s,u(s),z(s))}{s} \\
& \phantom{=}+ \integlin{s\km}{s\kk}{\pairing{}{H^1_\mathrm{D}(\Om;\R^d)}{ \bigF(s)}{\dot u(s)}}{s}
- \integlin{s\km}{s\kk}{\!\integ{\Om}{\CC(z(s\kk))e(u(s)):e(\dot u(s))}{x}}{s} \,.
\end{align*}
%
%
Next we sum up the previous inequality over $k=1,\dots,k_n$ and we
pass to the limit in $n$ in the last term thanks to \eqref{140522},
obtaining
\begin{equation}\label{e0-le-et}
\begin{aligned}
\E(0,u_0,z_0)&\le
\E(t,u(t),z(t))+\integ{\Om}{(z_0{-}z(t))}{x}-\!\integlin{0}{t}{\!\partial_t\E(s,u(s),z(s))}{s} \\
&\phantom{=}+\integlin{0}{t}{\!\!\!\pairing{}{H^1_{\mathrm D}(\Om;\R^d)}{ \bigF(s)}{\dot u(s)}}{s}
-\integlin{0}{t}{\!\integ{\Om}{\CC(z(s))e(u(s)):e(\dot u(s))}{x}}{s}
\,.
\end{aligned}
\end{equation}
Further, thanks to Remark \ref{rmk:more-regularity} we can test
\eqref{e:weak-momentum-variational} by $\dot u$ and get
\begin{equation}\label{testdotu}
\begin{aligned}
&\tfrac{\rho}{2} \|\dot u(t)\|^2_{L^2(\Omega;\R^d)}
+\integlin{0}{t}{\!\integ{\Omega}{\left( \DD(z,\theta)e(\dot u)
+  \CC(z)e(u)  -\theta\,\BB \right) \psm e(\dot u)}{x}}{s}
\\
&= \tfrac{\rho}{2}\|\dot u_0\|^2_{L^2(\Omega;\R^d)}
+\integlin{0}{t} {\pairing{}{H^1_{\mathrm D}(\Omega;\R^d)}{\bigF}{\dot u}}{s}\,,
\end{aligned}
\end{equation}
where we applied the by-part integration formula \eqref{gelfand}, as allowed by \cite[Lemma 7.3]{Roub05NPDE}.
Summing up \eqref{testdotu} with \eqref{e0-le-et} we obtain
\begin{align*}
\E(0,u_0,z_0)&\le \E(t,u(t),z(t))+
\tfrac\rho2{\integ{\Om}{\mod{\dot u(t)}^2}{x}}
+\integ{\Om}{(z_0-z(t))}{x}
-\integlin{0}{t}{\partial_t\E(s,u(s),z(s))}{s}\\
& \phantom{\le} - \tfrac\rho2{\integ{\Om}{\mod{\dot u_0}^2}{x}}
+\integlin{0}{t}{\!\integ{\Omega}{\left( \DD(z(s),\theta(s))e(\dot u(s))
-\theta(s)\,\BB \right) \psm e(\dot u(s))}{x}}{s}  \,.
\end{align*}
Combining this estimate with the reverse inequality
\eqref{en-ineq-lsc}  concludes the proof of
\eqref{mech-energy-ineq}.
\end{proof}
\par
In order to prove a stronger convergence of the displacements we shall repeatedly make use of the following result. 
Given two constants $C_1,C_2$ with $0<C_1\le C_2$, let  ${\mathcal T}_{C_1,C_2}$ denote the class of  tensors $\A\in \R^{d \times d \times d \times d}$ that are symmetric, i.e.,
\begin{equation*}
\A_{ijkl}=\A_{jikl}=\A_{ijlk}=\A_{klij} \,,
\end{equation*}
positive definite and bounded:
\begin{equation} 
\label{appendixconti-c} 
C_1 \mod{A}^2 \le \A\,A: A\le C_2 \mod{A}^2\quad\text{for every } A\in \Rsym\,.
\end{equation}
%
%
\begin{lemma}\label{appendix}
Let ${\mathcal K}_n$ be the functional defined by
\begin{equation*}
{\mathcal K}_n(e):=\int_0^T\int_\Omega \A_n(t,x)e(t,x):e(t,x)\, \d x\,\d t\quad\text{for every }e\in L^2((0,T)\times \Omega;\R^{d\times d})\,,
\end{equation*}
where $\A_n\in L^\infty((0,T)\times\Omega;{\mathcal T}_{C_1,C_2})$
are such that
\begin{subequations}
\begin{eqnarray}\label{1902}
\label{1902a} & \A_n(t,x)\to \A_\infty(t,x)\quad\text{for a.e. $t\in(0,T)$ and a.e. $x\in\Omega$}\,, \\
\label{1902b} & e_n\wto e_\infty\quad\text{weakly in }L^2((0,T)\times\Omega;\R^{d\times d})\,,\\
\label{1902c} &
\limsup_{n\to\infty}{\mathcal K}_n(e_n)\le{\mathcal K}_\infty(e_\infty)\,,
\end{eqnarray}
\end{subequations}
and  ${\mathcal K}_\infty$ is defined by
$$
{\mathcal K}_\infty(e):=\int_0^T\int_\Omega \A_\infty(t,x)e(t,x):e(t,x)\, \d x\, \d t\quad\text{for every }e\in L^2((0,T)\times \Omega;\R^{d\times d})\,.
$$
Then, $\lim_{n\to\infty} {\mathcal K}_n(e_n)  =  {\mathcal K}_\infty(e_\infty)$ and 
\begin{equation}\label{strongen}
e_n\to e_\infty\quad\text{ strongly in }L^2((0,T)\times\Omega;\R^{d\times d})\,.
\end{equation}
\end{lemma}
\begin{proof}
It is enough to observe that under the above hypotheses  $\A_\infty\in L^\infty((0,T)\times\Omega;{\mathcal T}_{C_1,C_2})$  and
\begin{equation*}
\begin{split}
\displaystyle {\mathcal K}_n(e_n-e_\infty)&=\int_0^T\int_\Omega \A_n(t,x)(e_n(t,x)-e_\infty):(e_n(t,x)-e_\infty(t,x))\, \d x\,\d t\\
&\displaystyle ={\mathcal K}_n(e_n)-2\int_0^T\int_\Omega \A_n(t,x)e_\infty(t,x):e_n(t,x)\, \d x\,\d t+{\mathcal K}_n(e_\infty)\,.
\end{split}
\end{equation*}
By \eqref{appendixconti-c} and \eqref{1902} we obtain $\limsup_n {\mathcal K}_n(e_n-e_\infty)\le0$.
Since $\A_n(t,x)\in{\mathcal T}_{C_1,C_2}$ we have ${\mathcal K}_n(e_n-e_\infty)\ge C_1\|e_n-e_\infty\|^2_{L^2((0,T)\times \Omega;\R^{d\times d})}$, so that \eqref{strongen} holds.
\end{proof}

Thanks to the mechanical energy inequality proven above, we may deduce strong convergence of the displacements, as provided in the following lemma.

\begin{lemma}[Stronger convergences]\label{strongerconv}
Let the assumptions of Theorem~\ref{thm:main} be satisfied  and let
$(u,z,\theta)$ be  a triple given by Proposition~\ref{Conv}. Then
\begin{equation}\label{strongconvD}
\lim_{n\to\infty}
\integlin{0}{T}{\!\integ{\Omega}{\DD(\underz,\underth) e(\dot u_n)\psm e(\dot u_n)}{x}}{t}
=\integlin{0}{T}{\!\integ{\Omega}{\DD(z,\theta) e(\dot u)\psm e(\dot u)}{x}}{t}
\end{equation}
and  
then
\begin{equation}\label{further-strong-convdote}
e(\dot u_n)\to e(\dot u)\quad\text{strongly in } L^2 ((0,T)\times\Omega;\R^{d \times d})\,.
\end{equation}
\end{lemma}
%
\begin{proof}
By lower semicontinuity, taking into account the convergences already proven in Proposition~\ref{Conv},
together with both the discrete mechanical energy inequality \eqref{disc-mech-energy-ineq}
and the mechanical energy equality~\eqref{mech-energy-ineq}, the following chain of inequalities holds:
\begin{align*}
&\integlin{0}{T}{\!\integ{\Omega}{\DD(z,\theta) e(\dot u)\psm e(\dot u)}{x}}{t}
+\int_\Omega (z_0-z(T))\,{\mathrm d} x\\ \displaybreak[0]
& \le\liminf_n\left(\integlin{0}{T}{\!\integ{\Omega}{\DD(\underz,\underth) e(\dot u_n)\psm e(\dot u_n)}{x}}{t}
+\int_\Omega (z_n(0){-}z_n(T))\,{\mathrm d} x\right)\\ \displaybreak[0]
&\le\limsup_n\left(\integlin{0}{T}{\!\integ{\Omega}{\DD(\underz,\underth) e(\dot u_n)\psm e(\dot u_n)}{x}}{t}
+\int_\Omega (z_n(0){-}z_n(T))\,{\mathrm d} x\right)
\\
& \le\limsup_n \Bigg(-\E_n(T,u_n(T),z_n(T))+ \mathcal{E}_n (0,u_0,z_0)
-\tfrac\rho2{\integ{\Om}{\mod{\dot u_n(T)}^2}{x}}
+\tfrac\rho2{\integ{\Om}{\mod{\dot u_0}^2}{x}} \\ \displaybreak[0]
&\phantom{\ \le\limsup_n\Big(} \left. +\integlin{0}{T}{\!\integ{\Omega}{\barth\,\BB\psm e(\dot u_n)}{x}}{t}
+\integlin{0}{T} {\partial_t \E_n (s,\underu,\underz) }{s}
\right) \\
&\le -\E(T,u(T),z(T)))+\mathcal{E} (0,u_0,z_0)-\tfrac\rho2{\integ{\Om}{\mod{\dot u(T)}^2}{x}}
+\tfrac\rho2{\integ{\Om}{\mod{\dot u_0}^2}{x}}\\ \displaybreak[0]
&\phantom{\le}+\integlin{0}{T}{\!\integ{\Omega}{\theta\,\BB\psm e(\dot u)}{x}}{t}
+\integlin{0}{T} {\partial_t \E (s,u,z) }{s}\\
&=\integlin{0}{T}{\!\integ{\Omega}{\DD(z,\theta) e(\dot u)\psm e(\dot u)}{x}}{t}
+\int_\Omega (z_0{-}z(T))\,{\mathrm d} x\,.
\end{align*}
Hence all inequalities above are actually equalities and we deduce that \eqref{strongconvD} holds.
\par
Next, we apply Lemma~\ref{appendix} with $\A_n=\DD(\underz,\underth)$, $\A_\infty=\DD(z,\theta)$, $e_n=e(\dot u_n)$, and $e_\infty=e(\dot u)$.
Indeed, \eqref{1902a} is obtained from the strong convergences \eqref{convptz} and \eqref{convustrtheta} up to the passage to a further subsequence converging pointwise; the weak convergence \eqref{1902b} is given in \eqref{convu}, while \eqref{1902c} is provided by \eqref{strongconvD}.
Therefore we deduce that \eqref{further-strong-convdote} holds (for the initial subsequence, since the limit is the same for all subsubsequences).
\end{proof}
Finally, we pass to the limit in the heat equation.
\begin{prop}[Limit passage in the weak form of the heat equation]
\label{Heat}
Let the assumptions of Theorem~\ref{thm:main} be satisfied, Let
$(u,z,\theta)$ be  a triple given by Proposition~\ref{Conv}, and let
$t\in[0,T]$. Then the weak formulation of the heat equation
\eqref{weak-heat} holds.
\end{prop}
\begin{proof}
Let us fix  $\eta\in H^1(0,T;L^2(\Omega))\cap \mathrm{C}^0([0,T];W^{2,d+\delta}(\Omega))$,
define $\eta_n^k:= \eta(t_n^k)$
for all $k=0, \ldots, n$, and let
$\eta_n$, $\bareta$ be the piecewise linear and constant interpolations of the values $(\eta_n^k)$.
It can be checked that
\begin{equation}
\label{convergences-interpolants-n}
\begin{aligned}
&
\bareta \to \eta \quad \text{in } L^p (0,T;W^{2,d+\delta}(\Omega))
\text{ for all } 1 \leq p<\infty\,, \qquad \bareta \wtos \eta \quad
\text{in } L^\infty (0,T;W^{2,d+\delta}(\Omega))\,,
\\
& \eta_n \to \eta \quad \text{in } H^1(0,T;L^2(\Omega))\cap \mathrm{C}^0(0,T;W^{2,d+\delta}(\Omega))\,.
\end{aligned}
\end{equation}
We now pass to the limit in the discrete heat equation \eqref{disc-heat}
tested by $\eta_n$. The first three integral terms on the left-hand side of  \eqref{disc-heat}
can be dealt with combining convergences
\eqref{convustrtheta}--\eqref{ptw-conv-teta} with \eqref{convergences-interpolants-n}.
In order to pass to the limit in the fourth one, we argue along the lines of
\cite[proof of Theorem 2.8]{RoRoESTC14} and derive a finer estimate for $ {(\KK(\barz,\barw)\nabla \barw )}_n$.
Indeed, thanks to \eqref{ass-K-b} we have
\[
|\KK(\barz,\barw)\nabla \barw |
\leq c_2 (|\barw|^{(\kappa -\alpha+2)/2} |\barw|^{(\kappa +\alpha-2)/2}
|\nabla \barw| + |\nabla \barw| ) \qquad \aein (0,T) \times \Omega\,,
\]
with $\alpha$ as in \eqref{additional-info}.
From this particular estimate we also gather that
$|\barw|^{(\kappa +\alpha-2)/2} |\nabla \barw|$ is bounded in $L^2 ((0,T) \times \Omega)$.
Since ${(\barw)}_n$ is bounded in $L^{8/3}((0,T) \times \Omega)$ if $d{=}3$ (and in
$L^3 ((0,T) \times \Omega)$ if $d{=}2$), choosing $\alpha \in (1/2,1)$ such that
$\kappa - \alpha<2/3$ (which can be done, since $\kappa<5/3$), we conclude that
$|\barw|^{(\kappa -\alpha+2)/2}$ is bounded in $L^{2+\delta} ((0,T) \times \Omega)$
for some $\delta>0$. All in all,
we have that $\KK(\barz,\barw)\nabla \barw$ is bounded in
$L^{1+\delta} ((0,T) \times \Omega;\R^d)$ for some $\delta>0$.
With the very same arguments as in \cite[proof of Theorem 2.8]{RoRoESTC14}, we show that
\[
\KK(\barz,\barw)\nabla \barw \rightharpoonup \KK(z,\theta) \nabla \theta \quad
\text{in } L^{1+\delta} ((0,T) \times \Omega;\R^d)\,,
\]
which, combined with convergences
\eqref{convergences-interpolants-n} for $\bareta$, is enough to pass
to the limit in the last term on the left-hand side of
\eqref{disc-heat}.

Combining \eqref{convu}, \eqref{conv-theta-lp},
and \eqref{convergences-interpolants-n} yields
$ \integlin{0}{\bartau(t)}{\!\integ{\Om}{\barw\,\BB \psm e(\dot u_n) \,\bareta}{x} }{s} \to
\integlin{0}{t}{\integ{\Om}{ \theta\,\BB \psm e(\dot u) \,\eta}{x}}{s}$
as $n{\to}\infty$, while the passage to the limit in the term
\[
\integlin{0}{\bartau(t)}{\!\integ{\Om}{\DD(\underz,\underth)e(\dot
u_n) \psm e(\dot u_n) \,\bareta}{x}} {s} 
\]
results from
\eqref{further-strong-convdote} combined with \eqref{convergences-interpolants-n}. 
Convergence \eqref{measure-convergence} allows us to deal with the second term
on the right-hand side of \eqref{disc-heat}, and we handle the last
two terms via \eqref{ass-dataHeat-n} and
\eqref{convergences-interpolants-n}, again. This concludes the proof
 of the weak heat equation and of the main existence result Theorem \ref{thm:main}.
\end{proof}
%
\section{Asymptotic behavior in the  slow loading regime: the  vanishing viscosity and inertia limit}
\label{s:6}
%
In this section we  address the limiting behavior of system
\eqref{ourPDE}  as the rate of the external load and of the heat sources
becomes slower and slower. Accordingly, we will rescale time by a factor $\eps>0$.
For analytical reasons we restrict to the case
of a Dirichlet problem in the displacement, namely  within this section  we shall suppose that
\begin{equation}
\label{Dir-b-c}
\partial_{\mathrm{D}}\Omega = \partial \Omega
\,.
\end{equation}
Like in the previous sections, we assume that the Dirichlet datum is homogeneous, cf.\ \eqref{bc-b}.
\par
As $\eps\downarrow0$ we will  \emph{simultaneously}
pass to
\begin{compactenum}
\item
a rate-independent system for the limit displacement and damage variables $(u,z)$,
which does not display any temperature dependence and which formally reads
\%label{our-pde-ris}
\begin{alignat*}{3}
&&&
-\div \CC(z) e(u)
= \fv
 \ & \text{in }
(0,T)\times\Omega \,,
\\
&&&
\partial \mathrm{R}_1 (\dot z) + \mathrm{D}_z G(z,\nabla z)
- \div(\mathrm{D}_\xi G(z,\nabla z) )  +\tfrac12 \CC'(z)e(u) : e(u) \ni 0 \ &
\text{in } (0,T)\times\Omega \,
\end{alignat*}
and will be weakly formulated through the concept of \emph{local solution} to a rate-independent system;
\item
a limit temperature $\theta=\Theta$, which is constant in space, but still time-dependent.
The limit passage in the heat equation amounts to the trivial limit $0=0$,
once more emphasizing that the limit system does not depend on temperature any more.
A rescaling of the heat equation at level $\eps$, however, reveals that $\Theta$ 
evolves in time according to an ODE in the sense of measures and the evolution
is driven by the rate-independent dissipation and a measure
originating from the viscous dissipation. 
\end{compactenum}
Indeed, for the limit system we expect that, if a change of heat is caused
at some spot in the material, then the heat must be conducted all over the material
with infinite speed, so that the temperature is kept constant in space.
This justifies a scaling of the tensor of heat
conduction coefficients for the systems at level $\eps$. 
More precisely, 
we will suppose that
\begin{equation}
\label{scaling-K} 
\KK_\eps(z,\theta):= \tfrac{1}{\eps^\beta}  \KK(z,\theta)\quad\text{with }\KK\text{ satisfying \eqref{ass-K}  and }  \beta>0  \,.
\end{equation}
 While Proposition \ref{Apriori-eps} holds with $\beta>0$, in Theorem \ref{Thm6.1} we shall require $\beta\ge2$. 
\subsection{Time rescaling}
\label{ss:6.1}
Let us now set up the vanishing viscosity analysis
following \cite{Roub09RIPV}, where this analysis was carried out for
\emph{isothermal} rate-independent processes in viscous solids, see also
\cite{DMScQEPP13} in the context of perfect plasticity and
\cite{Roub13ACVB,Scal14LVDP}  for delamination, still in the isothermal case.
We consider a family
${(\fve, \hve, \hse)}_\eps$ of data for system
\eqref{ourPDE} and
 we rescale
$\fve,\,\hve, \,\hse$
 by the factor $\eps>0$,
hence we introduce
\begin{equation*}
\bigF^\eps(t) 
:= \fve(\fte)\, \qquad
\hvue(t):= \hve(\fte)\,, \qquad
\hsue(t):=\hse(\fte) \qquad
\text{for } t \in [0,\tfrac{T}\eps]\,.
\end{equation*}
Theorem  \ref{thm:main} guarantees that for every $\eps>0$ there
exists an energetic solution $(u^\eps, z^\eps,\theta^\eps)$, defined
on $[0,\tfrac T\eps]$, to (the Cauchy problem for) system
\eqref{ourPDE} supplemented with the data
$\bigF^\eps, \,\hvue, \,\hsue$,
and with the matrix of  heat conduction coefficients $\KK_\eps$ from  \eqref{scaling-K}.
For later convenience,
let us recall that such solutions arise as limits of
the time-discrete solutions to Problem \ref{prob-discrete}.
We now perform a rescaling of the solutions in  such a way as to
have them defined on the interval $[0,T]$. Namely, we set
\begin{equation*}
\ures(t) := u^\eps (\tfrac t\eps)\,, \qquad
\zres(t) := z^\eps (\tfrac t\eps)\,, \qquad \thres(t):= \theta^\eps
(\tfrac t\eps) \qquad \text{for } t \in [0,T]\,.
\end{equation*}
It is not difficult to check that, after transforming the time scale,
the triple $(\ures,\zres,\thres)$ (formally) solves the
following system in $(0,T)\times\Omega$:
\begin{subequations}
\label{our-pde-van}
\begin{alignat}{3}
\label{eq:u-van} &&&
 \eps^2\rho\ddot u_\eps-\div\big(\eps\,\DD(z_\eps,\theta_\eps)e(\dot u_\eps) + \CC(z_\eps) e(u_\eps)
-\theta_\eps\,\BB\big)
= \bigFe\,,
\\
\label{eq:z-van} &&&
\partial \mathrm{R}_1 (\dzres) + \mathrm{D}_z G(\zres,\nabla \zres)
- \div(\mathrm{D}_\xi G(\zres,\nabla \zres) )
+\tfrac12 \CC'(\zres)e(\ures) : e(\ures) \ni 0\,,
\\
\label{eq:theta-van} &&&
\eps\dthres-\tfrac1{ \eps^\beta }\div(\KK(\zres,\thres)\nabla\thres)=\eps\mathrm{R}_1(\dzres)+
\eps^2\DD(\zres,\thres) e(\dures)\psm e(\dures) -
\eps\thres\,\BB\psm e(\dures) + H_\eps\,,
\end{alignat}
\end{subequations}
with  the original data $\bigFe:=\fve$, $\hve$, and $\hse$,
and complemented with  the boundary conditions \eqref{bc}.
Since in the following we will be interested in the limit
of \eqref{our-pde-van} as $\eps\downarrow 0$, for notational simplicity
we shall henceforth set $\rho=1$ in \eqref{eq:u-van}.
\paragraph{\bf Energetic solutions for the rescaled system \eqref{init-res}--\eqref{weak-heat-res}.}
For later reference in the limit passage procedure as $\eps \downarrow 0$,
we recall the defining properties of energetic solutions.
Given a quadruple of initial data $(\ures^0,\dures^0,\zres^0,\thres^0)$
satisfying
\eqref{assu-init}, a triple
$(\ures,\zres,\thres)$ is an energetic solution of the Cauchy problem for the PDE system
\eqref{our-pde-van} if
it has the regularity \eqref{reguu},
it complies with the
initial conditions
\begin{equation}
\label{init-res} \ures(0)=\ures^0\,, \quad \dures(0)=\dures^0\,,
\quad  \zres(0)=\zres^0\,, \quad \thres(0)=\thres^0 \quad \aein
\Omega\,,
\end{equation}
and fulfills
\begin{compactitem}
 \item
\emph{semistability and unidirectionality}: for a.a.\ $x\in\Om$, $\zres(\cdot,x)\colon[0,T]\to[0,1]$
is nonincreasing and for all $t\in [0,T]$
\begin{equation}
 \label{semistab-general-res} \forall\, \tilde{z}\in
\calZ\,, \ \tilde{z}\le \zres(t)\colon\quad
\E_\eps(t,\ures(t),\zres(t))\le\E_\eps(t,\ures(t),\tilde z)+\calR_1(\zres(t)-\tilde z) \,,
\end{equation}
with the mechanical energy
\begin{equation}
\label{mech-en-eps}
\E_\eps(t,u,z):=\integ{\Om}{(\tfrac12 \CC(z)e(u)   \psm
e(u)+G(z,\nabla z))}{x} -
\pairing{}{H^1_{\mathrm{D}}(\Om;\R^d)}{\bigF_\eps(t)}{u} \,;
\end{equation}
\item
\emph{weak formulation of the momentum equation}: for all test functions
$v \in L^2(0,T; H_{\mathrm{D}}^{1}(\Omega;\R^d)) \cap W^{1,1}(0,T;L^2(\Omega;\R^d))$
 and for all  $t \in[0,T]$ 
\begin{equation}
\label{e:weak-momentum-variational-res}
\begin{aligned}
& \eps^2\!\integ{\Omega}{\dures(t) \ps v(t)}{x}
- \eps^2\!\integlin{0}{t}{\!\integ{\Omega}{\dures \ps \dot v}{x}}{t}
+\integlin{0}{t}{\!\integ{\Omega}{\left(\eps\,\DD(\zres,\thres)e(\dures) + \CC(\zres)  e(\ures)
-\thres\,\BB \right) \psm e(v)}{x}}{s}
\\
&= \eps^2\!\integ{\Omega}{\dures^0 \ps v(0)}{x}
+ \integlin{0}{t}{\pairing{}{H^1_\mathrm{D}(\Om;\R^d)}{\bigFe}{v}}{s} \,;
\end{aligned}
\end{equation}
\item
\emph{mechanical energy equality}: for all  $t \in[0,T]$
\begin{equation}
\label{mech-energy-eq-res}
\begin{aligned}
&  \tfrac{\eps^2}2\!{\integ{\Om}{\mod{\dot \ures(t)}^2}{x}}+
\E_\eps(t,\ures(t),\zres(t))+ \integ{\Om}{(\zres^0{-}\zres(t))}{x}
+\integlin{0}{t}{\!\integ{\Omega}{\left(\eps\,\DD(\zres,\thres)e(\dures) 
{-} \thres \,\BB \right)\psm e(\dures)}{x}}{s}
\\
&= \tfrac{\eps^2}2\!{\integ{\Om}{\mod{\dures^0}^2}{x}}+
\E_\eps(0,\ures^0,\zres^0)
+\int_0^t\partial_t\E_\eps(s,u(s),z(s))\,\mathrm{d}s
\,;
\end{aligned}
\end{equation}
 \item
\emph{weak  formulation of the heat equation}: for all  $t \in[0,T]$
\begin{equation}
\label{weak-heat-res}
\begin{aligned}
& \eps\pairing{}{W^{2,d+\delta}}{\thres(t)}{\testw(t)}
-\eps\!\integlin{0}{t}{\integ{\Om}{\thres\,\dot\testw}{x}}{s}
+\tfrac{1}{ \eps^\beta }\!\integlin{0}{t}{\integ{\Omega}{\KK(\thres,\zres)\nabla\thres\cdot\nabla \testw}{x}}{s}
\\
&=\eps\!\integ{\Omega}{\thres^0\,\testw(0)}{x}
+\integlin{0}{t}{\integ{\Omega}{\left(\eps^2\DD(\zres,\thres) e(\dures)\psm e(\dures)
- \eps\thres\,\BB\psm e(\dures) \right)\testw}{x}}{s}\\
&\phantom{=}+\eps\!\integlin{0}{t}{\!\integ{\Om}{\testw\mod{\dzres}}{x}}{s}
+\integlin{0}{t}{\!\integ{\partial\Om}{\hse\,\testw}{\acca^{d-1}(x)}}{s}
+\integlin{0}{t}{\!\integ{\Om}{\hve\,\testw}{x}}{s}
\end{aligned}
\end{equation}
for all test functions $\eta\in H^1(0,T;L^2(\Omega))\cap \mathrm{C}^0(0,T;W^{2,d+\delta}(\Omega))$
(recall that $\mod{\dzres}$ denotes the total variation measure of $\zres$).
\end{compactitem}
\begin{rmk}
Let us also observe that
testing \eqref{weak-heat-res} by $\tfrac1\eps$ and
summing up with  \eqref{mech-energy-eq-res} leads to the rescaled total energy equality
\begin{equation}
\label{1750}
\begin{aligned}
 &  \tfrac{\eps^2}2\!{\integ{\Om}{\mod{\dures(t)}^2}{x}}+
\E_\eps(t,\ures(t),\zres(t)) + \int_\Omega \thres(t) \dd x 
\\
& = \tfrac{\eps^2}{2}\!{\integ{\Om}{\mod{\dures^0}^2}{x}}+
\E_\eps(0,\ures^0,\zres^0)   + \int_\Omega \thres^0 \dd x 
\\
& \phantom{=}  +\int_0^t\partial_t\E_\eps(s,\ures(s),\zres(s))\,\mathrm{d}s
+\tfrac1\eps\!\integlin{0}{t}{\!\integ{\partial\Om}{\hse}{\acca^{d-1}(x)}}{s}
+\tfrac1\eps\!\integlin{0}{t}{\!\integ{\Om}{\hve}{x}}{s} \,.
\end{aligned}
\end{equation}
\end{rmk}
%
\subsection{A priori estimates uniform with respect to $\eps$}
\label{ss:6.2}
%
As done in the proof of Theorem \ref{thm:main}, we shall derive the basic
a  priori  estimates on the rescaled solutions
${(\ures,\zres,\thres)}_\eps$ from the total energy equality
\eqref{1750}. Therefore, it is clear that we shall have to assume
that the families of data ${(\hve)}_\eps$ and ${(\hse)}_\eps$
converge to zero in the sense that there exists $C>0$ such that for
all $\eps>0$
\begin{equation}
\label{Heps-heps}
\integlin{0}{t}{\!\integ{\Om}{\hve}{x}}{s} \leq
C\eps\,, \qquad
\int_0^t\!\int_{\partial\Om}\hse\,\mathrm{d}\calH^{d-1}(x)\,\mathrm{d}s
\leq C\eps\,.
\end{equation}
Furthermore,  we shall suppose that  there exists $\bigF$ such that
\begin{equation}
\label{weak-limit-bigF} \bigFe \to \bigF  \qquad
\text{ in } H^1(0,T; H^1_{\mathrm{D}}(\Omega;\R^d)^*) \,.
\end{equation}

We are now in a position to derive a  priori bounds on the
rescaled solutions ${(\ures,\zres,\thres)}_\eps$, uniform with respect
to $\eps>0$. These estimates are the
time-continuous counterpart of the
\emph{First}--\emph{Third a priori estimates}
in the proof of Proposition \ref{Apriori}.
 Actually, the  calculations underlying the \emph{Second} and \emph{Third} estimates
  can be performed only
formally, when arguing on the energetic formulation of system
\eqref{our-pde-van}. Indeed, 
these computations are based on testing the weak  heat equation \eqref{weak-heat-res}  by $ \thres^{\alpha-1}$, which is not admissible since 
 $\thres^{\alpha-1} \notin \mathrm{C}^0
([0,T]; W^{2,d+\delta}(\Omega))$.

That is why Proposition \ref{Apriori-eps} below will be stated
not for \emph{all} energetic solutions to the rescaled system
\eqref{our-pde-van}, but just for
those arising  from the discrete solutions
 to \eqref{our-pde-van} 
constructed in  Section \ref{ss:4.1}.
More precisely, we shall call ``approximable solution''
to the rescaled system \eqref{our-pde-van}
any triple obtained in the time-discrete to continuous limit,
for which convergences \eqref{convs} of Proposition \ref{Conv} hold;   in Section \ref{s:5} we have shown
that any approximable solution is an energetic solution. 
Now, it can be checked that some of
the a  priori estimates on the discrete solutions in Proposition
\ref{Apriori} (i.e.\ those corresponding to \eqref{apriori-eps} below)
 are uniform with respect to $\tau$ \emph{and} $\eps$
as well. Therefore, Proposition \ref{Conv} ensures that they are
inherited by   the ``approximable'' solutions in the limit
$\tau\downarrow 0$, still  uniformly with respect to $\eps$. 
\par
Nonetheless,
to simplify the exposition,
 in the proof of Prop.\ \ref{Apriori-eps}  we will no longer work on the time-discrete scheme but rather develop the calculations directly
 (and sometimes only formally)
 on the time-continuous level.  
\par
%


%
%
\begin{prop}[A priori estimates]
\label{Apriori-eps} Assume
\eqref{ass-dom}--\eqref{assG}, \eqref{scaling-K}  with $\beta>0$,  
${(\hve)}_\eps \subset L^1 (0,T; L^1(\Omega)) \cap L^2 (0,T; H^1(\Omega)^*)$,
${(\hse)}_\eps   \subset L^1 (0,T; L^2(\partial\Omega))$ fulfill
\eqref{Heps-heps}, and 
${(\bigFe)}_\eps \subset H^1 (0,T; H_{\mathrm{D}}^1(\Omega;\R^d)^*)$  comply with
\eqref{weak-limit-bigF}.
In addition to \eqref{assu-init}, let the family
of initial data ${(\ures^0,\dures^0,\zres^0,\thres^0)}_\eps$
fulfill
\begin{equation}
\label{family-eps-below}
| \E_\eps(0,\ures^0,\zres^0)| + \eps\|\dures^0
\|_{L^2(\Omega;\R^d)}+ \|\thres^0 \|_{L^1(\Omega)} 
\leq C
\end{equation}
for a constant $C$ independent of $\eps$. 
Let ${(u_\eps,z_\eps,\theta_\eps)}_\eps$
be a family of  approximable solutions to system \eqref{our-pde-van}.
Then, there exists a constant
$C>0$ such that the following estimates hold for all  $\eps>0$:
\begin{subequations}
\label{apriori-eps}
\begin{align}
\label{apDu-eps}
\|u_\eps\|_{L^\infty(0,T;H_{\mathrm{D}}^1 (\Omega;\R^d))}&\leq C\,,\\ \displaybreak[0]
\label{ap-u-lin-linfty-eps}
\eps \| \dot{u}_\eps \|_{L^\infty (0,T; L^2(\Omega;\R^d))} &\leq C\,,\\ \displaybreak[0]
\label{zeps-bv}
\calR_1 (\zres(T) - \zres^0) &\leq C\,,\\ \displaybreak[0]
\label{ap-z-infty-eps}
 \| {z}_\eps \|_{L^\infty ((0,T) \times \Omega)}
&\leq 1\,,
\\ \displaybreak[0]
\label{ap-z-w1q-eps} \| {z}_\eps \|_{L^\infty (0,T;
W^{1,q}(\Omega))}&\leq C\,,
\\ \displaybreak[0]
\label{aptheta1-eps}
\|{\theta}_\eps\|_{L^\infty(0,T;L^1(\Omega))}&\leq  C\,, \\ \displaybreak[0]
\label{apthetanabla-eps}
 \| \nabla{\theta}_\eps\|_{L^2(0,T;L^2(\Omega;\R^d))}&\leq  C  \eps^{\beta/2}  \,,
\\ \displaybreak[0]
\label{apthetaH12-eps}
 \|{\theta}_\eps\|_{L^2(0,T;H^1(\Omega))}&\leq  C\,,
\\ \displaybreak[0]
\label{apthetaLp-eps}
\|{\theta}_\eps\|_{L^p((0,T)\times\Omega)}&\leq C \quad\text{for any
}p\in\left\{
\begin{array}{ll}
[1,8/3]&\text{if }d{=}3\,,\\
{[1,3]}&\text{if }d{=}2\,,
\end{array}
\right.
\end{align}
\end{subequations}
with $\calR_1$ from \eqref{dissip-potential}.
\end{prop}
\begin{proof}[Sketch of the proof]
{\bf First a priori estimate: ad \eqref{apDu-eps},
\eqref{ap-u-lin-linfty-eps}, \eqref{zeps-bv}, \eqref{ap-z-infty-eps}
\eqref{ap-z-w1q-eps}, \eqref{aptheta1-eps}: }
Estimate \eqref{ap-z-infty-eps} is obvious.
Estimate \eqref{zeps-bv} follows from the definition of $\calR_1$, \eqref{Gind}, and \eqref{hyp-init},
and the fact that the functions $z_\eps(\cdot,x)$ are nonincreasing.
We start from the total energy equality \eqref{1750}.
Also thanks to \eqref{weak-limit-bigF}, the energies $\E_\eps$ enjoy the
coercivity property \eqref{est-mechen}  with constants independent
of $\eps$. Therefore,  relying on the uniform bound \eqref{weak-limit-bigF}
for $\dot{\bigF}_\eps$, and using that $\thres >0$ a.e.\ in $(0,T)\times \Omega$
for every $\eps>0$,
one can repeat the very same calculations as
in the first step of the proof of Proposition \ref{Apriori}, and conclude
that the left-hand side of \eqref{1750} is uniformly bounded from
above and from below, whence
\eqref{apDu-eps},
\eqref{ap-u-lin-linfty-eps},
\eqref{ap-z-w1q-eps}, \eqref{aptheta1-eps}.

{\bf Second and third a priori estimates: ad \eqref{apthetanabla-eps},  \eqref{apthetaH12-eps},
and \eqref{apthetaLp-eps}: }
We (formally) test \eqref{weak-heat-res} by $\thres^{\alpha-1}$,
integrate in time,
and arrive at the (formally written) analogue of \eqref{tricky}, viz.\
\begin{equation}
\label{calc-eps-1}
\begin{aligned}
& \tfrac{c}{ \eps^\beta } \!\int_0^t \!\int_\Omega  \KK(\zres,\thres) \nabla
(\thres^{\alpha/2} )\ps \nabla (\thres^{\alpha/2})  \dd x \dd s  +
\eps^2\! \int_0^t \!\int_\Omega \DD(\zres,\thres) e(\dures)  \psm e(\dures)
\thres^{\alpha-1} \dd x \dd s
\\ & \phantom{=}
+\eps \!\int_0^t \!\int_\Omega \thres^{\alpha-1} |\dot{z}_\eps|  \dd x
\dd s
+
\int_0^t \!\int_{\partial\Omega}\hse \thres^{\alpha-1}
\,\mathrm{d}\mathcal{H}^{d-1}(x)\dd s
+ \int_0^t \!\int_\Omega \hve \thres^{\alpha-1} \dd x \dd s \\
& = \eps \!\int_0^t
\!\int_\Omega\dot{\theta}_\eps\thres^{\alpha-1}\,\mathrm{d}x  \dd s
+\eps  \!\int_0^t  \!\int_\Omega \thres\,\BB  \psm e(\dot u_\eps)
\thres^{\alpha-1}\,\mathrm{d}x \dd s \doteq I_1 +I_2\,.
\end{aligned}
\end{equation}
As in the proof of Proposition \ref{Apriori}, we estimate
\begin{equation}
\label{calc-eps-2}
I_1 = \eps  \!\int_\Omega  \tfrac{(\thres(t))^\alpha }\alpha \dd x
-  \eps \!\int_\Omega   \tfrac{(\thres^0)^\alpha}\alpha  \dd x\,,
\end{equation}
whereas we estimate $I_2= \iint \eps \thres \,\BB \psm e(\dot u_\eps) \thres^{\alpha-1}$ by
\begin{equation}
\label{calc-eps-3}
I_2 \leq \eps^2 \tfrac{C_{\DD}^1}2  \int_0^t \!\int_\Omega  |e(\dures)|^2 \thres^{\alpha-1}  \dd x \dd s +
C  \int_0^t  \!\int_\Omega |\thres|^2 \thres^{\alpha-1}  \dd x \dd s\,,
\end{equation}
where the constant $C$ subsumes  the norm $|\BB|$  as well. Combining \eqref{calc-eps-1}--\eqref{calc-eps-3}
and then arguing exactly in the same way as in the proof of Proposition\
\ref{Apriori}, we end up with
the analogue of \eqref{tricky-bis}, i.e.,
\begin{equation*}
\tfrac{1}{ \eps^\beta }
\int_0^{t}\!\!\! \int_\Omega\!\!  \KK (\zres,\thres)
\nabla(\thres^{\alpha/2}) \cdot \nabla(\thres^{\alpha/2})   \dd x \dd s
+\int_\Omega\tfrac{\eps}\alpha (\thres^0)^\alpha \dd x
\leq   \int_\Omega\!\!  \tfrac{\eps}\alpha (\thres(t))^\alpha  \dd x +
 C  \int_0^{t}\!\!\!
\int_\Omega \thres^{\alpha+1} (s) \dd x \dd s\,,
\end{equation*}
whence
$
\tfrac1{ \eps^\beta }\int_0^T \int_\Omega \KK (\zres,\thres)
\nabla(\thres^{\alpha/2}) \cdot \nabla(\thres^{\alpha/2} ) \dd x \dd t
\leq C $.
From this, with the same arguments as in the third step of the
proof of Proposition \ref{Apriori},  cf.\ \eqref{171207},  we infer  that
\[
\int_0^T \!\int_\Omega |\nabla \thres|^2 \dd x \dd t \leq C \eps^\beta ,
\]
i.e.\ \eqref{apthetanabla-eps}.
Then, \eqref{apthetaH12-eps} follows from  \eqref{apthetanabla-eps}
and \eqref{aptheta1-eps}, via the Poincar\'e inequality.
Finally, \eqref{apthetaLp-eps} ensues by interpolation, as in the proof of Proposition \ref{Apriori}.
\end{proof}
Observe that in the proof of Proposition \ref{Apriori-eps}  we have not been
able to repeat the calculations in the  \emph{Fourth and Fifth estimates}, 
cf.\ the proof of Proposition \ref{Apriori}.
In particular, from the mechanical energy equality \eqref{mech-energy-eq-res} we  have not
been able to deduce an estimate for $\eps^{1/2} e(\dures)$ in
$L^2(0,T; L^2(\Omega;\R^{d \times d}))$,  since we cannot bound the term
$\int_0^t \int_\Omega \thres  \colon e(\dures) \dd x \dd s$ on the
right-hand side of  \eqref{mech-energy-eq-res}.
Therefore, in the proof
of our convergence result for vanishing viscosity and inertia, Theorem \ref{Thm6.1}
below, we shall have to resort to careful arguments
in order to handle the terms containing  $e(\dures)$, in the passage to the limit
in the momentum equation and mechanical energy equality, cf.\
\eqref{omg}--\eqref{halle}. In particular, differently from Proposition \ref{Apriori},
for a vanishing sequence ${(\eps_n)}_n$ the convergences
\begin{equation}
\label{Strongly}
\begin{split}
&\eps_n e(\duresn) \to 0 \quad \text{\emph{strongly} in } L^2(0,T;L^2(\Omega;\R^{d \times d}))
\quad\text{ and }\quad
\int_0^t\!\int_\Omega\theta_{\eps_n}:e(\dot u_{\eps_n})\,\mathrm{d}x\,\mathrm{d}s
\to0\,,\\
&\theta_\eps\to\Theta\quad \text{\emph{strongly} in } L^2(0,T)\times\Omega)
\end{split}
\end{equation}
will now be extracted from the  weak heat equation  \eqref{weak-heat-res}, 
using integration by parts and the information that $\Theta$ is constant in space.
It is in this connection that we need to further assume
homogeneous Dirichlet boundary conditions for the displacement
on the whole boundary $\partial \Omega$, cf.\ \eqref{Dir-b-c}.
%
\subsection{Convergence to local solutions of the rate-independent limit system}
\label{ss:6.3}
%
Let us mention in advance that in Theorem \ref{Thm6.1} we will prove that, up to a subsequence, the functions
$(\ures,\zres,\thres)$ converge to a limit triple $(u,z,\Theta)$ such that $\Theta$ 
is spatially constant.
As we will see, the pair $(u,z)$ fulfills the (pointwise-in-time) \emph{static}
momentum balance (i.e.\ without viscosity and inertia), a semistability condition
with respect to the energy $\calE$ arising from $\calE_\eps$ \eqref{mech-en-eps} in the
limit $\eps\downarrow 0$, and an energy inequality,
where the viscous, the inertial, and the thermal expansion  contributions are no longer present.
This inequality holds on $[0,t]$ \emph{for every} $t\in [0,T]$
in the general case, and on $[s,t]$ for all
$t \in [0,T]$ and almost every $s\in (0,t)$, under a
further condition on the gradient term in the energy~$\calE$, i.e.\ that $q>d$.
Indeed, the three properties (momentum balance, semistability, energy inequality)
constitute the notion of \emph{local solution} \cite{Miel08?DEMF, Roub13ACVB,RoThPa13SDLS}
to the rate-independent system driven by $\calR_1$ and $\calE$.
Observe that, in fact,  the spatially constant $\Theta$ does not appear in these relations, because
it contributes with a zero term to the momentum balance.

Moreover, testing the weak heat equation \eqref{weak-heat-res} with
functions $\eta$  that are  constant in space 
(which is the property of the limit temperature $\Theta$ by \eqref{apthetanabla-eps})
and taking into account the bounds
\eqref{Heps-heps}, \eqref{family-eps-below}, \eqref{aptheta1-eps},
and convergence \eqref{Strongly},   we find   in the limit relation $0=0$.
  This shows   that the temporal evolution of $\Theta$ 
is irrelevant in the rate-independent limit model.    In fact, in order to gain insight into the time evolution of  $\Theta$, we will perform
the limit passage in the heat equation \eqref{weak-heat-res} rescaled by the factor $1/\eps$ and
tested by $\eta\in H^1(0,T)$, constant in space.
 In this way, the heat-transfer term involving $\KK_\eps=\frac1{\eps^\beta}\KK$ will  disappear. 
    This will lead to   an ODE for the limit function $\Theta$, cf.\ \eqref{levico}.  Such an ODE 
 involves a \emph{defect} measure $\mu$, i.e.\ a Radon measure on $[0,T]$ arising in the limit of 
 the viscous dissipation term
 $\|\eps \DD(\zres,\thres)e(\dures):e(\dures)\|_{L^1(\Omega)}$, see \eqref{171208} below.

In the following proof, notice that Steps 0--3 can be proven for $\beta>0$, while in Step 4 we need $\beta\ge2$.
 Furthermore,  the   condition that  the tensor $\mathbb{B}$ is constant in space 
will have a crucial role in  handling    the thermal expansion term 
$\theta_{\eps}\,\BB:e(\dot{u}_\eps)$ in the rescaled heat equation, cf.\ \eqref{pistar} ahead.

\begin{thm}
\label{Thm6.1}
Assume \eqref{ass-dom}--\eqref{mono}, \eqref{assG}, \eqref{ass-data},
and, in addition, let
\eqref{Dir-b-c}, \eqref{scaling-K}  with $\beta\ge2$,  \eqref{Heps-heps},  and \eqref{weak-limit-bigF} be satisfied.
Let the initial data  ${(\ures^0,\dures^0,\zres^0,\thres^0)}_\eps$ fulfill
\eqref{assu-init}, \eqref{family-eps-below},
\begin{equation}
\label{converg-dureszero}
\eps \dures^0 \to 0 \quad \text{in } L^2(\Omega;\R^d)\,,
\end{equation}
and suppose that there exist 
$u_0 \in H_{\mathrm{D}}^1(\Omega;\R^d) $  and $z_0 \in \calZ$ such that
\begin{equation}
\label{convergence-initial-energies}
\ures^0 \weakto u_0 \text{ in } H_{\mathrm{D}}^1(\Omega;\R^d), \qquad
\zres^0 \weakto z_0 \text{ in } \calZ, \qquad 
 \E_\eps(0,\ures^0,\zres^0) \to
\E (0,u_0,z_0) \qquad \text{as } \eps \downarrow 0\,,
\end{equation}
with $\E_\eps$ as in \eqref{mech-en-eps}. 

Then, the functions ${(u_\eps,\zres,\thres)}_\eps$ converge
 (up to subsequences)
to a triple $(u,z,\Theta)$ such that
\begin{equation}
\label{limit-triple}
\begin{gathered}
u \in L^\infty (0,T; H^1_\mathrm{D}(\Omega;\R^d))\,, \qquad
z \in L^\infty (0,T; W^{1,q}(\Omega)) \cap L^\infty ((0,T)\times \Omega) \cap \BV ([0,T]; L^1(\Omega))\,,
\\
\Theta\text{ is constant in space and }\;
\Theta  \in
L^p (0,T) \quad\text{for any }p\in\left\{
\begin{array}{ll}
[1,8/3]&\text{if }d{=}3\,,\\
{[1,3]}&\text{if }d{=}2\,.
\end{array}
\right.
\end{gathered}
\end{equation}
The pair $(u,z)$ fulfills the unidirectionality as well as
\begin{compactenum}
\item
the semistability condition \eqref{semistab-general}
for  all $t\in [0,T]$, with the mechanical energy $\E$ defined as in~\eqref{mech-en-eps}
with $\bigFe$ replaced by
the weak limit $\bigF$  of the
sequence ${(\bigFe)}_\eps$, see~\eqref{weak-limit-bigF};
\item the weak momentum balance for  all $t\in [0,T]$
\begin{equation}
\label{stat-mombal} \int_\Omega  \CC(z(t)) e(u(t)) \psm e(v) \dd x =
\pairing{}{H^1_\mathrm{D}(\Om;\R^d)}{\bigF(t)}{v} \qquad\text{for all } v
\in H^{1}_\mathrm{D}(\Omega;\R^d)\,;
\end{equation}
\item 
the mechanical energy inequality for  all $t \in [0,T]$
\begin{equation}
\label{limit-mech-energy}
\begin{aligned}
&
\E(t,u(t),z(t))  + \integ{\Om}{(z(0){-}z(t))}{x} \leq \E(0,u(0),z(0))
+\int_0^t\partial_t\E(r,u(r),z(r))\,\mathrm{d}r\,;
\end{aligned}
\end{equation}
\end{compactenum}
If in addition the function $G$ fulfills the growth condition \eqref{G-growth} with $q>d$,
then $(u,z)$ also fulfill
\begin{equation}
\label{limit-mech-energy-better}
\begin{aligned}
&
\E(t,u(t),z(t))  + \integ{\Om}{(z(s){-}z(t))}{x} \leq \E(s,u(s),z(s))
+\int_s^t\partial_t\E(r,u(r),z(r))\,\mathrm{d}r
\end{aligned}
\end{equation}
for all $t\in [0,T]$ and for almost all $s\in (0,t)$.
\par
Moreover,
assume in addition that there exists $\widetilde H\in L^1(0,T)$ such that 
\begin{equation}
\label{additional-heat-sources}
\tfrac{1}{\eps}(\|H_\eps\|_{L^1(\Omega)}+\|h_\eps\|_{L^1(\partial\Omega)})
\rightharpoonup \widetilde H \quad \text{ in $L^1(0,T)$}. 
\end{equation}
Then, 
$\Theta$ fulfills
\begin{equation}
\label{levico}
\testw(t)\int_\Omega\Theta(t)  \dd x 
-\integlin{0}{t}{\dot\testw\integ{\Om}{\Theta}{x}}{s}
-\testw(0)\integ{\Omega}{\Theta(0)}{x}
=\int_0^t\eta\,\mathrm{d}\mu(s)
+\integlin{0}{t}{\testw\integ{\Om}{\mod{\dot z}}{x}}{s}
+\int_0^t\widetilde H\,\eta\,\mathrm{d}s\
\end{equation}
for a.a.\ $t\in(0,T)$ and for every $\eta\in H^1(0,T)$ constant in space, with 
the \emph{defect}  measure $\mu$ given by 
\begin{equation}\label{171208}
\|\eps \DD(\zres,\thres)e(\dures):e(\dures)\|_{L^1(\Omega)}\to \mu
\quad\text{ in the sense of Radon measures in $[0,T]$}\,. 
\end{equation}
\end{thm}
\begin{proof}
\textbf{Step $0$, compactness: } It follows from Proposition
\ref{Apriori-eps} that for every vanishing  sequence ${(\eps_n)}_n$
there exist a (not relabeled) subsequence and a triple $(u,z,\Theta)
$ as in \eqref{limit-triple} such  that the following convergences
hold
\begin{subequations}
\label{convs-eps}
\begin{alignat}{3}
\label{convu-0-eps}
\uresn &\weaksto\; && u&&\text{ in }L^\infty(0,T;H^1_\mathrm{D}(\Omega;\R^d)) \,,
\\
\label{convu-eps-added}
\eps_n \uresn&\weaksto&& 0&& \text{ in }W^{1,\infty}(0,T;L^2(\Omega;\R^d))\,,
\\
\label{convoz-eps}
\zresn &\weaksto&& z&&\text{ in }L^\infty(0,T;W^{1,q}(\Omega)) \cap L^\infty ((0,T) \times \Omega)\,,
\\
\label{convoz-ptw-eps-weal}
\zresn(t) &\weakto&& z(t)\; &&\text{ in }W^{1,q}(\Omega)
\quad \text{for all } t \in[0,T]
\\
\label{convoz-ptw-eps}
\zresn(t) &\to&& z(t) &&\text{ in }L^r(\Omega)
\quad \text{for all } 1 \leq r <\infty \text{ and for all } t \in[0,T]\,,
\\
\label{convotheta-eps}
\thresn&\rightharpoonup&&\Theta 
&&\text{ in }L^2(0,T;H^1(\Omega)) \cap L^p ((0,T) \times \Omega)
\text{ for all $p$ as in \eqref{apthetaLp-eps}}\,.
\end{alignat}
\end{subequations}
Indeed, \eqref{convu-0-eps} ensues from \eqref{apDu-eps}, and it gives, in particular, that
$\eps_n \uresn \to 0$ in $L^\infty(0,T;H^1_\mathrm{D}(\Omega;\R^d))$.
Then, convergence \eqref{convu-eps-added} directly follows
from estimate \eqref{ap-u-lin-linfty-eps}.
Convergences  \eqref{convoz-eps}--\eqref{convoz-ptw-eps}  ensue
from the very same compactness arguments as in the proof of
Proposition \ref{Conv}, also using the Helly Theorem.
Furthermore, \eqref{convotheta-eps} follows from estimates
\eqref{apthetaH12-eps}--\eqref{apthetaLp-eps} by weak compactness.
Observe that in view of \eqref{apthetanabla-eps} we have that
\begin{equation}
\label{true2}
\nabla \thresn \to 0 \quad \text{ in }
L^2(0,T; L^2(\Omega;\R^d))\,.
\end{equation}
Therefore, we conclude that
$ \nabla\Theta=0 $ a.e.\ in $(0,T)\times \Omega$. Since $\Theta$ is spatially constant,
hereafter we will write it as a function of the sole variable $t$.

We now prove the enhanced convergence
\begin{equation}
\label{omg}
\thresn \to \Theta\;\text{in }L^2(0,T;L^2(\Omega))\,.
\end{equation}
In fact, we use
the  Poincar\'e inequality
\[
\|\thresn-\Theta\|_{L^2(0,T;L^2(\Omega))}\leq
\|\nabla(\thresn-\Theta)\|_{L^2(0,T;L^2(\Omega;\R^d))}
+C(\Omega,T)\left|\int_0^T\!
\int_\Omega(\thresn-\Theta)\,\mathrm{d}x\,\mathrm{d}s\right|\longrightarrow0\,,
\]
where
the gradient term   tends to $0$ by \eqref{true2},
and the convergence of the second term follows from \eqref{convotheta-eps}.

Finally, let us show that
\begin{equation}
\label{strongly}
\eps_n \, e(\duresn) \to 0 \quad \text{\emph{strongly} in } L^2(0,T;L^2(\Omega;\R^{d \times d}))\,.
\end{equation}
Preliminarily, observe that,
since  $\BB$ and  the limit function $\Theta$  are constant in space, we have by
 integration by parts
\begin{equation}
\label{pistar}
\int_0^t\!\int_\Omega\Theta\,\BB:e(\duresn)\,\mathrm{d}x\,\mathrm{d}s
= \int_0^t\!\int_{\partial\Omega}\Theta\,\BB\,\nu\cdot\duresn\,\mathrm{d}\calH^{d-1}(x)\,\mathrm{d}s
-\int_0^t\!\int_\Omega\div(\Theta\,\BB)\cdot \duresn\,\mathrm{d}x\,\mathrm{d}s=0\,,
\end{equation}
where we used $\partial_\mathrm{D}\Omega =\partial\Omega$,
hence $\duresn \in L^{ 2} (0,T; H_{\mathrm{D}}^1 (\Omega;\R^d))$
implies that $\duresn =0$ a.e.\ in $(0,T) \times \partial\Omega$.
Using \eqref{pistar} in the weak heat equation
\eqref{weak-heat-res}
tested by $1$ and applying Young's inequality, we find
\begin{equation}
\label{halle}
\begin{split}
\eps_n\left( \int_\Omega (\thresn(t)-\thresn^0) \, \mathrm{d} x \right)
&\geq \int_0^t\!\int_\Omega \left[ \eps_n^2\DD (\zresn,\thresn)
 e(\duresn): e(\duresn)-\eps_n(\thresn{-}\Theta\,\BB):e(\duresn) \right]
\,\mathrm{d}x\,\mathrm{d}s\\
&\geq \int_0^t\!\int_\Omega\eps_n^2\tfrac{C_\DD}{2}|e(\duresn)|^2\,\mathrm{d}x\,\mathrm{d}s
-C \|\thresn-\Theta\|_{L^2(0,T;L^2(\Omega))}^2\,
\end{split}
\end{equation}
with $C = |\BB|/2$.
From this, taking into account that ${(\thresn^0)}_n$
is bounded in $L^1(\Omega)$ by \eqref{family-eps-below},
estimate~\eqref{aptheta1-eps} for ${(\thresn)}_n$, and
convergence \eqref{omg},
we conclude that
$\lim_{\eps_n \downarrow 0} \eps_n \| e(\duresn) \|_{L^2 (0,T; L^2(\Omega;\R^{d \times d}))} = 0$,
whence \eqref{strongly}.

In fact, by Korn's inequality we conclude  that
\begin{equation}
\label{correct?}
\eps_n \uresn \to 0 \quad \text{in } H^1 (0,T; H^1_\mathrm{D}(\Omega;\R^d))\,.
\end{equation}
\par
\textbf{Step $1$, passage to the limit in the momentum
balance \eqref{e:weak-momentum-variational-res}: }
Convergence \eqref{correct?}, joint with the boundedness
\eqref{assCD-3} of the tensor $\DD$, ensures that the  first and the
second summands on the left-hand side of
\eqref{e:weak-momentum-variational-res}  tend to zero.
Arguing as in the proof of Proposition \ref{prop:Mombal}, we show that for every test function
$v$ in \eqref{e:weak-momentum-variational-res},
$\CC(\zresn) e(v) \to \CC(z) e(v)$ in $L^2 ((0,T)\times \Omega;\R^{d \times d})$. We combine this with
\eqref{convu-0-eps}  and, also using
\eqref{convotheta-eps}, we pass
to the limit in the third term on the left-hand side
of \eqref{e:weak-momentum-variational-res},
 recalling that the fourth summand converges to zero similarly to \eqref{pistar}.
As for the right-hand side,
by \eqref{family-eps-below} we have
\begin{equation}
\label{quoted-below-zero}
\eps_n^2 \duresn^0 \to 0 \qquad \text{ in $L^2(\Omega;\R^{d})$,}
\end{equation}
hence the first term converges to zero. The second one
tends to zero for almost all $t \in (0,T)$ by \eqref{convu-eps-added}, which in particular gives
\begin{equation}
\label{quoted-below-dot}
\eps_n^2 \duresn \to 0 \qquad \text{ in $L^\infty (0,T; L^2(\Omega;\R^d)).$}
\end{equation}
For the third one, we use \eqref{weak-limit-bigF}.
We thus conclude that  \eqref{stat-mombal} holds at almost all $t \in (0,T)$.

In order to check it at \emph{every} $t\in [0,T]$, we observe
that for every $t\in [0,T]$ from the bounded sequence ${(\uresn(t))}_n$
(along which convergences \eqref{convs-eps} hold)
we can extract a subsequence, possibly depending on~$t$,
weakly converging to some $\bar{u}(t)$ in $H^1_{\mathrm{D}}(\Omega;\R^d)$.
Relying on  convergence \eqref{convoz-ptw-eps} for ${(\zresn(t))}_n$
and on~\eqref{weak-limit-bigF} for $(\bigF_{\eps_n}(t))$,
with the same arguments as above we conclude that
$\int_\Omega  \CC(z(t)) e(\bar{u}(t)) \psm e(v) \dd x =
\pairing{}{H^1_\mathrm{D}(\Om;\R^d)}{\bigF(t)}{v} $ for all $ v
\in H^{1}_\mathrm{D}(\Omega;\R^d)$. Since this equation has a unique solution, we conclude that
$\bar u(t) = u(t)$ for almost all $t\in (0,T)$, and that
the \emph{whole} sequence $\uresn(t)$ weakly converges to $\bar u(t)$ for every $t\in [0,T]$.
In this way  $u$ extends to a function defined on $[0,T]$,
such that
\begin{equation}
\label{weak-ptw-u}
\uresn(t) \weakto u(t) \quad \text{in } H^{1}_\mathrm{D}(\Omega;\R^d) \quad \text{for all } t \in [0,T]\,,
\end{equation}
solving  \eqref{stat-mombal} at all $t \in [0,T]$.

\par
\textbf{Step $2$, enhanced convergences for ${(\uresn)}_n$: }
As a by-product of this limit passage, we also extract
convergences  \eqref{enhanced-below} and \eqref{halleluja} below for ${(\uresn)}_n$,
which we will then use in the
passage to the limit in the semistability and in the mechanical energy inequality.
Indeed, we test \eqref{e:weak-momentum-variational-res} by $\uresn$,
thus obtaining
\[
\begin{aligned} 
& 
\limsup_{n \to \infty}
\integlin{0}{t}{\!\integ{\Omega}{\left( \CC(\zresn)  e(\uresn) {-}\thresn\,\BB\right)
\psm e(\uresn)}{x}}{s}
\\
&
\leq   \limsup_{n \to \infty} \eps_n^2\integlin{0}{t}{\!\integ{\Omega}{|\duresn|^2}{x}}{t}
- \liminf_{n\to \infty} \integlin{0}{t}{\!\integ{\Omega}{\eps_n\DD(\zresn,\thresn)e(\duresn)
\psm e(\uresn)}{x}}{s}
\\
& \phantom{\le}
+ \limsup_{n\to\infty}\eps_n^2\integ{\Omega}{\duresn^0 \ps \uresn^0}{x}
-  \liminf_{n\to \infty}  \eps_n^2\integ{\Omega}{\duresn(t) \ps \uresn(t)}{x}
+ \limsup_{n\to\infty}\integlin{0}{t}{\pairing{}{H^1_\mathrm{D}(\Om;\R^d)}{\bigF_{\eps_n}}{\uresn}}{s}
\\
&
=0+0+0+0 + \integlin{0}{t}{\pairing{}{H^1_\mathrm{D}(\Om;\R^d)}{\bigF}{u}}{s}
= \integlin{0}{t}{\!\integ{\Omega}{ \CC(z)  e(u) \psm e(u)}{x}}{s}
\end{aligned}
\]
where the first term
 in the right-hand side
converges to zero thanks to \eqref{correct?}, the second one by
the boundedness of $\DD$,
\eqref{convu-0-eps},  and \eqref{correct?}, the third one by \eqref{quoted-below-zero}
combined with the boundedness of ${(\uresn^0)}_n$, the fourth one by
\eqref{convu-0-eps} and  \eqref{quoted-below-dot}.
The fifth term passes to the limit by \eqref{weak-limit-bigF} and \eqref{convu-0-eps}.
The last identity follows from  \eqref{stat-mombal}. 
Remark that the second term in the left-hand side converges to zero by
\eqref{convu-0-eps} and \eqref{convotheta-eps}, as done for \eqref{pistar}.

From the above chain of inequalities we thus obtain that
\[
\limsup_{n \to \infty} \integlin{0}{t}{\!\integ{\Omega}{ \CC(\zresn) e(\uresn)
\psm e(\uresn)}{x}}{s}
\leq
\integlin{0}{t}{\!\integ{\Omega}{  \CC(z) e(u)  \psm   e(u)}{x}}{s}.
\]
Next, we may apply Lemma~\ref{appendix}
to deduce that $e(\uresn)$ strongly converges
to $e(u)$ in $L^2( (0,T) \times \Omega;\R^{d \times d}) $, see also Lemma \ref{strongerconv}. Hence, by Korn's inequality, we ultimately infer
\begin{equation}
\label{halleluja}
\uresn \to u \quad \text{ in } L^2 (0,T;H^1_\mathrm{D}(\Omega;\R^d))\,.
\end{equation}
For later convenience, we observe that, in particular, this yields
\begin{equation}
\label{enhanced-below}
\integ{\Omega}{ \CC(\zresn(t))   e(\uresn(t))  \psm e(\uresn(t))}{x} \to
\integ{\Omega}{ \CC(z(t))   e(u(t))  \psm e(u(t))}{x} \qquad \foraa t \in (0,T)\,.
\end{equation}
\par
\textbf{Step $3$, passage to the limit in the
semistability condition: }
In view of the pointwise convergences
 \eqref{convoz-ptw-eps-weal}--\eqref{convoz-ptw-eps}
for $\zresn$ and
$\uresn(t) \to u(t)$ in $H^1_\mathrm{D}(\Omega;\R^d)$
(by \eqref{halleluja}) for  all $t\in [0,T]$, we may apply the mutual recovery sequence
construction from Theorem \ref{MRS}
in order to pass to the limit as $\eps_n \downarrow 0$ in the semistability \eqref{semistab-general-res}.
Also taking into account convergence \eqref{weak-limit-bigF} for
${(\bigF_{\eps_n})}_n$, we  conclude
that $(u,z)$ comply with the semistability condition \eqref{semistab-general} for every $t\in [0,T]$.
\par
\textbf{Step $4$, passage to the limit in the mechanical energy inequality on $(0,t)$: }
By lower semicontinuity it follows from
convergences \eqref{weak-limit-bigF}, \eqref{weak-ptw-u},
\eqref{convoz-ptw-eps-weal}, and  \eqref{convoz-eps} that
\begin{equation}
\label{energy-liminf}
\liminf_{n \to \infty}\E_{\eps_n} (t,\uresn(t),\zresn(t)) \geq \E(t,u(t),z(t))
\qquad \text{for all } t \in [0,T]\,.
\end{equation}
Furthermore,
combining \eqref{weak-limit-bigF} with \eqref{convu-0-eps} we infer that
\begin{equation}
\label{power-cvrg}
\partial_t \E_{\eps_n} (t,\uresn,\zresn) =
- \pairing{}{H^1_\mathrm{D}(\Om;\R^d)}{\dot{\bigF}_{\eps_n}(t)}{\uresn}  \to
- \pairing{}{H^1_\mathrm{D}(\Om;\R^d)}{\dot{\bigF}(t)}{u} = \partial_t \E(t,u,z)
\quad \text{ in } L^2 (0,T) \,.
\end{equation}
We are now in a position to pass to the limit in the mechanical
energy inequality  \eqref{mech-energy-eq-res}.
 We notice that the first  term
on the left-hand side of \eqref{mech-energy-eq-res} is positive. 
For the second one we use \eqref{energy-liminf}
and the  third  one converges to $\int_\Omega (z(0)- z(t)) \dd x $ by
\eqref{convoz-ptw-eps}.
The fourth one,  given by
$$
\integlin{0}{t}{\!\integ{\Omega}{\left(\eps\,\DD(\zres,\thres)e(\dures) 
{-} \thres \,\BB \right)\psm e(\dures)}{x}}{s},
$$ 
is bounded from below by
$$
-\int_0^t\!\int_\Omega\thresn\,\BB:e(\duresn)\,\mathrm{d}x\,\mathrm{d}s.
$$
We can again argue as in \eqref{pistar}
\begin{equation}
\label{bella-1}
\begin{aligned}
\int_0^t\!\int_\Omega\thresn\,\BB:e(\duresn)\,\mathrm{d}x\,\mathrm{d}s
& = \int_0^t\!\int_{\partial\Omega}\thresn\,\BB\,\nu\cdot\duresn\,\mathrm{d}\calH^{d-1}(x)\,\mathrm{d}s
-\int_0^t\!\int_\Omega\div(\thresn\,\BB)\cdot \duresn\,\mathrm{d}x\,\mathrm{d}s
\\
& =0 -\int_0^t\!\int_\Omega\div(\thresn\,\BB)\cdot \duresn\,\mathrm{d}x\,\mathrm{d}s\,,
\end{aligned}
\end{equation}
where we have used that $\duresn$ complies with homogeneous Dirichlet conditions
on $\partial_{\mathrm{D}}\Omega=\partial\Omega$,
and then observe that
\begin{equation}
\label{bella-2}
\| \div(\thresn\,\BB)\cdot \duresn \|_{L^1 ((0,T) \times \Omega)}
= \|\eps_n^{-1} \div(\thresn\,\BB)\cdot \eps_n\duresn \|_{L^1 ((0,T) \times \Omega)}
\leq C\|\eps_n\duresn \|_{L^2 ((0,T) \times \Omega)} \to 0\,,
\end{equation}
due to estimate \eqref{apthetanabla-eps} and \eqref{correct?}.
 Notice that here we have used the fact that $\beta\ge2$;
this is the only point where we use such requirement. 
As for the right-hand side, we observe that
the first term  converges to zero
by \eqref{converg-dureszero}.
The second term passes to the limit by  the convergence
\eqref{convergence-initial-energies} for the initial energies, and the  third  one by \eqref{power-cvrg}.

 Therefore we conclude that
$$
\E(t,u(t),z(t))+\int_\Omega (z(0)- z(t)) \dd x\le \E(0,u(0),z(0))+\int_0^t\partial_t \E(s,u,z)\mathrm{d}s\,.
$$
\par
\textbf{Step $5$, case $q>d$, enhanced convergence for $(\zresn)$ and energy convergence: }
We now prove that
\begin{equation}
\label{vorrei}
\lim_{n \to \infty} \int_\Omega G(\zresn(t),\nabla \zresn(t)) \dd x
=  \int_\Omega G(z(t),\nabla z(t)) \dd x \qquad \foraa t \in (0,T)\,,
\end{equation}
which, combined with  \eqref{weak-limit-bigF},
\eqref{enhanced-below} and \eqref{halleluja} will yield the pointwise convergence of the energies
\begin{equation}
\label{energy-convergence}
\lim_{n \to \infty}\E_{\eps_n} (t,\uresn(t),\zresn(t)) =  \E(t,u(t),z(t))
\qquad \foraa t \in (0,T)\,.
\end{equation}

We obtain \eqref{vorrei} testing semistability \eqref{semistab-general-res}
by a suitable recovery sequence ${(\tilde{z}_{\eps_n} )}_n $  for  $\tilde z = z(t)$;
in the following lines, to avoid overburdening notation we will
drop $t$ when writing $\zresn(t)$, $z(t)$, $\uresn(t)$, and $u(t)$.
Following \cite[Lemma 3.9]{MiRou06}, where the recovery sequence right below has
been introduced to deduce energy convergence, we set
\begin{equation*}
\tilde{z}_{\eps_n}:= \max\{ 0, z- \|z_{\eps_n} - z\|_{L^\infty(\Omega)}\} \,.
\end{equation*}
Now,  for $q>d$ the convergence $z_{\eps_n} \weakto z$ in $W^{1,q}(\Omega)$, see \eqref{convoz-ptw-eps-weal},
implies $z_{\eps_n} \to z $ in $L^\infty(\Omega)$.
Thus, it can be checked that
\begin{equation}
\label{strong-w1q}
\tilde{z}_{\eps_n} \to z  \text { \emph{strongly} in  } W^{1,q}(\Omega) \,.
\end{equation}
Since $\tilde{z}_{\eps_n} \leq \zresn$, we can choose it as a test function
in \eqref{semistab-general-res}. The term
$ -
\pairing{}{H^1_\mathrm{D}(\Om;\R^d)}{\bigF_{\eps_n}(t)}{\uresn}$
on both sides of the inequality cancels out and we deduce
\begin{equation}
\label{fretta-1}
\begin{aligned}
&
\limsup_{n \to \infty}
\left( \integ{\Om}{(\tfrac12 \CC(\zresn) e(\uresn) \psm
e(\uresn)+G(\zresn,\nabla \zresn))}{x} \right)
\\
& = \limsup_{n \to \infty}
\left( \integ{\Om}{\tfrac12 \CC(\tilde z_n) e(\uresn) \psm
e(\uresn)}{x}
 + \integ{\Om}{G(\tilde{z}_{\eps_n},\nabla \tilde{z}_{\eps_n})}{x} \right)
\leq I_1 +I_2\,,
\end{aligned}
\end{equation}
where
\begin{equation*}
I_1 := \lim_{n \to \infty}
 \integ{\Om}{\tfrac12 \CC(\tilde z_n) e(\uresn) \psm
e(\uresn)}{x} \leq
\integ{\Om}{\tfrac12 \CC( z) e(u) \psm e(u)}{x}\,,
\end{equation*}
combining \eqref{strong-w1q} with \eqref{halleluja}
via the Lebesgue Theorem. It follows from \eqref{strong-w1q},
condition \eqref{G-growth} on the growth of $G$ from above,
and again the Lebesgue Theorem that
\begin{equation}
\label{fretta-3}
I_2 : =  \lim_{n \to \infty}  \int_{\Omega} G(\tilde{z}_{\eps_n}, \nabla \tilde{z}_{\eps_n}) \,
\mathrm{d} x =  \int_{\Omega} G({z}, \nabla z) \, \mathrm{d} x\,.
 \end{equation}
Taking into account  the previously proven \eqref{enhanced-below},
from \eqref{fretta-1}--\eqref{fretta-3}
we ultimately infer
\[
\limsup_{n \to \infty}  \integ{\Om}{G(\zresn,\nabla \zresn))}{x} \leq \integ{\Om}{G(z,\nabla z))}{x}\,,
\]
whence \eqref{vorrei}.
\par
\textbf{Step $6$, case $q>d$, passage to the limit in the mechanical energy inequality on $(s,t)$: }
We now pass to the limit in \eqref{mech-energy-eq-res}
written on an interval $[s,t] \subset [0,T]$, for every $t\in [0,T]$ and almost all $s\in (0,t)$.
Clearly, it is sufficient to discuss the limit passage on the
right-hand side of \eqref{mech-energy-eq-res}, evaluated at $s$.
The first summand tends to zero for almost all $s$, thanks to \eqref{correct?}, which
in particular ensures $\eps_n \duresn(s) \to 0$ in $L^2(\Omega;\R^d)$
for almost all $s\in (0,T)$. The second term passes to the limit  by
\eqref{energy-convergence}, while the third and the fourth ones can be dealt with
by \eqref{bella-1}--\eqref{bella-2} and \eqref{power-cvrg}, respectively.
\par
 \textbf{Step $7$, limit passage in the rescaled heat equation and temporal evolution of $\Theta$: }
We consider the heat equation  \eqref{weak-heat-res} rescaled by the factor $1/\eps$ and
tested by $\eta\in H^1(0,T)$, constant in space, which results in
\begin{equation}
\label{weak-heat-res2}
\begin{aligned}
&\testw(t) \int_\Omega \thres(t) \dd x 
-\integlin{0}{t}{\dot\testw \integ{\Om}{\thres}{x}}{s}
\\
&=\testw(0)\integ{\Omega}{\thres^0}{x}
+\integlin{0}{t}{\testw\integ{\Omega}{\left(\eps\,\DD(\zres,\thres) e(\dures)
{-} \thres\,\BB \right)\psm e(\dures)}{x}}{s}\\
&\phantom{=}+\integlin{0}{t}{\testw\integ{\Om}{\mod{\dzres}}{x}}{s}
+\tfrac{1}{\eps}\integlin{0}{t}{\testw\integ{\partial\Om}{\hse}{\acca^{d-1}(x)}}{s}
+\tfrac{1}{\eps}\integlin{0}{t}{\testw\integ{\Om}{\hve}{x}}{s} \,.
\end{aligned}
\end{equation}
From  the mechanical energy balance \eqref{mech-energy-eq-res}
we deduce by a comparison argument that
\begin{equation*}
\eps\int_0^T\!\int_\Omega\DD(\zres,\thres)e(\dures):e(\dures)\,\mathrm{d}x\,\mathrm{d}s\leq C\,,
\text{ hence also }\;
\eps\int_0^T\!\!\!\eta\int_\Omega\DD(\zres,\thres)e(\dures):e(\dures)\,\mathrm{d}x\,\mathrm{d}s
\leq C\|\eta\|_\infty
\end{equation*}
for every $\eta\in H^1(0,T)$,
taking into account \eqref{weak-limit-bigF}, \eqref{family-eps-below} as well as \eqref{Strongly}.
This allows us to conclude that
 there exists a Radon measure $\mu$ such that  \eqref{171208} holds.
A comparison argument in \eqref{weak-heat-res2} leads to
\begin{equation*}
\left|\eps\int_0^t\testw\int_\Omega\,\theta_\eps\,\BB:e(\dot u_\eps)\,\mathrm{d}x\,\mathrm{d}s\right|
\leq C\|\eta\|_\infty\,,
\end{equation*}
also in view of the bounds \eqref{Heps-heps}, \eqref{apthetaLp-eps} and \eqref{zeps-bv}.
Since $\eta$ is constant in space, integration by parts and an argument along the lines of Step 4 
yield that indeed
$\int_0^t\int_\Omega\eta\,\theta_\eps\,\BB:e(\dot u_\eps)\,\mathrm{d}x\,\mathrm{d}s\to0$.
Moreover, the third convergence in \eqref{Strongly} implies
that $\theta_\eps(t)\to\Theta(t)$ in $L^2(\Omega)$ for a.e.\ $t\in(0,T)$.
Using  \eqref{additional-heat-sources}, we finally pass to the limit in \eqref{weak-heat-res2} and  find that  $\Theta$ satisfies \eqref{levico}.  
\end{proof}

\bigskip
\paragraph{\bf Acknowledgments.}
This work has been supported by  the Italian Ministry of Education, University, and Research
through the PRIN 2010-11 grant for the project \emph{Calculus of Variations}, by  the
European Research Council through the two Advanced Grants
\emph{Quasistatic and Dynamic Evolution Problems in Plasticity and Fracture (290888)}
and \emph{Analysis of Multiscale Systems Driven by Functionals (267802)},
and by GNAMPA (Gruppo Nazionale per l'Analisi Matematica, la
Probabilit\`a e le loro Applicazioni) of  INdAM  (Istituto Nazionale
di Alta Matematica)  through the project \emph{Modelli variazionali per la propagazione di fratture, 
la delaminazione e il danneggiamento}.
G.L.\ acknowledges also the support of the University of W\"urzburg, of the DFG grant SCHL 1706/2-1, 
of SISSA, of the University of Vienna, and of the FWF project P27052.
This paper was submitted on October 24, 2014; the first referee report was received by the authors on December 6, 2017.
\newpage
\bigskip
{\small
\bibliographystyle{alpha}
\bibliography{LRTT_bib}
}
\end{document}